\documentclass [10pt,a4paper,reqno]{amsart}
 \usepackage{amsmath,amscd,amssymb}
 \usepackage{graphicx}
 \makeatletter
\renewcommand*\env@matrix[1][*\c@MaxMatrixCols c]{%
  \hskip -\arraycolsep
  \let\@ifnextchar\new@ifnextchar
  \array{#1}}
\makeatother
\def\dispace{\setlength{\itemsep}{2pt}}

\def\onto{\twoheadrightarrow}

\def\now{\newcount\hours
\newcount\minutes
\hours=\time
\divide\hours by 60
\minutes=\time
\count255=\hours
\multiply\count255 by 60
\advance\minutes by -\count255
\the\hours:\the\minutes
}

\newcommand\smat[4]{
 \left[\begin{smallmatrix}
#1 & #2 \\
#3 & #4\\
\end{smallmatrix}\right]}

\newcommand\bigzero{\mbox{{\huge $0$}}}

\def\mig{\operatorname{mig}}
\def\stab{\operatorname{stab}}
\def\star{\operatorname{st}}
\def\Trns{\operatorname{Trans}}
\def\bTrns{\operatorname{Trans}}
\def\FGH{\operatorname{FGH}}
\def\GS{\operatorname{\Sigma  G}}
\def\QF{\operatorname{QF}}

\def\Dg{\operatorname{Diag}}
\def\dg{\operatorname{diag}}
\def\Tt{\operatorname{T}}
\def\Quot{\operatorname{Quot}}
\def\FG{\operatorname{FG}}

\def\Idm{\operatorname{Idm}}

\def\Bri{\operatorname{Bri}}
\def\Gap{\operatorname{Gap}}
\def\Lo{\operatorname{Lo}}
\def\stab{\operatorname{stab}}

\def\SG{{S}}

\def\as{a.s.\,}
\def\ub{u.b.\,}

\def\NNr{\NN^r}

 \input xy
\xyoption{all}

\newcommand{\longright}[1]{\;{\count255=0 \loop \relbar\mathrel{\mkern-6mu}%
    \advance\count255 by1\ifnum\count255<#1\repeat\rightarrow}\;}
\newcommand{\Right}[2]{\overset{#2}{\longright{#1}}}

\newcommand\isoto{\xrightarrow{
   \,\smash{\raisebox{-0.45ex}{\ensuremath{\scriptstyle\sim}}}\,}}
\newcommand{\Isoto}{\Right{1}{\,\smash{\raisebox{-0.45ex}{\ensuremath{\scriptstyle\sim}}}\,}}

\newcommand{\etype}[1]{\renewcommand{\labelenumi}{(#1{enumi})}}
\def\eroman{\etype{\roman} \dispace}
\def\ealph{\etype{\alph} \dispace}

\def\pSkip{\vskip 1.5mm \noindent}
\newcommand{\ds}[1]{\, {#1} \, }
\newcommand{\dss}[1]{\quad {#1} \quad }

\def\sm{\setminus}
\def\00{ \{ 0 \}}
\def\NN{\N_0}

\def\To{\longrightarrow}

\def\vrp{\varphi}

\def\zt{\zeta}

\def\Ng{\; \neg\;}

\def\E{{\operatorname{E}}}

\def\Ed{{\E_d}}
\def\F{{\operatorname{F}}}
\def\EC{{\operatorname{EC}}}
\def\Idm{{\operatorname{Idm}}}
\def\ext{{\operatorname{ext}}}
\def\X1{X_1}
\def\Y1{Y_1}

\def\hteq{\;  \widehat{=} \; }

\def\tlu{\tilde u}

\def\tlU{\widetilde U}

\def\tlM{\widetilde M}

\def\iS{S^{-1}}

\def\w{w}
\def\iw{\w^{-1}}

\def\tS{\mathcal S}

\def\N{\mathbb N}
\def\R{\mathbb R}
\def\BB{\mathbb B}

\def\hteq{\; \widehat = \; }

\def\htgm{\widehat {\gm}}
\def\htsig{\widehat {\sig}}

\def\dwu{u^\downarrow}

\def\crA{A^\circ}
\def\crB{B^\circ}

\def\olA{\overline A}

\def\eps{\varepsilon}

\def\al{\alpha}
\def\tlal{\widetilde{\al}}
\def\bt{\beta}
\def\gm{\gamma}
\def\Gm{\Gamma}
\def\sig{\sigma}

\newtheorem{thm}{Theorem} [section]
\newtheorem*{thm*}{Theorem}
\newtheorem{cor}[thm]{Corollary}
\newtheorem{lem}[thm]{Lemma}
\newtheorem{lemma}[thm]{Lemma}
\newtheorem{prop}[thm]{Proposition}

\newtheorem{axiom}[thm]{Axiom}

\newtheorem{convention}[thm]{Convention}
\newtheorem*{claim*} {Claim}
\newtheorem*{theorem4.5'} {Theorem 4.5$'$}
\newtheorem{acknowledgment*}[thm] {Acknowledgment}
\newtheorem{example}[thm]{Example}
\newtheorem{examp}[thm]{Example}
\newtheorem*{examp*}{Example}

\newtheorem{examples}[thm]{Examples}
 \newtheorem{assumption}[thm]{Assumption}
 \newtheorem{remark}[thm]{Remark}

 \newtheorem*{remark*}{Remark}
 \newtheorem{defn}[thm]{Definition}
\newtheorem{construction}[thm]{Construction}

\newtheorem{schol}[thm]{Scholium}

\newtheorem{problem}[thm]{Problem}

\newtheorem*{notation*} {Notation}
\newtheorem*{notations*} {Notations}

\theoremstyle{remark}
\newtheorem*{caution*} {Caution}
\newtheorem*{comment*} {Comment}
\newtheorem{comment}[thm]{Comment}

 \renewcommand{\sectionmark}[1]{}

\newcommand{\bfem}[1]{\textbf{#1}}

 \newcommand{\dl}{\delta}
\newcommand{\Dl}{\Delta}
\newcommand{\lm}{\lambda}
\newcommand{\Lm}{\Lambda}
\newcommand{\om}{\omega}
\newcommand{\Om}{\Omega}

 \newcommand{\Eq}{{\operatorname{Eq}}}

 \newcommand{\supp} {\operatorname{supp}}

\begin{document}

\title[Almost subtractive submonoids]
{Almost subtractive submonoids \\[2mm] of additive monoids  \\[2mm]
and related quadratic forms   }
 \author[Z. Izhakian]{Zur Izhakian}
\address{Faculty of Sciences, Tel-Hai University, Upper Galilee,  12210, Israel }
\email{zzur@telhai.ac.il}

%\address{Faculty of Sciences, Tel-Hai University, Upper Galilee,  12210, Israel.}
  %  \email{zzur@telhai.ac.il}
%\email{zzuriz@gmail.com}

\author[M. Knebusch]{Manfred Knebusch}
\address{Department of Mathematics,
NWF-I Mathematik, Universit\"at Regensburg 93040 Regensburg,
Germany} \email{manfred.knebusch@mathematik.uni-regensburg.de}
%\author[L. Rowen]{Louis Rowen}
% \address{Department of Mathematics,
% Bar-Ilan University,  Ramat-Gan 52900, Israel}
% \email{rowen@math.biu.ac.il}

%\thanks{The research of the first and third authors was supported  by the
% Israel Science Foundation (grant No.  448/09).}

%\thanks{The research of the second author was supported in part by
% the Gelbart Institute at
%Bar-Ilan University, the Minerva Foundation at Tel-Aviv
%University, the Department of Mathematics   of Bar-Ilan
%University, and the Emmy Noether Institute at Bar-Ilan
%University.}

%\thanks{The research of the first author was supported  by the
%Oberwolfach Leibniz Fellows Programme (OWLF), Mathematisches
%Forschungsinstitut Oberwolfach, Germany.}

%******************************* AMS classification ***********************
%\subjclass{??? Primary 12K10, 13B25; Secondary 51M20}
%\subjclass[2010]  {Primary: 13A18, 13F30, 16W60, 16Y60; Secondary:
%03G10, 06B23, 12K10,   14T05}

%******************************* AMS classification ***********************

%******************************* date *************************************
\date{\today,  \now}
%******************************* keywords *********************************


%******************************* file name *********************************

%\thanks{\noindent \underline{\hskip 3cm } \\ File name: \jobname}

%******************************* abstract *********************************

\begin{abstract}
We introduce and study the class of almost subtractive submonoids of additive monoids. This property relaxes the strict condition of subtractivity while still allowing for a simulated subtraction operation sufficient for many algebraic applications. We explore the structural properties of these monoids and demonstrate their significance in the theory of modules over semirings, specifically regarding quadratic forms and their companions.

% -----------------------------------
%
%We introduce and study the class of \textit{almost subtractive} (a.s.) submonoids of additive monoids. This property relaxes the strict condition of subtractivity while still allowing for a simulated subtraction operation sufficient for many algebraic applications. We explore the structural properties of these monoids, their generation via additive idempotents, and their crucial role in the theory of modules over semirings. Specifically, we apply this framework to the study of quadratic forms, establishing the uniqueness of companions for forms mapping to a.s. submonoids and developing a geometric theory of ``walks'' and ``trails'' within the set of such forms to analyze their interrelations.
\end{abstract}

\maketitle
\setcounter{tocdepth}{1}

{\small \tableofcontents}

\numberwithin{equation}{section}

\section*{Introduction}

The development of a module theory over semirings \cite{Cos,golan92}, structured parallel to classical module theory over rings, was initiated in \cite{Dec} and has been significantly advanced in recent years \cite{Aml,SA,Gen,Arch}. These works have introduced fundamental algebraic notions such as decomposition, generation, amalgamation, and extension into the semiring context. The central challenge in this theory is the inherent lack of negation in the underlying additive monoid structure of a semiring.

In classical ring theory, the additive group allows for subtraction, which is pivotal for defining kernels, exact sequences, and cancellations. Additive monoids in general lack this operation,
however they may contain a subtractive  submonoid $U$,
that is,  $u \in U$ and $u +z \in U$ imply  $z \in U$.
This property is often too restrictive for broad application in tropical algebra and valuation theory.
%
% Nevertheless, this property is too restricted and is replaced by a weaker property to obtain almost subtractive monoids. These monoids are the main object of study in the current paper, giving rise to modules over compatible semirings and quadratic forms on these modules.
%
 %While some monoids contain \textit{subtractive submonoids}—structures where if $u$ and $u+z$ are in the submonoid, then $z$ must also be—this condition is often too restrictive for broad application in tropical algebra and valuation theory.
 To address this, we introduce and analyze a weaker property, yet robust,  called  \textbf{almost subtractive}, and denoted by  \textbf{\as}. These structures serve as the primary objects of study in this paper and provide the necessary foundation for developing a nuanced theory of quadratic forms over semirings.

The impetus for developing a module theory over semirings stems largely from tropical and supertropical algebra \cite{zur05TropicalAlgebra,nuAlg,IzhakianRowen2007SuperTropical}, and, more generally, from algebraic structures over idempotent semirings, where $a+a=a$ (or variations thereof). Beyond these settings, such module structures arise naturally in diverse mathematical contexts. For instance, additive commutative monoids can be viewed as modules over the semiring $\mathbb{N}_{0}$ of nonnegative integers. Similarly, the sets of positive elements in ordered rings form modules over the semiring of positive scalars.
Modern areas of mathematics make extensive use of monoids and semirings, including discrete mathematics, logic, and automata theory, providing a wide range of natural examples of modules over semirings.

A particularly rich source of examples arises from non-Archimedean valuation theory. Such valuations $v:\mathbb{K}\to~ R$, from a field $\mathbb{K}$ (e.g., a field of Puiseux series) to a semiring $R$, transfer standard linear-algebraic structures, such as polynomials, matrices, and quadratic forms, to the semiring setting by replacing field operations with the corresponding semiring operations \cite{IMS}. However, the absence of additive inverses in $R$ means that the classical definitions of bilinear and quadratic forms do not behave as expected \cite{IKR3}. For example, the relationship between a quadratic form $q$ and its associated bilinear form (its ``companion'') is not necessarily unique or well behaved \cite{QF1,QF2,VR1,QFSym}.

Our study reveals that the {almost subtractive} property is key to recovering this desirable behavior. It ensures that differences, whenever they exist in the target semiring, are unique, thereby restoring rigidity to the theory of quadratic forms over idempotent semirings.
Important non-tropical semirings to which our theory also applies include structures arising in real algebra, such as the positive cone of an ordered field \cite[p.~18]{BCR} and partially ordered commutative rings \cite[p.~32]{Br}. A further application arises in the algebra of groups over a splitting field.%, as described briefly at the end of this overview.

By replacing the rigid requirement of subtractivity with the more flexible notion of almost subtractivity, we provide a framework that is broad enough to encompass key examples from tropical geometry and valuation theory, but also rigid enough to support a rich structure theory for quadratic forms.

\subsection*{Algebraic framework}
An \textbf{additive monoid} $R = (R,  +)$  is a (nonempty) set $R$ equipped  with an   associative
binary operation $+ : R \times R \to R$ and an identity element $0=0_R$.
 $R$ is said to be a \textbf{u.b.} (for ``upper bound''), if $a+b+c=a$ always implies $a+b=a,$. It is said to be \textbf{lacks zero sums}, if it  has the weaker property that  $a+b = 0$ implies $a = b = 0$ \cite{Dec}.

  A submonoid $U = (U,+)$ of an additive monoid $(R,+)$ with a partial ordering  $\leq_R$   is  said to be \textbf{almost subtractive}, if for any two elements  $ u \leq_R \tlu$ in  $U$ there  exists a unique $z \in R$ such that  $u +z = \tlu$.
We call this element $z$ the \textbf{difference} of $u$ and $\tilde{u}$. Unlike strict subtractivity, $z$ need not belong to $U$; it resides in the ambient monoid $R$. This seemingly slight relaxation has profound consequences. Most notably, if $U$ is a.s. in ~$R$, then $U$ itself must be a cancellative monoid.

 A (commutative) \bfem{semiring} is a set~$R = (R,+, \cdot \; )$ equipped with addition and multiplication,
  such that both $(R,+)$ and $(R,\cdot \;)$ are abelian monoids
 % \footnote{A monoid means a semigroup that has a neutral element.
with identity element $0_R$ and $1_R$ respectively, and multiplication
distributes over addition in the usual way.
  % We always assume that $R$ is a commutative semiring with $1.$
Namely, $R$ satisfies all the properties of a commutative
ring except the existence of negation under addition.

 A semiring $R$ is  \textbf{u.b.} (respectively, \textbf{lacks zero sums}) if its underlying additive monoid has this property.
 A semiring $R$ is a  \textbf{bipotent semiring},  if $a+b \in \{ a, b\} $ for all $a,b \in R$.
% A semiring
 %$R$ is a \bfem{semifield}, if every nonzero element of $R$ is invertible;
  % i.e., $R\setminus\{0\}$ is an abelian group.
 The max-plus algebra $(\R \cup \{-\infty\}, \max, + ) $ and the boolean algebra $(\BB, \vee, \wedge)$ are well-known examples of bipotent semirings.

%In classical algebra, semirings appear as the targets of non-archimedean valuations $v: \mathbb K \to R$, applied to the field $\mathbb K$  of \emph{Puiseux series} over an algebraically closed field  $F$ of characteristic $0$. Valuations apply naturally to standard structures in linear algebra (such as polynomials, matrices, and quadratic forms) simply by replacing the classical addition and multiplication with those of the semiring target.  However,  since semirings lack negatives, classical theory is not always accessible, and one has to apply an alternative approach.
%Bilinear and quadratic forms over semirings, defined in the familiar way
%on (semi)modules over
%  a  semiring $R$,   are central example where such discrepancy arises \cite{IKR-LinAlg}.
%

 Recall from \cite[\S1-\S4]{QF1} that a \bfem{module} $V$ over as semiring $R$
(also called a \bfem{semimodule}) is an abelian monoid
$(V,+)$ equipped with a scalar multiplication $R\times V\to V,$
$(a,v)\mapsto av,$ where the same customary axioms of modules over a ring hold: $a_1(bv)=(a_1 b)v,$
$a_1(v+w)=a_1v+a_1w,$ $(a_1+a_2)v=a_1v+a_2v,$ $1_R\cdot v=v,$ and
$0_R\cdot v=0_V=a_1\cdot 0_V$ for all $a_1,a_2,b\in R,$ $v,w\in V.$
%We write~$0$ for both $0_V$ and~$0_R$, and 1 for $1_R.$

Modules over semirings have several
notions  of ``basis'', since dependence and spanning do not coincide \cite{IKR-LinAlg}. This paper uses the
standard categorical notion, i.e., an $R$-module ~$V$ is
\bfem{free}, if there exists a family $(e_i \ds|i\in I)$ in $V$, called a \textbf{basis},
such that every $x\in V$ has a unique presentation
$x=\sum_{i\in I} x_i e_i$ with scalars $x_i\in R$ and only
finitely many nonzero $x_i $.   Any free module with an $n$-element  basis is  isomorphic to $R^{n}$ under the map
$\sum_{i=1}^n x_ie_i\mapsto (x_1, \dots, x_n).$

\subsection*{Paper outline and main results}

%\subsubsection*{Underlying structure}
In \S\ref{sec:1}, we establish the basic theory of \as submonoids, and distinguish  basic properties. We prove that these properties are  preserved under intersection and union of chains (Proposition~ \ref{prop:1.3}), allowing for the existence of maximal \as submonoids (Corollary~ \ref{cor:1.4}). A motivating result for the entire paper (Theorem~\ref{thm:1.2}) links this abstract definition to quadratic forms:
\begin{quote}
    \textit{If $q: V \to R$ is a quadratic form on an $R$-module $V$ such that $q(V) \subseteq U$ and~$U$ is a.s. in $R$, then $q$ is rigid, and its companion bilinear form $b$ is unique.}
\end{quote}
This theorem assures us that over \as monoids, the ambiguity typically associated with semiring quadratic forms vanishes. We also show that amalgamation of certain modules works well in this framework  (Theorem \ref{thm:1.11}).

%\subsubsection*{Constructions involving idempotents}
In \S\ref{sec:2}, we explore methods to construct new \as submonoids from existing ones. A powerful tool here is the utilization of additive idempotents $s \in R$ (where $s+s=s$). We define the construction $[U \neg s]_0$ as the set of elements $x \in R$ such that $x+s \in U$, augmented with~$0$.  Theorem \ref{thm:2.1} proves that if $U$ is a.s., then $[U \neg s]_0$ is also a.s. in $R$. This allows us to ``extend'' the subtractive properties of~$U$ using the idempotent structure of $R$, a technique particularly relevant for supertropical semirings where idempotents (ghost elements) play a structural role.

%\subsubsection*{Gaps, bridges, and structural analysis}

To deeply understand how $U$ locates  inside $R$, we introduce the concepts of {gaps} and {bridges} in~ \S\ref{sec:6} and~\S\ref{sec:a4}.
%\begin{itemize}
    %\item
    A \textbf{gap} is an element $z \in R \setminus U$ that serves as the difference between two elements of $U$ (i.e., $u+z = \tilde{u}$).
    %\item
    A \textbf{bridge} is a sequence or sum of gaps that connects elements within the structure.
% \end{itemize}
We analyze these using equivalence relations with weights, specifically focusing on quotient monoids of the form~ $\mathbb{N}_0^r / E_d$, where $E_d$ is an equivalence relation compatible with a weight function. This analysis provides a class of examples where we can explicitly compute the gap structure and verify the a.s. property. We provide a classification for the equivalence classes of such relations (Theorem \ref{thm:6.8}).

%\subsubsection*{Quadratic forms: walks, trails, and grids}

Sections \S\ref{sec:a4}--\S\ref{sec:a9} of the paper are mainly  devoted  to a geometric and combinatorial analysis of the set of $U$-valued quadratic forms on a free module $V$.
%
%\subsubsection*{Walks and Trails}
We view the set of quadratic forms not just as an algebraic set, but as nodes in a graph. We define a \textbf{walk} as a sequence of forms connected by specific modifications.
%\begin{itemize}
 %   \item
 An \textbf{$S$-walk} in $U$ connects members of $U$ by adding ``gap elements'' from a finite set $S \subset R$.
 %   \item
 A \textbf{blue walk} is a special type of walk where the modifications are restricted to ``blue nodes'', these are points in the structure that exhibit specific stability properties.
%\end{itemize}
This graph-theoretic perspective allows us to define ``distance'' and ``connectivity'' between quadratic forms. In ~\S\ref{sec:b6} and \S\ref{sec:a9}, we define \textbf{trails} as walks without repeating edges, and analyze how one form can be transformed into another via a sequence of elementary gap additions.

%\subsubsection*{Confluences and grids}
In \S\ref{sec:b9}, we study the interaction between different blue walks. We introduce the concept of \textbf{confluence}, where two distinct walks share a sequence of nodes. This leads to the construction of \textbf{grids}, these are rectangular commutative diagrams of walks. The {Translation Lemma} (Lemma ~\ref{lem:b9.4}) is central here, showing that walks can be ``translated'' by adding a minimal gap-sum, preserving their structural properties. This creates a regular, lattice-like structure within the set of quadratic forms.

%\subsubsection*{Ancestors and idempotent supports}
A novel contribution of this work is the theory of \textbf{ancestors} developed in
\S\ref{sec:d10}. Given a blue  walk  in $U \subseteq R$, we can look ``backwards'' in time (algebraically speaking) to find a ``primitive'' walk starting $p_0 \in U$ below $u_0$ from which the current one descends via the addition of elements from $U$. This is metaphorically described as ``drilling into the glacier'' at a node $u_0$ to find an underlying structure.
Furthermore, in \S\ref{sec:e10}--\S\ref{sec:f10}, we analyze walks supported by idempotent gaps. These are walks where the transition between nodes is mediated by an idempotent element. Such walks exhibit high stability and are central to understanding the asymptotic behavior of blue walks in $U$ and $U$-valued quadratic forms forms on the $R$-module  $V$.

\subsection*{Applications}
The general framework of \as submonoids developed in this paper offers powerful theoretical tools across several mathematical disciplines where addition lacks an inverse operation.

%\subsection{Rigidity and Invariant Theory in Supertropical Algebra}
In classical linear algebra, quadratic forms $q(x)$ and bilinear forms $b(x,y)$ are tightly linked through the identity $b(x,y) = q(x+y) - q(x) - q(y)$. In tropical and supertropical semirings, the absence of negation causes this link to break, leading to non-unique companion forms and pathological behavior. Applying Theorem 1.2, restricting target values of quadratic forms to an \as submonoid $U \subseteq R$ guarantees that
%\begin{enumerate}\ealph
 %   \item %\textbf{Unique companions:}
    every $U$-valued quadratic form $q$ admits a unique, balanced companion bil
        inear form $b_q$.
    %\item %\textbf{Invariants and Witt-type theories:}
    Furthermore,
     since the companion $b_q$ is unique and rigid, classical invariants (such as rank, determinant, and discriminant) can be uniquely associated with $q$ over supertropical modules, laying the groundwork for a tropical analogue of Witt Cancellation Theory.
%\end{enumerate}

%\subsection{Valuation Theory and Module Extensions over Semirings}
Non-Archimedean valuations $v: \mathbb{K} \to R$ translate field-theoretic metric structures into semirings.
%\begin{enumerate}\ealph
 %   \item % \textbf{Exact Sequences without Additive Groups:}
    By using \as  submonoids $U$, one recovers cancellative monoid mechanics inside non-cancellative semirings. This enables the definition of kernels and exact sequences for semimodules over valuation semirings without requiring underlying additive groups.
  %  \item % \textbf{Fractional Extensions:}
    The localization theory developed in \S7, i.e., $S^{-1}U \subset S^{-1}R$, allows for seamless extensions of quadratic and bilinear forms from free modules $V$ to localized modules $S^{-1}V$, matching field-of-fractions techniques used in valuation theory and non-Archimedean geometry.
%\end{enumerate}

%\subsection{Discrete Optimization, Grid Walks, and Tropical Geometry}
Quotient monoids of the form $\mathbb{N}_0^r / E_d$ (studied in \S3) and the graph-theoretic analysis of $S$-walks (developed in \S4, \S8--\S10) have direct applications to discrete optimization.
%\begin{enumerate}\ealph
  %  \item %\textbf{Dynamic Programming and State Transitions:}
    Gap elements $z \in \text{Gap}(R,U)$ and bridge sums represent path costs or transition penalties between valid state configurations $u_0, u_1 \in U$.
   % \item %\textbf{Grid Regularity on Tropical Hypersurfaces:}
    The Translation Lemma (Lemma~\ref{lem:b9.4}) and the construction of grids of blue $S$-walks provide lattice-like regularity conditions for optimal paths, modeling stable solution trajectories in tropical linear programming.
% \end{enumerate}

%\subsection{Positivity Certificates in Real Algebraic Geometry}
In real algebraic geometry, the positive cone $C = \{x \in A \mid x \ge 0\}$ of a partially ordered ring $A$ forms a semiring lacking additive inverses. Applying \as submonoids to positive cones yields a cancellation in positive cones. identifying \as submonoids within $C$ allows unique difference calculations $u + z = \tilde{u}$ while keeping $z \in A$, preserving global geometric information without requiring full ring negation. Furthermore,  $U$-valued quadratic forms over positive cones provide well-behaved positivity certificates for polynomials, ensuring uniqueness of underlying bilinear structures when testing for sum-of-squares decompositions over ordered semirings.

\section{Almost subtractive monoids}\label{sec:1}

Let  $R$ be an additive monoid, i.e., an abelian semigroup $(R,+)$ with identity element $0 = 0_R$.
Any submonoid  ~$D$ of $R$ induces a $D$-\textbf{quasiordering}, defined for $x,y  \in R$ by
$$ x \leq_D y \dss{\Leftrightarrow} x + d = y  \text{ for some } d \in D.$$
The monoid  $R$  is called \textbf{upper bound}, if the $R$-quasiordering $\leq_R$ is antisymmetric; in this case it is a partial ordering on $R$. The term ``upper bound'' refers to the fact that $x+y$ serves as a built-in upper bound for  the set $\{x,y\}$.
Assume for simplicity, that $R$ is upper bound, abbreviated as  \textbf{\ub}.
%i.e., $\leq_R$ is a partial order on $R$ (instead of just a quasiordering).
A submonoid $U$ of $R$ is called \textbf{subtractive} in $R$, if, for all $u \in U$ and $z \in R $, the condition  $u +z \in U$  implies that $z \in U$
%for any $u \in U$ and $z \in R$ such that $u +z \in U$
see \cite[p. 154]{golan92}). We  introduce a slightly weaker property of submonoids of $R$.

\begin{defn}\label{defn:1.1} A submonoid $U = (U,+)$ of $R$   is  said to be \textbf{almost subtractive in $R$} (written \textbf{\as}for short), if for any two elements  $ u \leq_R \tlu$ in  $U$ there  exists  a \textbf{unique} $z \in R$ such that  $u +z = \tlu$. That is, if $u, \tlu \in U$, $z,z' \in R$, and $u + z= u +z' = \tlu$, then $z = z'$.
We call such an  element $z$ the \textbf{difference} of $u$ and $\tlu$.
\end{defn}
%Here we mean by $u <_R \tlu$ that $u \leq_R \tlu$, but not $\tlu \leq_R u$. In other words $u \leq_R \tlu$, $u \neq \tlu$.

%Note that, if $u$ is almost   subtractive in $R$, then the monoid $(U,+)$ is cancellative. In particular, $(R,+)$ is cancellative iff $(R,+)$ is \as in $R$.

To measure  the subtractiveness of $U$ in $R$, we define for $u, \tlu \in Y$ the ``gap''
\begin{equation}\label{eq:1.1}
H_R(u,\tlu) := \{ z \in R \sm U \ds| u +z = \tlu \},
\end{equation}
which is either empty or a singleton.
If $U$ is \as in $R$, then the monoid $U$ is cancellative, i.e.,
$$u+z = \tlu = u +z' \dss{\Rightarrow} z = z'$$
for all $z,z',u \in R$.
In particular, $R$ is \as in itself  if and only if $(R,+)$ is cancellative.
Our interest in this condition stems from the following fact.

\begin{thm}\label{thm:1.2} Let $R = (R, \cdot\, , +)$ be a semiring, let  $(U,+)$ is  \as in $(R,+)$, and let   $q: V \to R$ be a quadratic form on an $R$-module $V$ such that  $q(V) \subseteq U$, cf. \cite{QF1} and \S\ref{sec:a5} below. Then,~ $q$ is rigid, and its unique companion $b$ is balanced.\footnote{The intriguing point here is that $V$ may be any $R$-module.}
\end{thm}
%The intriguing point here is, that $V$ may be any $R$-module.
\begin{proof}
Suppose   that $q$ has two distinct companions $b_1$ and $b_2$. Choose $x_1, x_2 \in V$ such that  $b_1(x_1, x_2) \neq b_2(x_1, x_2)$, and let $u = q(x_1)+q( x_2)$, $\tlu = q(x_1+ x_2)$,
$z_1 = b(x_1, x_2)$, $z_2= b_2(x_1, x_2)$.  Then $\tlu = u + z_1 = u + z_2$, and therefore
$z_1 = z_2 $, since $(U, +)$ is \as in $(R,+)$.  Hence,  $q$ has a unique  companion $b$. For any $x \in V $, the equality
$$ 4 q(x) = q (x +x ) = 2 q(x) + b(x,x)$$
holds in $U$, and thus  $2q(x) = b (x,x)$.
\end{proof}

Before detailing the consequences of Theorem~\ref{thm:1.2} in quadratic form theory, we first identify settings in which \as submonoids occur.
\begin{prop}\label{prop:1.3}
  Let  $U = (U,+)$ be  an \as submonoid of a semiring $R = (R,+)$.
  \begin{enumerate}\ealph
    \item Every submonoid of $U$ is also  \as in $R$.
    \item If $(U_\lm \ds| \lm \in \Lm)$ is a chain of \as submonoids of $R$, then $\bigcup_{\lm \in \Lm} U_\lm$ is \as in $R$.
  \end{enumerate}
\end{prop}
\begin{proof}
a): Obvious by Definition \ref{defn:1.1}.
\pSkip
b): Let $x,y \in \bigcup_{\lm \in \Lm} U_\lm$, $z \in R$, $z' \in R$, and $x+z = x +z' =y$. There exists $\lm_0 \in \Lm$ such that $x,y \in U_{\lm_0}$. It follows that $z = z'$, since $U_{\lm_0}$ is \as in $R$. Thus, $\bigcup_{\lm \in \Lm} U_\lm$ is \as in~ $R$.
\end{proof}

\begin{cor}\label{cor:1.4}
  Every \as submonoid $U$ of $R$ is contained in a maximal \as submonoid $\tlU$. Moreover, every submonoid of~ $\tlU$ is \as in $R$.
\end{cor}
\begin{proof}
Clear from Proposition \ref{prop:1.3} by the use of Zorn's Lemma.
\end{proof}

\begin{prop}\label{prop:1.5}
If $U$ is \as in $(R,+)$, and $S$ is a set of minimal elements of $U \sm \00$ under  $\leq_R$, then~ $U \sm S$ is a submonoid of $U$, and hence  $U\sm S$ is  \as in $R$.
\end{prop}
\begin{proof}
  Let $x,y \in U \sm S$, and assume  that $x \neq 0$ and  $y \neq 0$. Then,
  $x + y \in U$. To verify that $x+y \in U \sm S$, suppose that $x + y =s \in S$.  Then, $x \leq s$ and $y \leq s$. Since $s$ is minimal and $x,y \neq 0$, it follows that $x=y =s $,a  contradiction.
\end{proof}

This gives rise to a class of examples of \as submonoids.
\begin{example}\label{exmp:1.6}
  Let $R = (R,+)$ be a  cancellative  monoid, and let $S$ be the set of minimal elements of $R \sm \00$. Then,  $R$ is trivially \as in $R$, and thus  $U= R \sm S$ is \as in~ $R$ by Proposition~\ref{prop:1.5}.
\end{example}
Assume  that $U$ is an \as submonoid of $R$ that is \textbf{not} maximal. Then,  there exists  $t \in R \sm U$ such that
$$ U' := U + \N t = U + \N_0 t$$
is  \as in $R$. We aim to understand  the structure of $U'$ in terms of its components $U$ and $\N_0 t$ under $\leq_R$.
Note that $\N_0 t$ is \as in $R$, and therefore is a cancellative monoid. Thus, for $m \in \N_0$,  the  elements $m t$ pairwise disjoint. In other words, there is  a monoid  isomorphism
\begin{equation}\label{eq:1.2}
  \N_0 \Isoto \N_0 t, \qquad m \mapsto mt.
\end{equation}
For later use, we  explore a more general situation.
\begin{assumption}\label{assumption:1.6} $U$ is a  submonoid of $R$,  $t\in R \sm U$, and $U + \N_0 t$ is cancellative.
\end{assumption} Under  this assumption,
$D = (\N_0 t)\cap U$ is a submonoid of $\N_0 t$, and thus either
$D = \00$ or
\begin{equation}\label{eq:1.3}
  D =  (\N_0 t)\cap U = \N_0 d t = d \N_0 t \quad  \text {for some } d\in \N,
\end{equation} which we call the \textbf{period} of $t$ (with respect to U).  If $D = \00$, we set $u_0 = 0 $. In case of \eqref{eq:1.3},  set $u_0 = dt $, so that  $u_0 > 0 $.  This yields the following diagram of cancellative submonoids of $R$
\begin{equation}\label{eq:1.4}
 \begin{array}{c}
   \xymatrixrowsep{3mm}
\xymatrixcolsep{6mm}
    \xymatrix@R=0.3em@C=0.5em{
    & & U + \N_0 t = U' \ar@{-}[dd] \\
    \N_0 t      \ar@{-}[rru] \ar@{-}[dd] & & \\
    & & U  \\
    D \ar@{-}[rru]   &  & \\
   }
   \end{array}
\end{equation}

We study the exchange equivalence relation $\EC = \EC_{U, \N_0 t}$ associated with the diagram \eqref{eq:1.4}, as defined in~ \cite{Aml}. This relation records what happens to a pair $(u, k t)$ when we either extract copies of $u_0 = dt$ from $k t$ and add them to $u$, or extract copies of $u_0$ from $u$ and add them to $kt$. In other words, the relation $\EC$ is generated by the ``elementary equivalences'':
\begin{equation}\label{eq:1.5}
  \begin{array}{llll}
    (u, kt)  \sim_\EC  (u + u_0, (k-d))t, & & \text{if } k \geq d,   \\[2mm]
    (u, kt)  \sim_\EC  (u_1, (k+d))t, & & \text{if } k < d \text{ and } u = u_1 + u_0.   \\
  \end{array}
\end{equation}

\begin{lem}\label{lem:1.8}
Under Assumption \ref{assumption:1.6}, for every pair $(u', k't)$ with $u' \in U$ and  $ k' \in \N$ there exists a unique exchange equivalent pair $(u, kt)$ where  $u \in U$ and    $k < d$, if $d > 0$,
while $u \in U$ and $ k \in \N_0$, if $d =0$.
  %\begin{equation}\label{eq:1.4}
%    z= u + k t
%  \end{equation}
%  with unique $k < d$, $u \in U$ if $d > 0$, and unique $u,k \in \N_0$ if $d =0$.
\end{lem}
\begin{proof} Since every $k' \in \N_0$ decomposes as $k' = k + \ell d$, where  $0 \leq k < d$ and  $\ell \in \N_0$, it suffices to prove the uniqueness assertion when
 $d> 0$. Suppose that $ u + kt = v + \ell t$, where  $u,v \in U$ and  $0 \leq k \leq \ell < d $. Since $U + \N_0 t$ is cancellative and $t < u$, this implies that $u = v + (\ell-k)t$, and thus $\ell - k = 0$. Hence, $\ell = k $  and $u=v$.
%Since $t < u $ for revery $u \in U \sm \00$, this forces $\ell = k$, and so $u = v$. The case of $d =0$ is similar.
\end{proof}
The  presentation
$z = u + kt$ of an element $z$ in $U + \N_0 t$ with $u$ and  $k $ as in the lemma, is called  the \textbf{reduced presentation} of $z$.

\begin{schol}\label{schol:1.9}
Assume that $d > 0$, $k_1, k_2 \in [0,d-1]$, and  $u_1, u_2 \in U$. Then,
\begin{enumerate}\ealph
  \item
  $ (u_1 + k_1 t) + (u_2 + k_2 t) = \left\{
    \begin{array}{lll}
    (u_1 + u_2) + (k_1 + k_2)t,  & \text{if } k_1 + k_2 < d,&  \\[2mm]
    (u_1 + u_2 + u_0) + (k_1 + k_2 -d)t,  & \text{if } k_1 + k_2 \geq d.&  \\
  \end{array}
  \right.$
  \item
  $ u_1 + k_1 t \leq u_2 + k_2 t \dss \Leftrightarrow  \left\{
    \begin{array}{lll}
    u_1 \leq u_2, k_1 \leq k_2, & \text{if }  k_2 < d,&  \\[2mm]
    u_1 + u_0 \leq  u_2 ,  k_1 \leq k_2 - d,  & \text{if }  k_2 \geq d.&  \\
  \end{array}
  \right.$
\end{enumerate}
  If $d =0 $, then for all $u_1, u_2 \in U$ and $k_1, k_2 \in \N_0$:
\begin{enumerate}\ealph \setcounter{enumi}{2}

  \item $(u_1 + k_1 t) + (u_2 + k_2 t) = (u_1 + u_2) + (k_1  + k_2) t$.

  \item $(u_1 + k_1 t) \leq  (u_2 + k_2 t) \ds{\Leftrightarrow} u_1 \leq u_2$,  $k_1 \leq k_2 . $
\end{enumerate}
\end{schol}

\begin{prop}\label{prop:1.10}
  Under Assumption \ref{assumption:1.6},
  the pair $(U, \N_0 t)$ has amalgamation in $R$.
\end{prop}

\begin{proof}
  Follows immediately  from the explicit description of the elements of $U ' = U + \N_0 t$ in Lemma \ref{lem:1.8} and Scholium \ref{schol:1.9}.
\end{proof}

The next theorem follows from the considerations following  Proposition \ref{prop:1.5}.

\begin{thm}\label{thm:1.11}
Let $R= (R,+)$ be a \ub  monoid, and let $U$ be a non-maximal  \as submonoid of ~$R$. Then, there exists  $t \in R \sm U$ such that  $(U, \N_0 t)$ has amalgamation in~ $R$,  and $U + \N_0 t$ is \as in ~ $R$.

\end{thm}

\noindent
\emph{Comment.} In view of Theorem~\ref{thm:1.2}, it is of interest in quadratic form theory to find larger \as submonoids of $R$ starting from a given one. However, Theorem~\ref{thm:1.11} leaves much to be desired, as we have not yet obtained an explicit construction of the element $t$.
\pSkip

Nevertheless, we search for explicit examples of \as submonoids $U$ of $R$ that are related to this problem. We begin with an obvious observation.
\begin{lem}\label{lem:1.12}
Let  $R$ be an additive monoid. Then, a subset $U$ of $R$ is a submonoid if and  only if, for all  $x,y \in R$, the condition    $x+y \notin U$ implies that  $x \notin U$ or $y \notin U$ (in particular~ $0\in U$).
\end{lem}
The ``easy case'' is when $R$ itself is cancellative.
\begin{prop}\label{prop:1.13} If $R$ is cancellative, then every submonoid $U$ of $R$ is \as in $R$.
\end{prop}
\begin{proof}
  In this case,  $R$  is \as in $R$, and thus  $U$ is \as in $R$.
\end{proof} \noindent
Note that Example \ref{exmp:1.6} falls in this case.
\pSkip
In our search for examples of \as submonoids of a different type, we first need to fix a non-cancellative \ub monoid ~$R$.
 To this end, we consider the free additive monoid
$$\N_0 \times \N_0 = \N_0 t_1 + \N_0 t_2$$
 with generators
$ t_1 = (1,0)$, $t_2 = (0,1)$.
Let
$$ \w: \N_0 \times \N_0 \longrightarrow  \N_0$$
be a surjective additive map satisfying  $\w(t_1) = \w(t_2) =1$ and  $\w((0,0)) = 0$. Then,
\begin{equation}\label{eq:1.6a}
  \w((n_1, n_2)) = \w(n_1t_1 +n_2 t_2) = n_1 + n_2
\end{equation}
for every $n_1, n_2 \in \NN$.
We call  $\w$ the (standard) \textbf{weight function} on $\NN^2 = \NN \times \NN$. For  $d \geq 1$,  we define an additive equivalence relation $\E_d$ on $\NN^2$ by setting the equivalence class of $z \in \NN^2$ to be
\begin{equation}\label{eq:1.7}
  [z]_{\E_d} = \left\{
                 \begin{array}{ll}
                   \{ z \} & \hbox{if } \w(z) < d, \\[2mm]
                   \iw(z)& \hbox{if } \w(z) \geq d. \\
                 \end{array}
               \right.
\end{equation}
A straightforward verification shows that $\E_d$ is indeed an additive equivalence relation on ~$\NN^2$.

\begin{prop}\label{prop:1.14}
$\E_d$ is the finest additive equivalence relation on $\NN^2$  for which
$\xi = \iw(d)$ is an equivalence class.
\end{prop}
\begin{proof}
  It suffices to verify that, if $E$ is an  additive equivalence  relation on $\NN^2$ for which  $\iw(d)$   an equivalence class  of $E$, then  $\iw(d+1)$ is also an equivalence class of $E$.
  Given $x,y \in \iw(d+1)$,
   suppose first that  $x= u +t_i$ and  $y = v+t_i$ for some
   $u,v \in \NN^2$ and  $i = 1,2$. Then, $\w(u)= \w(v) = d$, and hence $u \sim_E v$.
   By the additivity of $\E$, this implies that  $u + t_i \sim_E v + t_i$. In the  remaining  case, $x = (d+1)t_1$ and  $y = (d+1)t_2$. Since $d \geq 1$, we have
  $$ x \sim_\E d t_1 + t_2, \qquad y \sim_\E d t_1 + t_2,$$
 and hence   $x \sim_\E y$.
\end{proof}
% We  take the finest additive equivalence relation $\E_w$ on  $\N_0 \times \N_0$ which for a given $d \geq 2$ identifies the set $\{ z\in \N_0 \times \N_0  \ds | w(z) = d \}$ as one class $\xi$.
The elements of the monoid
\begin{equation}\label{eq:1.8}
  R_d := (\N_0 \times \N_0)/ \E_d
\end{equation}
are the equivalence  classes $[z] = [z]_\Ed$ of the relation $\E_d$, with  addition defined by  $[z_1] + [z_2] = [z_1 + z_2]$. If $x = y+z$ in $\NN^2$ with $z \neq 0$, then $\w(x) > \w(y)$. Thus the monoid $R_d$ is clearly~ \ub.
We identify each class $[z]_\Ed$ with $\w(z)< d$ with  its unique representative $z \in R$. Note that $R$ is violently not cancellative, if $d \geq 2$.

\begin{prop}\label{prop:1.14b}
The \ub monoid $R_d$, with   $d \geq 2$, contains  no nonzero \as submonoid.
\end{prop}
\begin{proof}
Assume  that  $U$ is a nonzero submonoid $\neq \00$ of $R_d$. Choose an element $[z]$ of smallest weight~ $\geq d$ in $U$, and write
$ z= r \xi + n_1 t_i + n_2 t_2$,
where  $w(\xi) =d$, $r \geq 1$, and $0 \leq n_1 + n_2 < d$.
Then,  $(r+1) \xi$ can be obtained from $z$  by  two different ways within $R_d$:
$$ z + (d - n_1 - n_2) t_1 = z + (d - n_1 - n_2) t_2 = (r + 1) \xi.$$
Thus, the \as property fails for $U$.
\end{proof}

To positively  answer  the existence of \as submonoids, we restrict
 $R_d$ to  its  submonoid
\begin{equation}\label{eq:1.9}
  R_d^+ := \big\{ [n_1t_1 + n_2t_2] \ds | n_1 > n_2 \big \}.
\end{equation}

\begin{prop}\label{prop:1.15} The \as submonoids of $R^+_d$ are precisely the submonoids of $\N_0 t_1$. If $U$ is such a submonoid and $U \neq \00$, then $U  = Hx$, where  $x$ the element of smallest weight in $U \sm \00$, and $H$ a submonoid of $\N_0$; both uniquely determined by $U$.
\end{prop}
\begin{proof}
Suppose  that $U$ is a submonoid of $R^+_d$ with $\NN t_1 \subsetneq U$.
Let $ x \in U \sm \NN t_1$ be an  element of smallest weight. Write
$$ x = n_1 t_1 + n_2 t_2 + r \xi ,$$
where $w(\xi) = d$, $ r \geq 0$, $n_1 + n_2 < d$, and  $n _2 < n_1$.
We have
 $n_2 >  0$,
  since otherwise $x$ would belong to ~$\NN t_1$.
  %We can reach $(r+1) \xi$ from $x$ only in one way, by adding $(d - n_1) t_1$. If $n_2 > 0$, $n_1 = n_2 + m$, then also $m > 0$, and
  We can obtain  $(r+1) \xi$  from $x$ in two different ways:
\begin{align*}
     (r+1) \xi & = x + (n_1 - n_2) t_1 + (d- n_1) t_2 \\
 & = x + (d- n_1 - n_2) t_1 .
      \end{align*}
      Note that  $n_1 -n_2 > d - n_1$.
Thus, $U$ is not \as in~ $R_d^+$.
The uniqueness statement for \as submonoids of ~$\N_0 t $ is clear by using  the weight function.
\end{proof}
\begin{cor}\label{cro:1.16}
$\N_0 t_1$ is the unique maximal \as submonoid of $R^+_d.$ It is subtractive in $R^+_d$.
\end{cor}

\begin{proof}
  Evident by Proposition \ref{prop:1.15}.
\end{proof}

More generally, we may work in quotients of the free monoid $\N_0^r$ for $r \geq 2$,
\begin{equation}\label{eq:1.10}
\N_0 ^r = \bigoplus_{i=1}^r\N_0 t_i, \qquad t_i = (0, \dots, 0, \underset{i}{1}, 0, \dots, 0).
\end{equation}
We take  the weight function $w$ with $w(t_i) = 1$ for $1 \leq i \leq r$ and $w((0,\dots,0)) = 0$. As seen  above (Proposition ~\ref{prop:1.14}), for any $d \geq 1$, the finest equivalence relation  $\E_d$  for  which $\iw(d) = \xi$ is an equivalence class  is given by ~\eqref{eq:1.7}.
 We can then identify  patterns of \as submonoids of
\begin{equation}\label{eq:1.11}
  R = \N_0^r / \E_d
\end{equation}
by the same method as above.% We give an example.
\begin{examp}
  Let $r = 3 $, and let $d \geq 2$. Let $U$ be the submonoid  of $R = \N_0^3 / \E_d$ consisting of the  classes
  $  [n_1t_1 + n_2 t_2 + n_3 t_3] \in E_d$ such that  $n_1 > n_2 + n_3 $ and $n_1 + n_2 + n_3 > d$. Then, $R$  has the following submonoids containing $t_1$:
  $$
  \begin{array}{lllll}
    U_{12}  & =  \{ [n_1 t_1 + n_2 t_2] \ds | n_1 > n_2  \}, & &
    U_{13}  & =  \{ [n_1 t_1 + n_3 t_3] \ds | n_1 > n_3  \} .
  \end{array}
  $$
  Clearly, $U_{12} \cap U_{13} = \N_0 \xi$, and we have the following amalgamation diagram in $R$
$$
 \begin{array}{c}
   \xymatrixrowsep{3mm}
\xymatrixcolsep{6mm}
  \xymatrix@R=0.5em@C=0.5em{
    &  U  \ar@{-}[rd] \ar@{-}[ld] \\
        U_{12} \ar@{-}[rd] & & U_{13} \ar@{-}[ld]  \\
    & \N_0 \xi    &
   }
   \end{array}
$$
By the same reasoning as above it is easily seen that $\N_0 t_1$ is the unique maximal \as submonoid of each of  three monoids $U$, $U_{12},$ and $U_{13}$, and that $\N_0 t_1$ is subtractive in each of them. Every \as submonoid of $U$ has the form~  $H t_1$, where  $H$ a submonoid of $\N_0$.
\end{examp}

\section{The almost subtractive submonoid $[U \Ng s]_0$ with
 an additive idempotent  $s$}\label{sec:2}

Given a \ub monoid $R$, there is another way to produce new \as submonoids of $R$ from existing  ones. Let
\begin{equation}\label{eq:2.1}
  \Idm(R) := \{ s \in R \ds | s +s =s\}
\end{equation}
denote the set of additive idempotents of $R$.
Clearly, this is a submonoid of $R$.  Given a submonoid $U$ of $R$ and  $s \in \Idm(R)$,  we define
\begin{equation}\label{eq:2.2}
  [U\Ng s] = [U \; \neg_R \; s] := \{ x \in R \ds | x+s \in U\}.
  \end{equation}
It is immediate that $[U \Ng s]$ is closed under addition, and so
\begin{equation}\label{eq:2.3}
  [U \Ng s]_0 := [U \Ng s] \cup \00
  \end{equation}
is a submonoid of $R$. Indeed, if $x_1 + s = u_1$ and $x_2 + s = u_2$, then
$ (x_1 + x_2) + s = (x_1 + s) + (x_2 + s ) = u_1 + u_2,$
since $s +s =s $.

\begin{thm}\label{thm:2.1}
Let  $U$ be an  \as submonoid of~ $R$, and let $s \in \Idm(R)$. Then, $[U \Ng s]_0$ is  \as in~ $R$.
\end{thm}
\begin{proof}
  Let $x,y \in [U \Ng  s]_0$, $z_1, z_2 \in R$, and
  \begin{equation*}
    x+z_1 = x+z_2 = y. \tag{$*$}
  \end{equation*}
 We need to verify that $z_1 =z_2$, which is obvious for  $x = 0$. If $x  \neq 0$, then  $y \neq 0 $ as well, and hence  $x,y \in  [U \Ng  s]$. Thus  $x + s \in U$ and $y +s \in U$.  Adding $s$ to  $(*)$ gives
  $ z_1 = z_2,$ since $U$ is \as\ in $R$.
\end{proof}
If $s =0$, then
$[U \Ng  s]= U$. Thus,  Theorem \ref{thm:2.1} is meaningful only when
$R$ has  nonzero  idempotents.
\begin{remark}\label{rem:2.2} If $s$ is a nonzero idempotent, then $s \notin U$, since $s +s = s + 0 =s$, while $U$ is cancellative.
\end{remark}
\begin{thm}\label{thm:2.3}
Let  $U$ be an  \as~submonoid of $R$.
\begin{enumerate}\ealph
  \item For every $s \in \Idm(R)$, the semigroup $[U \Ng s]$ is a union of cosets $x + U $ of $U$.

  \item The union of all these cosets
  \begin{equation}\label{eq:2.4}
    \tlU := \bigcup \big\{ [U \Ng  s] \ds | s \in \Idm(R)\big\}
  \end{equation}
  is an \as submonoid of  $R$ containing $U$.

\end{enumerate}
\end{thm}
\begin{proof} (a):  If $s +x \in U$, then $s +x + U \subseteq U$, and hence  $x + U \subseteq [U \Ng  s]$.
\pSkip (b):
Given $s,t \in \Idm(R)$,  $x \in [U \Ng  s]$, and $y \in[U \Ng s]$,  we have $s +x = u \in U$ and $t +y = v \in U$. Adding these equalities gives $s+t+x + y = u +v \in U$.
Thus,
\begin{equation}\label{eq:2.5}
  [U\Ng  s] + [U \Ng  t] \subseteq [U \Ng  s+t].
\end{equation}
In particular, $[U \Ng  s] + U = [U \Ng  s] $, since $[U \Ng  0] = U$, and so
\begin{equation}\label{eq:2.6}
  U \subset [U \Ng  s]_0
\end{equation}
for all $s \in \Idm(R)$. The sets $[U \Ng  s]_0$ and  $s \in \Idm(R)$ form  an upward  directed system of \as submonoids of $R$. Thus, their union $\tlU$ is again an \as  submonoid of $R$.
\end{proof}

\begin{cor}\label{cor:2.4}
Let $U$ be a maximal  \as submonoid of $R$.
  Then,
  %\begin{equation}\label{eq:1.12}
   $$ s +x \in U \dss{\Rightarrow} x \in U$$
  %\end{equation}
  for every $s \in \Idm(R)$ and $x \in R$.
\end{cor}

\begin{proof}
  In consequence of Theorem \ref{thm:2.3},  we have $\tlU = U$.
\end{proof}

\section{Equivalence relations on additive monoids with weights}\label{sec:6}
To construct \as submonoids of a monoid $R$, we introduce equivalence relations defined by weight functions
\begin{defn}\label{def:6.1} A \textbf{weight function} on an additive monoid $R$ is  an additive map
$\w:R \to \N_{0}$ satisfying  $\iw(0) = \00$.
We then say,  briefly, that $R$ is a \textbf{monoid with weights}.
\end{defn}

As is customary for graded monoids,  we also write $\iw(d) = R_d$, and so
$R = \bigoplus_{d \geq 0} R_d$.

\begin{remark}\label{rem:6.1} If $R$ is a monoid with weights, then $R$ is \ub. Indeed, suppose that $x+a = y$ and  $y+b =x$. Then,
$\w(x)+ \w(a) + \w(b) = \w(x),$
whence $\w(a)+\w(b) = 0$. Since $\w(a), \w(b) \in \NN$,  it follows that $a=b= 0$, and hence  $x=y$.
\end{remark}

The prototype of an additive monoid with weights is the free additive  monoid
\begin{equation}\label{eq:6.1}
  R = \N_0^r = \bigoplus_{i=1}^r \NN t_i,
\end{equation}
where
$t_i = (0, \dots, 0, \overset{i}{1}, 0 , \dots, 0)$ is the $i$-th canonical basis element. This monoid  already appeared in \S\ref{sec:1}.
The monoid ~$\NN^r$ carries a unique weight function $\w:\NN^r \onto \NN$, with~
$\w(t_i) = 1$ for $i =1, \dots, r$, and hence
\begin{equation}\label{eq:6.2}
\w(n_1, \dots, n_r)= \w\bigg(
\sum_{i=1}^{r} n_i t_i \bigg) =
n_1 + \cdots + n_r,
\end{equation}
which we henceforth call the  \textbf{standard weight on} $\NN^r$.

\begin{defn}\label{def:6.3}
An additive equivalence relation $\E$ on a monoid $R$ with weight function
$\w:R \to \NN$ is called \textbf{weight-compatible}, if the elements in every equivalence class
$A = [x]_\E$ have the same weight; that is  $w|A$ is constant. Then, $\w$ induces a weight function
\begin{equation}\label{eq:6.3}
(w/\E)([x]_\E) := \w(x)
\end{equation}
 on the quotient $R/\E$.
\end{defn}

Note that a similar situation appeared in \S\ref{sec:1}, cf. \eqref{eq:1.8}.
Our main concern in this section is to obtain a tractable class of \as submonoids of additive monoids by factoring out a suitable weight-compatible equivalence relation~ $E$ on such a monoid $R$. We carry this out in the case where $R$ is finitely generated by classifying all weight-compatible additive equivalence relations on $R$.

We may choose a suitable surjective additive map
$\pi: \NN^r \onto R$ such that $\w \circ \pi : \NN^r \onto \NN$ is the standard weight on $\NN^r$. Namely, if $R = \sum_{i=1}^{r} \NN x_i$, we take the additive map $\pi: \NN^r \to R$ defined by
$\pi(t_i) = x_i$. Thus, we transfer  the classification problem from $R$ to the free monoid ~ $\NN^r$.

We recall a basic lemma on equivalence relations and omit the straightforward proof.
Here,  $R$ may  be any additive monoid with weights.

\begin{lemma}\label{lem:6.4}
  Assume that $(A_\lm \ds | \lm \in \Lm)$ is a family of subsets of $R$ such that  $\bigcup_{\lm \in \Lm} A_\lm = R$ and that $\w:R \to \NN$ is  constant on each $A_\lm$.
  \begin{enumerate} \ealph
    \item There exists  a unique finest equivalence relation $\F$ on $R$ such that every $A_\lm$ is contained in an equivalence class of  $\F$. The equivalence classes of $\F$ are the maximal unions $B_i = \bigcup _{\lm \in \Lm_i}A_\lm$, with $\Lm_i \subset \Lm$, such that $B_i$ cannot be written as a union  $B_i' \cup B_i''$ of two disjoint  sets $B_i'$ and  $B_i''$
        of this form. Thus, ($B_i$ is a ``connected component'' of the family $\{ A_\lm \ds| \lm \in \Lm\}$.) We denote this finest equivalence relation  by $\F = \Eq(A_\lm \ds | \lm \in \Lm)$.
    \item In this notation,
    \begin{equation}\label{eq:6.4}
      G = \Eq( A_\lm + x \ds | \lm \in \Lm, x \in R)
    \end{equation}
  is the finest \textbf{additive} equivalence relation on $R$ for which every set $A_\lm$ is contained in an equivalence class of $G$.
  \end{enumerate}
\end{lemma}

%We add a facet to our study of the set $[B:A]$ in \S\ref{sec:3}, as follows.

Let $\E$ be an additive equivalence relation on $\NNr = \bigoplus_{i =1} ^r \NN t_i$, where $r \geq 2$, that is compatible with the standard weight function. There exists  a smallest $d \in \N$ such that $\E$ identifies distinct elements of $\iw(d)$; that is,  \textbf{not} every $\E$-equivalence class in $\iw(d)$ is a singleton. We assume that $d > 1$. Let $\Pi = (A_\lm \ds | \lm \in \Lm)$ denote the partition of $\iw(d)$ into its $\E$-equivalence classes.
We will determine the $\E$-equivalence classes in $\iw(d+1)$. These are the connected components of the family of sets $(A_\lm + t_i \ds | \lm \in \Lm, i = 1, \dots, r )$.

%We first establish a more intrinsic defintion of the support of an element in $\NNr$ than used in \S\ref{sec:4}.

\begin{defn}\label{def:6.4} Given a nonzero element $x = \sum_{i=1}^r n_i t_i$ of $\NNr$,  we call $n_i$ the \textbf{coordinate} of~ $x $ at the generator~ $t_i$, and define the \textbf{support} $\supp(x)$ to be the set of generators $t_i$ for which ~ $x$ has a positive coordinate.
 \end{defn}

We exhibit the singleton classes in $\iw(d+1)$.

\begin{defn}\label{def:6.6}
We call an element $x \in \iw(d+1 )$ \textbf{extremal} with respect to $\Pi$,
% (or, for~ $\Pi$),
 if $x = (d+1) t_i$ for some
$i \in \{ 1, \dots, r\} $, and the unique class $A_\lm \subset \iw(d)$ containing $d t_i$ is a singleton $A_\lm = \{ d t_i\}$. We then  call the class $[x]_\E$ extremal.
\end{defn}

In this case there is no $A_\mu \neq A_\lm$ such that  $x \in A_\mu + t_i$, since
 every $z \in A_\mu$ has
 some generator $t_j$, with $j \neq i$, in $\supp(z)$, which then also belongs to
 $\supp(z +t_i)$.

 We write more succinctly
 \begin{equation}\label{eq:6.5}
   \E = \E(\Pi,d) = \E(\Pi,d,r)
 \end{equation}
 and denote  by $\ext(\Pi, d+1)$ the set of extremal elements of $\Pi$ in $\iw(d+1)$. The following remark follows by~ \eqref{eq:6.4} and Definition~ \ref{def:6.6}.

 \begin{remark}\label{rem:6.7}
 If $\Pi'$ is a partition of $\iw(d)$ that  is coarser  than $\Pi$, i.e., every set in $\Pi'$ is a union of sets in~$\Pi$, then $\E(\Pi', d)$ is coarser than $\E(\Pi, d)$. However,
 %\begin{equation}\label{eq:6.6}
  $ \ext(\Pi, d+1) \supset \ext(\Pi',d+1)$.
 %\end{equation}
 \end{remark}

A partition $\Pi$ of a set is said to be \textbf{nontrivial}, if it does not consist entirely of singletons.

 \begin{thm}\label{thm:6.8}
   If $\Pi$ is a nontrivial partition of $\iw(d)$, then the equivalence classes of $\E(\Pi,d)$ in $\iw(d+1)$ consist of the extremal classes with respect to $\Pi$ (which are  singletons)   and one additional class
   $B_{d +1} = B_{d+1} (\Pi)$, which we call  the ``big cell'' and which containe all other elements of $\iw(d+1)$.

 \end{thm}

 \begin{proof}
   a) We first verify the statement  for a ``binary partition'' of $\iw(d)$, denoted by $\Pi_{x,y}$,  which consists of one set $\{ x,y \}$, where  $x \neq y$, and singleton sets.  The resulting equivalence relation $\E(\Pi_{x,y}, d)$ is then the finest equivalence relation  $\Eq(x,y)$ identifying  $x$ and $y$. The extremals with respect to $\Pi_{x,y}$
   are the classes $\{ (d+1)t_i \}$, for $i = 1, \dots, r$ (Definition \ref{def:6.6}).
   Every non-extremal element $u$ of $\iw(d+1)$ has a support containing at least two generators $t_i$ and  $t_j$. Hence,  $u = z + t_i = w + t_j$ for  $z \neq w$, and thus
   $$ u \in (\{z\} + t_i) \cap (\{w\} + t_j).$$
   This proves the claims for all binary partitions $\Pi_{x,y}$ and allows us to handle the big cell
   $$ B_{x,y} := \iw (d+1) \sm \ext(\Pi_{x,y}, d+1).$$
   Observe that $B_{x,y}$ contains every element $u$ of weight $d+1$ with $|\supp(u)| > 1$.

   \pSkip b) Assume that $\Pi$ is any (nontrivial) partition of $\iw(d)$. Fix some  $u_0 \in \iw(d+1) $ with $|\supp(u_0)| > 1$, and let $u \in \iw(d+1)$ with $u \in B_{x,y}$. We claim that there exist $x \neq y$ in $\iw(d)$ such that both   $ u$ and $u_0$ belong to $ B_{x,y}$. This is evident, if $|\supp(u)| > 1$. Otherwise,
   $u = (d+1) t_i$ for some $i \in \{ 1, \dots, r\} $, and there is a class  in $\Pi$ that contains $d t_i$  and at least one other element~ $y$. Hence, $u \in B_{dt_i,y}$, and thus there exists a pair $(x,y)$ in some class of $\Pi$ (typically many pairs) such that  both $u$ and $u_0$ in~ $B_{x,y}$. This implies that $\Pi$ is coarser than $\Pi_{x,y}$. By Remark~ \ref{rem:6.7}, we conclude that $u$ and $u_0$ belong to the same class of $\E(\Pi,d)$,
   as desired.
 \end{proof}

 Assume that $\E$ is the additive equivalence relation $\E(\Pi,d)$ appearing in Theorem \ref{thm:6.8}. We will obtain the partition of $\iw(d+2)$ into $\E$-classes by a trick, introducing a new equivalence relation $\F$ on $\NNr$.
 Given  $x,y \in \NNr$,  we stipulate that
 \begin{equation}\label{eq:6.7}
   \begin{array}{lll}
   x \sim_\F y & \text{if} & \text{either $\w(x) < d+1$ and $x =y$,} \\
   & & \text{or $\w(x) = \w(y)$ and $x \sim_\E y$.}
   \end{array}
   \end{equation}
   It is immediate that $\F$ is an additive equivalence relation on $\NNr$. In the above notation,
   \begin{equation}\label{eq:6.8}
     \F = \E(\Psi, d+1),
   \end{equation}
   where $\Psi$ is the partition of $\iw(d+1)$ into the extremal classes of $\E$ and the big cell $B_{d+1} (\Pi)$.
 It follows from Theorem \ref{thm:6.8} that the partition of $\iw(d+2)$ into $\F$-classes is the same as the partition of $\iw(d+2)$ into $\E$-classes.
 Iterating this argument, we obtain the following result.
 \begin{thm}\label{thm:6.9}
  Let $\E$ be an  additive equivalence relation on $\NNr$, where $r \geq 2$, that  is compatible with the standard weight function $\w$ on $\NNr$. Assume  that $\E$ is nontrivial and that all classes of weight $1$ are singletons. Then, there is a minimal $d \geq 2$ such that some $\E$-classes of weight $d$ are not singletons.

 Let $\Pi$ denote the partition of $\iw(d)$ into $\E$-classes, and let $T(\Pi) \subseteq \{ t_i \ds| i =1, \dots, r\}$ denote the set of generators $t_i$ such that the  singletons $\{ d t_i\}$ are members of $\Pi$. For any $e > d$ the partition of $\iw(e)$ into $\E$-classes consists of the singletons $\{ e t_i \} $ with $t_i \in T(\Pi)$ and one additional class
 \begin{equation}\label{eq:6.9}
   B_e (\Pi) := \iw(e) \sm e T(\Pi).
 \end{equation}
 \end{thm}

 \begin{defn}\label{def:6.10}
 We call $B_e(\Pi)$, for $e >d$,
  the \textbf{big cell} of $E$ in $\iw(e)$, and say that the $t_i \in S$ are the \textbf{extremal generators} with respect to $\E(\Pi,d)$.

 \end{defn}
 \begin{examp}\label{exmp:6.11}

 It is clear from Theorem \ref{thm:6.9} that,  for the  additive equivalence relation $\Ed$ introduced  in \S\ref{sec:1}, cf. ~\eqref{eq:1.10} and \eqref{eq:1.11}, the set $T(\Pi)$ is empty, and hence  $\iw(e)$ is the unique equivalence class of weight  $e > d$. For any other introduced relation $\E(\Pi,d)$, the set $T(\Pi)$ is nonempty.
 \end{examp}

If $T(\Pi) \neq \emptyset$, we further define
$B_d(\Pi) = \iw(d) \sm d T(\Pi)$, and call $B_d(\Pi)$ the \textbf{big cell of weight}~ $d$. Thus,
$B_d(\Pi)$ is  the union of all classes in $\iw(d)$ that are not singletons.
However, $B_d(\Pi)$ itself is  generally  not an equivalence class of ~$E(\Pi,d)$.

% Our study of the \as extension  of $\NN t$ in $R = \NNr / \Ed$ for nonzero $t \in \Rd$ in \S\ref{sec:4} can be readily extended to the present more general monoids $R$.
% \begin{thm}\label{thm:6.12} The statement of Theorem \ref{thm:4.12} remains valid in
% $R = \NNr / \E(\Pi,d)$ if $S \neq \emptyset$. In addition, if $t = n t_i$ for  $t_i \in S$, $n \in \N$ (not necessarily $n <d$), then $\NN t_i$ is the unique maximal \as extension of $\NN t$ in $R$.
% \end{thm}
%
% \begin{proof}
%   Assume that $t \in \Rd$, $t \neq 0$, and $t \notin \NN t_i$ for a generator $t_i \in S$. Running through the arguments following Remark \ref{rem:4.3}, we see that $\NN t_0$ with $t_0 = \prm{t}$ (= the primitive  element of $\NNr$) is still the unique maximal \as extension of $\NN t$. The generators $t_i \in S$ do not play a disturbing role here.
%
%   Assume finally that $t = n t_i$ with $t_i \in S$, $n \in \N$. Given $p < q$ and $z \in R$, for which
%   $pt + z = qt$, we conclude that $\supp(z) = \supp (qt) = \{ t_i\}$, and so $pt,z, qt \in \N t_i$. Thus  $z$ is uniquely determined by its weight, implying  $z = (q-p)t$. This gives the last claim of the theorem.
% \end{proof}
%

\section{Gaps and bridges for an almost subtractive submonoid $U$} \label{sec:a4}
Let $U$ be an \as submonoid of
$(R,+)$, equipped  with the partial ordering ~$\leq_R$ on $R$.

\begin{defn}\label{def:a4.1} Let $x \in R$.

\begin{enumerate} \ealph
  \item Given elements $u_0, u_1  \in U$, we say that $x$ is a \textbf{bridge from $u_0$ to $u_1$}, if $u_0 + x = u_1$.
  \item We call $x$ a \textbf{bridging element} for  $U$ in $R$, or simply say that $x$ is bridging,   if there exist $u_0, u_1 \in U$ such that  $u_0 + x = u_1$, i.e.,
      $(x + U) \cap U \neq \emptyset$. If $x$ is not bridging for $U$, i.e., $(x + U) \cap U = \emptyset$, we say that $x$ is \textbf{lonely for} $U$. (Roughly speaking, this means that  $x $ is useless for bridges.)
  \item We call $x$ a \textbf{gap element} (for $U$), if $x$ is bridging and $x \in R \sm U$.
  We then also say, for short,  that $x$ is a \textbf{gap}.

  \item We denote the set of  bridging elements by
  $\Bri(R,U)$, the subset of gap elements by $\Gap(R,U)$, and the set of lonely elements by $\Lo(R,U)$.
\end{enumerate}
\end{defn}
Note that there can be at most one bridge from $u_0$ to $u_1$, since $U$ is \as in $R$. For formal reasons, we include the case  $u_0 = u_1$, i.e., $x =0$ (the zero bridge).
      In other terms (Definition \ref{defn:1.1}), $x$ is the difference of $u_0$ and ~$u_1$.
      There exists  a bridge from $u_0$ to $u_1$ if and only if $u_0 \leq_R u_1$.
\begin{prop}\label{prop:a4.2}  $ $
\begin{enumerate} \ealph
\item $\Bri(R,U)$ is a submonoid of $(R,+)$ containing $U$.
\item
$\Lo(R,U)+ U \subseteq \Lo(R,U)$. Thus, $\Lo(R,U)$
is a union of cosets of $U$ in $R$.
%consists of cosets of $U$ in $R$.
\item $\Gap(R,U)$ is disjoint from all cosets of $U$ in $R$.
\end{enumerate}

\end{prop}

\begin{proof}
  (a): If $u_0 + x = u_1$ with $u_0, u_1 \in U$, then $(u_0 + u') + x = u_1 + u'$ for any $u'
   \in U$. This verifies that $\Bri(R,U) + U \subseteq \Bri(R,U) $. If, further, $v_0 + y = v_1$ with
   $v_0,v_1 \in U$, then
   $(u_0 + v_0 ) + (x+y) = u_1 + v_1$. This proves that $\Bri(R,U)$ is closed under addition. Clearly,  $ 0 \in \Bri(R,U)$, and thus $\Bri(R,U)$ is a submonoid of $R$.
   \pSkip
   (b): Given $x \in \Lo(R,U)$, suppose that $u +x \notin \Lo(R,U)$ for some $u \in U$. Then, $ u +x$ is bridging; that is,
   $u_0 + u +x = u_1$ for some $u_0, u_1 \in U$. Thus, $x $ is a bridge from $u_0 + u$ to $u_1$, a contradiction. Hence,
   $\Lo(R,U) + U \subseteq \Lo(R,U)$.
  \pSkip
   (c): If $x$  is a gap element, then $u_0 +x = u_1$ for some $u_0, u_1 \in U$, and thus $u_0+x$ is not a gap.
\end{proof}

\begin{defn}\label{def:a4.3} $ $
%\begin{enumerate} \ealph
%\item
A \textbf{walk} in $(R,U)$ (or, equivalently,  on $U$ in $R$) is an increasing sequence
\begin{equation}\label{eq:a4.3a}
u_0 \ds{\leq_R} u_1 \ds{\leq_R} u_2 \ds{\leq_R} \cdots \ds{\leq_R} u_r
\end{equation}
of elements of $U$. Then, there are unique bridging elements $z_1, \dots, z_r \in R$ such that
$u_{i-1} + z_i = u_i$ for $i = 1, \dots, r.$
We often denote such
a walk $P$ as
\begin{equation}\label{eq:a4.1}
     \xymatrixrowsep{3mm}
\xymatrixcolsep{6mm}
    \xymatrix@R=4em@C=2em{
  P:u_0 \ar@{->}[r]^{z_1} & u_1 \ar@{->}[r]^{z_2} & u_2 \ar@{->}[r] & \cdots  \ar@{->}[r]^{z_r} & u_r .
   }
\end{equation}
%\item
We say that  $P$ \textbf{starts} at $u_0$ and \textbf{ends} at $u_r$, and call $r$ the \textbf{length} of $P$, although this terminology may be somewhat counterintuitive when some $z_i = 0$.
%    \item
We say that $P$  is a \textbf{cycle},
if $u_r \in u_0 + U$, i.e., $z_1+ \cdots + z_r \in U$.
    We call $P$ a \textbf{gap walk}, if all  $z_i \in R \sm U$.
%    \item
 We call the  $u_i$, $ 0 \leq i \leq r$, the \textbf{nodes} of $P$.
A gap walk $P'$ is a \textbf{refinement} of $P$, if $P'$ starts at $u_0$, ends at $u_r$, and every  node of ~$P$ is also a node of $P'$.
%\end{enumerate}
\end{defn}

\begin{remark}\label{rem:a4.4}
$ $
\begin{enumerate} \ealph
\item In   \eqref{eq:a4.1}, we have
$u_0 + z_1 + \cdots + z_k = u_k$ for $k = 1,\dots, r$, and hence the sum
$z_1 + \cdots + z_k$ is again bridging. We call this sum the \textbf{bridge sum} of the walk.

\item If $x$ is a  gap element and $u' \in U$, then $u' + x$ may not be  a  gap element.  However,  if $u'+x$ is a  gap element, then $x$ is  a gap element.

\item Every walk $P$ on $U$ in $R$ is a walk on $U$ in $R \sm \Lo(R,U)$.

\item There is no gap walk  starting at $u_0 = 0$.
\end{enumerate}

\end{remark}
\begin{examp}\label{def:a4.5}
If $u_1 + z  = u_2$ with $u_1, u_2 \in U$ and  $z \in R$, then, for any
$u \in U$, we obtain the following two walks  of length $2$ from $u$ to $u_1 +u_2$:

\begin{equation}\label{eq:a4.2}
     \xymatrixrowsep{3mm}
\xymatrixcolsep{6mm}
    \xymatrix@R=1.0em@C=2em{
     & u + z  \ar@{->}[rd]^{u_1}
    \\
  u \ar@{->}[ru]^{z}
  \ar@{->}[rd]^{u_1} & & u + u_2
  \\ & u + u_1  \ar@{->}[ru]^{z} &
   }
\end{equation}

\end{examp}

%\begin{defn}\label{def:a4.6} We say that a walk \eqref{eq:a4.1} in $(R,U)$  is a \textbf{cycle},
%if $u_r \in u_0 + U$, i.e., $z_1+ \cdots + z_r \in U$.
%\end{defn}

\begin{lem}\label{lem:a4.7}
Let $u_1,u_2 \in U$ and $z \in R$ such that  $u_1 + z = u_2 + z  = u \in U$. Then,
$u_1 = u_2$.
\end{lem}

\begin{proof}
  If $u_1 + z = u_2 + z $, then $u_1 + u  = u_1 +  u_2 + z = u_2  +  u_1 + z = u_2 + u$.
  The claim then follows, since the monoid $U$ is cancellative, as noted already in \S\ref{sec:1}).
\end{proof}

\begin{thm}\label{thm:a4.8}
Let
$$       \xymatrixrowsep{3mm}
\xymatrixcolsep{6mm}
    \xymatrix@R=4em@C=2em{
  u_0 \ar@{->}[r]^{z_1} & u_1 \ar@{->}[r]^{z_2} & u_2 \ar@{->}[r] & \cdots  \ar@{->}[r]^{z_r} & u_r ,
   }
    \quad
         \xymatrixrowsep{3mm}
\xymatrixcolsep{6mm}
    \xymatrix@R=4em@C=2em{
 u'_0 \ar@{->}[r]^{z'_1} & u'_1 \ar@{->}[r]^{z'_2} & u'_2 \ar@{->}[r] & \cdots  \ar@{->}[r]^{z'_s} & u'_s
   }
$$
be gap walks in $(R,U)$ that end at the same element $u_r = u'_s$ and have the same bridge sum. Then,
$u_0 = u'_0$.
\end{thm}
\begin{proof}
   Follows by Lemma \ref{lem:a4.7}, since $u_0 + (z_1 + \cdots + z_r) = u_r = u'_s = u'_0 + (z'_1 + \cdots + z'_s)$.
\end{proof}

\begin{cor}\label{cor:a4.9}
Any two cycles in $(R,U)$ that  end at the same element of $U$ and have the same bridge sum start at the same element of $U$.
\end{cor}
\section{$U$-valued quadratic forms}\label{sec:a5}

We recall standard terminology and notation for quadratic forms on a module $V$ over a semiring $R$; cf.~\cite[\S2]{QF1}. A map $q: V\to R$ is a \bfem{quadratic form}, if
\begin{equation}\label{eq:1.1}
q(cx)= c^2q(x)
\end{equation}
for all $x\in V$  and $c\in R,$ and there exists a symmetric
$R$-bilinear form $b: V\times V\to R,$ called a \bfem{companion}
of ~$q$, such that
\begin{equation}\label{eq:1.2}
q(x+y)=q(x)+q(y)+b(x,y)
\end{equation}
for all $x,y\in V.$ We then say that $(q,b)$ is a
\bfem{quadratic pair} on $V.$ It follows from \eqref{eq:1.1} and~
\eqref{eq:1.2} that, for all $x\in V$,
$$4q(x)=2q(x)+b(x,x).$$
We call the companion $b$ of $q$ \bfem{balanced}, if
\begin{equation}\label{eq:1.3.2}
 b(x,x)=2q(x).
\end{equation}
We also  say that the quadratic pair $(q,b)$ is balanced.
A quadratic form $q:V\to R$ is \bfem{rigid}, if it has only one
companion.

Let  $R$ be an \ub semiring and $V$ a free $R$-module with a fixed finite basis
$(e_1, \dots, e_n)$. We denote a quadratic pair $(q,b)$ consisting of a quadratic form
$q: V \to R$ and a companion $b: V \times V \to R$ by a triangular scheme
\begin{equation}\label{eq:a5.1}
  (q,b) \hteq \begin{bmatrix}
                \al_1 & \bt_{12} & \bt_{13} & \cdots & \bt_{1n} \\
                & \al_2 & \bt_{23} & \cdots & \bt_{2n} \\
                &  & \ddots &  & \vdots \\
              &&  & \ddots & \bt_{(n-1)n} \\
                         &      && & \al_n  \\
              \end{bmatrix},
\end{equation}
where $\al_i = q(e_i)$ and $\bt_{ij} = b(e_i, e_j)$.
For any $\lm_1, \dots, \lm_n \in R$, we have
\begin{equation}\label{eq:a5.2}
 q(\lm_1 e_1 + \cdots + \lm_n e_n) = \sum_{i =1} ^n \lm_i^2 \al_i
 + \sum_{i < j }  \lm_i \lm_j  \bt_{ij}.
\end{equation}
This property follows from  the definition of a quadratic form on $V$, see  \cite[Eq. (0.1) and ~(0.2)]{QF1}.

The set of all quadratic forms on $V$, denoted by $\QF(V)$,  is an $R$-module with addition given by
\begin{equation}\label{eq:a5.3}
 (q + q')(x) = q(x) + q'(x),
\end{equation}
and scalar multiplication given  by
\begin{equation}\label{eq:a5.4}
(\lm q ) (x) = \lm \cdot q(x),
\end{equation}
for $q, q' \in \QF(V)$, $x \in V$, and $\lm \in R$.
The monoid $(\QF(V), + )$ is  upper bound, since the underlying semiring $(R,+)$ is upper bound.

Let $\Dg(V)$ denote the set of all quadratic forms $q \in \QF(V)$ with $\bt_{ij} = 0$ for all
$1 \leq i < j \leq n$, and let $\Tt(V)$ denote the set of all $q \in \QF(V)$ with
$\al_i =0 $ for all $1 \leq i \leq n$.
We call these quadratic forms, respectively,  \textbf{diagonal forms} and \textbf{upper trigonal forms} on $V$. Clearly, both $\Dg(V)$ and $\Tt(V)$  are $R$-submodules   of $\QF(V)$, and
\begin{equation}\label{eq:a5.5}
\QF(V) = \Dg(V) \oplus \Tt(V).
\end{equation}
Given a quadratic form  $q \in \QF(V)$ and a decomposition $q = q_0 + q_1$ with
$q_0 \in \Dg(V)$ and  $q \in \Tt(V)$, we call~ $q_0$ the \textbf{diagonal} part and $q_1$ the \textbf{upper triangular part} of $q$ (more precisely, of $(q,b)$). The diagonal part with entries
$\al_1, \dots, \al_n$ will  often be denoted by
$\dg(\al_1, \dots, \al_n)$.

%This is a sloppy notation.
The scheme \eqref{eq:a5.1} may depend on the choice of the basis  $e_1, \dots, e_n$.

\begin{convention}\label{conv:a5.1}
By the $R$-module $V$ we always mean the free $R$-module $V$ equipped with a fixed basis $e_1, \dots, e_n$. \end{convention}

Nevertheless, there exist semirings $R$  for which every free $R$-module has a unique  basis up to permutation. We exhibit one class of such semirings.
\begin{prop}\label{prop:a5.2}
  Let  $R$ be a \ub semiring with $\lm \geq_R 1$ for each $\lm \neq 0$ in $R$. Let~ $V$ be a free $R$-module with basis $(e_i \ds | i \in I)$. Then, $(e_i \ds | i \in I)$ is the set of all minimal elements of $V \sm \00$ with respect to $\leq_V$,\footnote{Note that, $(V,+)$ is \ub}
  and hence $(e_i \ds | i \in I)$ is  the unique basis of $V$ up to permutation.
\end{prop}

\begin{proof} Let $x \in V \sm \00 $. There is a unique subset $J$ of $I$ such that
$x = \sum_{j \in J} \lm_j e_j$ with  $\lm_j \neq 0$. Then,  $\lm_j \geq_R 1$  for all $j \in J$, which  implies that
$\lm_j e_j \geq_V e_j$,
$x \geq_V  \sum_{j \in J} \ e_j \geq_V e_i$ for each $i \in J$, and $x >_V e_i$ if $\{ i \} \varsubsetneqq J$.
  \end{proof}

  \begin{examp}\label{exmp:a5.3} Let
  $\NN[t_, \dots, t_p]$ be the polynomial semiring in variables $t_1, \dots, t_p$ over the semiring $\NN$ of nonnegative integers. The subsemiring
  $$R = \00 \cup (\N + \NN[t_1, \dots, t_p]), $$
    which consists of all polynomials with constant term in $\N $ together with the zero polynomial,
    has the minimal element $1$ in $R \sm \00$. Hence,  every free $R$-monoid, up to permutation, has a unique basis. The same holds, if we replace $\NN[t_1, \dots, t_p]$ by the semiring
    $\NN[[t_1, \dots, t_p]]$ of formal power series in variables $t_1, \dots, t_p$ over $\NN$.
  \end{examp}

Assume that $U$ is an \as submonoid of $(R,+)$, and $q:V \to U$ is a \textbf{$U$-valued quadratic form}, i.e., $q(x) \in U$ for every $x \in V$. Then,  Theorem \ref{thm:1.2} implies that
$q$ has a unique  companion $b = b_q$. Thus, we may omit the explicit mention of the companion $b$ in \eqref{eq:a5.1}. The coefficient  $\bt_{ij}$ is then determined by
\begin{equation}\label{eq:a5.6}
\al_i + \al_j + \bt_{ij} = q(e_i + e_j).
\end{equation}
Moreover, $q$ is also balanced by  Theorem \ref{thm:1.2}, so
\begin{equation}\label{eq:a5.7}
2 \al_i= b(e_i, e_i).
\end{equation}
When $q$ is $U$-valued, it may happen that some coefficient $\bt_{ij} \notin U$, in which case
$\bt_{ij}$ is a gap element by \eqref{eq:a5.6}.
However, even if all $\bt_{ij} \in U$, we cannot exclude, in general, that $b(x,y) \notin U$ for some
$x,y \in V$. In that case,  $b(x,y)$ is a gap element for $U$, since
$q(x) + q(y)+ b(x,y) = q(x+y)$.

We next consider cases where this does not happen; that is,  $b: V \times V \to R$ take  all its  values in $U$.
In general, we have the following result.
\begin{thm}\label{thm:a5.2}
Let $(q,b)$ be a quadratic pair on $V$ with trigonal  scheme \eqref{eq:a5.1}, where
\begin{equation}\label{eq:a5.8}
  \begin{array}{c}
    \al_i = q(e_i) \in U, \qquad b(e_i,e_i) = 2 \al_i, \qquad
    \bt_{ij} = b(e_i, e_j) \in U \text{ for } i<j.
  \end{array}
\end{equation}
Then, $q(x) \in U$ and $b(x,y) \in U$ for all  $x,  y$ in the additive submonoid $\sum_{i =1}^{n} \NN e_i$ of $(V,+)$ generated by $e_1,\dots, e_n $.
\end{thm}

\begin{proof} Let
$x = \sum_{i =1}^n \lm_i e_i$ and $y = \sum_{j =1}^n \mu_j e_j $ with $\lm_i, \mu_j \in \NN$.
Then,
$$ q(x) = \sum_{i =1}^n \lm_i^2 \al_i + \sum_{i < j} \lm_i \lm_j \bt_{ij} \in U, $$
since all $\lm_i^2$ and $\lm_i \lm_j$ are in $\NN$, and $U$ is closed under addition in $R$ .
In the same vein
\begin{align*}
  b(x,y) & = b\bigg( \sum_{i =1}^n \lm_i e_i,\sum_{j =1}^n \mu_j e_j \bigg) \\
   & =  \sum_{i =1}^n \lm_i \mu_i b(e_i, e_i) + \sum_{i < j} \lm_i\mu_j b(e_i, e_j)
   +  \sum_{i > j} \lm_i\mu_j b(e_i, e_j) \in U,
\end{align*}
because $b(e_i, e_i) = 2 \al_i \in U$, and
$b(e_i, e_j) = b(e_j, e_i) = \bt_{ij} \in U$ for $i <j$.
(Recall that any companion bilinear form is by definition symmetric.)
\end{proof}

\begin{cor}\label{ver:a5.3}
 If $R$ is the semiring $\NN$ of nonnegative integers, then the quadratic pair ~$(q,b)$ with trigonal scheme \eqref{eq:a5.1} represents  a $U$-valued quadratic form with its unique companion if and only if
  all $\al_i \in U$ and all  $\bt_{ij} \in U$.
\end{cor}
\begin{proof}
  Observe that in the  proof of Theorem \ref{thm:a5.2}  all $\lm_i^2 q(e_i)\in U$ and all
  $\lm_i \mu_j b(e_i, e_j) \in~ U$.
\end{proof}

More generally,  the same arguments give the following result.
\begin{thm}\label{thm:a5.4}
Assume that $U$ is an \textbf{ideal} of the  semiring  $R$, i.e., $U+ U \subseteq U$ and $\lm U \subseteq U$ for every $\lm \in R$, and that  $U$ is  \as in $(R,+)$. Let  $(q,b)$ be a quadratic pair on $V$ with a trigonal scheme \eqref{eq:a5.1}, in which
all $\al_i \in U$ and all $\bt_{ij} \in U$. Then, $q(x) \in U$ for all $x \in V$, and hence  $b$ is the unique companion of the $U$-valued quadratic form $q$.
\end{thm}

The \as property of $U$ in $(V,+)$ yields   various \as submonoids of $\QF(V)$.

\begin{thm}\label{thm:a5.7}  $ $
\begin{enumerate}\ealph
  \item

  The set $\QF(V,U)$ of $U$-valued quadratic forms on $V$ is an \as submonoid of $(\QF(V), +)$.
  \item The sets $$ \Dg(V,U) : = \Dg(V) \cap \QF(V,U)$$
  and $$ \Tt(V,U) : = \Tt(V) \cap \QF(V,U),$$
 which  consist of the $U$-valued diagonal and upper trigonal forms on $V$, are both \as submonoids of $(\QF(V),+)$.
  \item $\QF(V,U) = \Dg(V,U) \oplus \Tt(V,U)$.
  \item  The trigonal scheme of any $q \in \Tt(V,U)$ contains no gap elements, and
  $$
    q \hteq \begin{bmatrix}
                0 & \bt_{12} & \bt_{13} & \cdots & \bt_{1n} \\
                & 0 & \bt_{23} & \cdots & \bt_{2n} \\
                &  & \ddots &  & \vdots \\
              &&  & \ddots & \bt_{(n-1)n} \\
                         &      && & 0  \\
              \end{bmatrix}
$$
with all $\bt_{ij} \in U$.
\end{enumerate}
\end{thm}
\begin{proof} (a): Let $q_1,q_2 \in \QF(V,U)$ and  $\vrp, \psi \in \QF(V)$, and assume that
$q_1 + \vrp = q_1 + \psi = q_2$. Then, for every $x \in V$, we have
$$q_1(x) + \vrp(x) = q_1(x) + \psi(x) = q_2(x), $$
  where    $q_1(x),q_2(x) \in U$ and $\vrp(x), \psi(x) \in R$.
  Since $U$ is \as in $(R,+)$, it follows that $\vrp(x) = \psi(x)$ for every $x \in V$. Hence, $\vrp = \psi$.
\pSkip
    (b): Clear, since these sets are closed under addition in $\QF(V,U)$, where  $\QF(V,U)$ is an \as submonoid of $(\QF(V),+)$.
  \pSkip
  (c): Immediate from \eqref{eq:a5.5}.
  \pSkip (d): A consequence of \eqref{eq:a5.6}.
\end{proof}

\begin{remark}\label{rem:a5.8}
Let $q_1, q_2 \in \QF(V,U)$. If $q_1 + q_2 \in \Tt(V,U)$, then $q_1, q_2 \in T(V,U)$. Ditto for
$\Dg(V,U)$.
\end{remark}
\begin{defn}\label{def:a5.8} Let $V$ be an $R$-module, and let $q$ be $U$-valued quadratic form.
%\begin{enumerate} \ealph
 % \item Given a $U$-valued quadratic form $q$ on $V$,
  We call the entries in the trigonal scheme~
  \eqref{eq:a5.1} of $q$ (with its unique companion $b$) that are not in $U$ the \textbf{gaps of $q$} (with respect to the basis $e_1, \dots, e_n$ of $V$).
  %\item
  Every $\al_i = q(e_i)$ belongs to $U$, but some $\bt_{ij}$, with $i < j$, may lie in $R \sm U$, in which case they  are gap elements of $U$ in $R$, since
      $$ \al_i + \al_ j + \bt_{ij} = q(e_i + e_j) \in U.$$
      In this case,  we say that $q$ \textbf{has a gap at the place $(i,j)$} and that $\bt_{ij}$ is a \textbf{binary gap value} of $q$.
  %\item
  When  all $\bt_{ij} \in U$, we say that $q$ is \textbf{free of gaps}. We denote the set of such  quadratic forms by
      $\FG(V,U)$.
%\end{enumerate}
\end{defn}

Obviously, the sum of $q_1 + q_2$ of two $U$-valued quadratic forms $q_1$ and $ q_2$ that   are free of gaps, is again free of gaps, and  every $U$-valued diagonal form is free of gaps. Thus the following holds.
\begin{prop}\label{prop:a5.9}
$\FG(V,U)$ is a submonoid of $(\QF(V,U), +)$ containing  the submonoid $\Dg(V,U)$ of
$\QF(V,U)$.
\end{prop}

We look for ways to ``close'' (or to ``cover'') the gaps of $q$ by adding suitable   forms to $q$.

\begin{defn}\label{def:a5.10}
 Given a $U$-valued quadratic form $q$ on $V$, we say that a form $\vrp \in \FG(V,U)$ \textbf{covers
 the gaps of $q$} (or, is a \textbf{gap cover} of $q$), if
 $q + \vrp$ is free of gaps.
 \end{defn}
 For formal reasons, we do not exclude the case that $q$ itself is free of gaps.  In that case, the zero form  $\vrp = 0 $ is a gap cover of $q$.

\begin{prop}\label{prop:a5.11}
Let $\vrp_1, \vrp_2 \in \FG(V,U)$.
\begin{enumerate} \ealph
\item If $q \in \QF(V,U)$ and $\vrp$ is a  gap cover of  $q$, then, for
any $\psi \in \FG(V,U)$,  $\vrp + \psi$
 is also a gap cover of $q$.

\item If $\vrp_1$ is a gap cover of $q_1 \in \QF(V,U)$ and
$\vrp_2$ is a gap cover of $q_2 \in \QF(V,U)$,
then
$\vrp_1 + \vrp_2$ is a gap cover of $q_1 + q_2$.
\end{enumerate}
\end{prop}

\begin{proof} This follows from the study of bridges in $(R,U)$ as discussed in \S\ref{sec:a4}.
\pSkip
(a): Observe that, if $i< j$ and $u_0, u_1 \in U$ with
$ \bt_{ij} + u_0 = u_0' \in U$,
%qquad  \bt_{ij} + u_1 = u_1' \in U$$
then
$ \bt_{ij} + u_0 +u_1 = u_0' + u_1 \in U$.
\pSkip
(b): If  $\bt_{ij}, \gm_{ij}  \in R$ and $u_0,v_0, u_0', v_0' \in U$ are  such that
$ \bt_{ij} + u_0 = u_0' \in U$ and $\gm_{ij} + v_0 = v_0' \in U$,
then
$ \bt_{ij} + \gm_{ij} + u_0 +v_0 = u_0' + v_0' \in U.$ \end{proof}

For a  $U$-valued quadratic form $q \in \QF(V,U)$ with trigonal scheme \eqref{eq:a5.1}  we define the  form
$C(q) \in \Tt(V,U)$ to be the form whose entries are
$$ C(q)_{ij} :=
\left\{
  \begin{array}{ll}
    \al_i + \al_j, & 1 \leq i < j \leq n, \\[1mm]
    0, & i =j.
  \end{array}
\right.
$$
\begin{thm}\label{thm:a5.12}
$C(q)$ covers the gaps of the $U$-valued quadratic form  $q$.
\end{thm}

\begin{proof}
The entries of $q + C(q)$ are
$$ (q + C(q))_{ij}
= \left\{
\begin{array}{ll}
  \bt_{ij} + \al_i + \al_j = q(e_i +e_j) \in U, & i \neq j, \\[1mm]
  \al_i \in U , & i =j,
\end{array}
\right.
$$
for all
$1 \leq i < j \leq n$.
\end{proof}

\begin{remark}\label{rem:a5.13}
$ $
\begin{enumerate} \ealph
\item  The map
$$  \QF(V,U) \To \Tt(V,U), \qquad  q \mapsto C(q),$$
is an additive surjection. It restricts to the zero map on $\Tt(V,U)$ and to a monoid endomorphism
$\Dg(V,U) \hookrightarrow \Tt(V,U)$ of $\Dg(V,U)$.

\item If $U$ is an ideal of the semiring $R$, then $\QF(V,U) $ is an $R$-submodule of $\QF(V)$ by Theorem \ref{thm:a5.7}. Both $\QF(V,U)$ and $\Tt(V,U)$ are $R$-submodules of $\QF(V)$, and $q \mapsto C(q)$ respects scalar multiplication, i.e.,
%\begin{equation}\label{eq:a5.9}
 $ C(\lm q ) = \lm C(q)$
% \end{equation}
 for every $\lm \in R$ and $q \in \QF(V,U)$.

\item Let $q \in \QF(V)$. Then,
$$ q \in \Tt(V,U) \dss \Leftrightarrow C(q) = 0 \dss \Leftrightarrow q + C(q) = q.$$
\end{enumerate}

\end{remark}

\section{Walks and trails in a set of $U$-valued quadratic forms}\label{sec:b6}
Let $V$ be a free module with finite  basis $e_1, \dots, e_n$ over an \ub semiring ~$R$, and let  $U$ be  an \as submonoid of $(R,+)$. As stated in \S\ref{sec:a5}, the set $\QF(V)$ of forms on $V$ is a \ub $R$-module, and the set $\QF(V,U)$ of $U$-valued  forms on $V$
is an \as submonoid of $(\QF(V), +)$, which, under suitable conditions, is a submodule of $\QF(V)$, cf. Theorem \ref{thm:a5.4}\}.

We complement the terminology on gaps and bridges for $U$-valued forms introduced at the end of \S\ref{sec:a5} with analogues of the definitions in \S\ref{sec:a4} for  elements of $(R, U)$ (Definitions~\ref{def:a4.1} and~\ref{def:a4.3}).
\begin{defn}\label{def:b6.3} $ $
\begin{enumerate} \ealph
   \item
  We call a form $\zt \in \QF(V)$ a \textbf{bridge} from   $q_0 \in \QF(V,U)$ to
  $q_1 \in \QF(V,U)$, if $q_0 + \zt = q_1$. Such a bridge exists if and only if $q_0 \leq_{\QF(V)} q_1$, in which case it is unique.
  \item
  We call $\zt \in \QF(V)$ \textbf{bridging for $U$}, if there exists
  $q_0 \in \QF(V,U)$ such that  $q_0 + \zt \in \QF(V,U)$. Then, $\zt$ is a bridge  from $q_0$ to
  $q_0 + \zt$. We denote the set of all bridge forms by $\Bri(V,U)$, and the set of forms with gaps by $\Gap(V,U)$.
  \item

  We say that a form $\zt \in \QF(V)$ is \textbf{lonely} (for $U$), if $\zt$ is not bridging. In other words, $\zt$ is lonely if and only if
      $ (\zt + \QF(V,U)) \cap \QF(V,U) = \emptyset.$
      We denote the set of lonely forms on $V$ for $U$ by $\Lo(V,U)$.
\end{enumerate}

\end{defn}

\begin{prop}\label{prop:b6.2} $ $
 \begin{enumerate}\ealph
   \item $\Bri(V,U)$ is a submonoid of $(\QF(V), + )$ containing $\QF(V,U)$.
   \item  $\Lo(V,U) + \QF(V,U) \subseteq \Lo(V,U)$. Thus, $\Lo(V,U)$ consists of cosets of $\QF(V,U)$ in $\QF(V)$.
   \item $\Gap(V,U)$ is disjoint from all cosets of $\QF(V,U)$ in $\QF(V)$.
 \end{enumerate}
\end{prop}
\begin{proof}
  The proof is the same as that of Proposition~\ref{prop:a4.2}, mutatis mutandis.
\end{proof}

\begin{defn}\label{def:b6.3}  $ $

\begin{enumerate}\ealph
   \item
A \textbf{walk} in $\QF(V,U)$ is an increasing sequence
$ q_0 \ds{\leq_{\QF(V)}} q_1 \ds{\leq_{\QF(V)}} \cdots \ds{\leq_{\QF(V)}} q_r$
 in $\QF(V,U)$. Then, there exist unique forms
 $\zt_1, \dots, \zt_r$ in $\QF(V)$ for which  $q_{i-1} + \zt_i = q_i$ for every $1 \leq i \leq r$.
 We often denote such a walk $P$ by
 \begin{equation}\label{eq:b6.1}
     \xymatrix@R=0.9em@C=1.7em{
  P: \ q_0 \ar@{->}[r]^{\quad  \zt_1} & q_1 \ar@{->}[r]^{\quad  \zt_2} & q_2 \ar@{->}[r] & \cdots \ar@{->}[r]^{\zt_r} & q_r.
   }
   \end{equation}
\item
When every bridge $\zt_i$ has gaps, we say that $P$ is a \textbf{gap walk} in $\QF(V)$
for $U$.
\item
We call a form $\zt \in \QF(V)$ a \textbf{gap-zero form}, or say that $\zt$ is a \textbf{purely gap},
if   the diagonal entries of its trigonal scheme are in $U$, while  its off-diagonal entries are either zero or  elements of $R \sm U$.
 \item
In the case $n=2$, i.e.,
on $V = R e_1 + R e_2$,  we call a gap form a \textbf{binary gap form}.
\end{enumerate}

\end{defn}

 Based on binary gap forms, we construct special gap-zero walks  in $\QF(V,U)$, where $V = Re_1 + \cdots + R e_n$.
To this end,  we need the following elementary but central fact.

\begin{lem}\label{lem:b6.3}  $ $
\begin{enumerate} \ealph
  \item  The binary gap forms are precisely the forms
$\left[\begin{smallmatrix}
u_1 & z \\
& u_2 \\
\end{smallmatrix}\right]$
with $u_1, u_2 \in U$, $z \in R \sm U$, and $u_1 + z = u_2$ or
$u_2 + z = u_1$, which forces that $u_1 \neq 0$ and $u_2 \neq 0$.

  \item
  If $\left[\begin{smallmatrix}
u_1 & z \\
& u_2 \\
\end{smallmatrix}\right]$ is a binary gap form, then, for any $v \in U$, the form
$\left[\begin{smallmatrix}
u_1 + v & z \\
& u_2 + v \\
\end{smallmatrix}\right]$
is again a binary gap form.
\end{enumerate}
\end{lem}

\begin{proof} Part
(a) is clear by Definition \ref{def:a5.8}.(b), and part (b) follows from part (a).
\end{proof}

\begin{defn}\label{def:b6.5} Let $V = Re_1 + \cdots + R e_n$.
 \begin{enumerate}\ealph
\item
A
 \textbf{binary gap-zero} form $\zt$
 on $V$
 is a $U$-valued form on $V$ whose trigonal scheme exactly three nonzero entries  at
 positions $(r,r)$, $(r,s)$, $(s,s)$ with
 $1 \leq r < s \leq n$, such that
 $\eta = \smat{\al_r}{\bt_{rs}}{}{\al_s}$
  is a binary gap form.
  We then also say, that  $\zt$ is \textbf{inflated from $\eta$
to a form on} $V$ by inserting  zeros.
\item
A \textbf{gap-zero trail} (or, more briefly, \textbf{gap trail}) in $\QF(V,U)$ is a gap walk whose bridges are all  binary gap-zero forms.

\end{enumerate}
\end{defn}
\begin{construction}\label{cons:b6.5}

  We construct  a gap-zero trail
   \begin{equation}\label{eq:b6.2}
     \xymatrix@R=0.9em@C=1.7em{
  P: \ q_1 \ar@{->}[r]^{\quad  \zt_2} & q_2 \ar@{->}[r]^{\zt_3} & q_3 \ar@{->}[r] & \cdots \ar@{->}[r]^{\zt_n} & q_n
   }
   \end{equation}
on $V = R e_1 + \cdots + R e_n$, where $n \geq 2$,  starting with a form  $q_1$,  instead of $q_0$, since this labeling is easier to handle.
The  trigonal schemes of the forms $\zt_k$ and $q_k$  have
%Every form $\zt_k$ and $q_k$ will have in its trigonal scheme
only zeros in each row beyond the $k$'th row. We do not write  these rows.   We choose $\al_{10} \in U$ and begin  with
\begin{equation*}\label{eq:b6.3a}
q_1 =
\begin{bmatrix}
  \al_{10} & 0 & \cdots & 0
\end{bmatrix}.
\end{equation*} Then, we choose a binary gap form
  $\smat{  \al_{11} }{\bt_{12} }{}{
   \al_{2}},$
     define
  $$ \zt_2 = \footnotesize  \begin{bmatrix}
  \al_{11} & \bt_{12}  & 0 & \cdots & 0 \\
   & \al_{2}  & 0 & \cdots & 0  \\
    \end{bmatrix} =
  \begin{bmatrix}  \begin{array}{ll}
    \al_{11} & \bt_{12}  \\
   & \al_{2}    \\
   \end{array} &
   \bigzero &
   \end{bmatrix}
        ,$$
     and  obtain
  $$ q_2 = q_1  + \zt_2 = \footnotesize \begin{bmatrix}
  c_{1} & \bt_{12}  & 0 & \cdots & 0  \\
   & \al_{2}  & 0 & \cdots & 0  \\
    \end{bmatrix}
    =\begin{bmatrix}
    \begin{array}{ll}
  c_{1} & \bt_{12}   \\
   & \al_{2}    \\
   \end{array} & \bigzero &
    \end{bmatrix},$$
     where  $c_1  = \al_{10} + \al_{11}$.

 Next  we choose a binary gap form
  $\smat{  \al_{12} }{ \bt_{13} }{} {\al_{3}},$
define
  $$ \zt_3 = \footnotesize \begin{bmatrix}
  \al_{12} & 0 & \bt_{13}  &   \\
  & 0 & 0 &  \bigzero  &  \\
   &  & \al_{3}  &   \\
    \end{bmatrix},$$
     and  obtain
  $$ q_3 = q_2   + \zt_3  = \footnotesize \begin{bmatrix}
  c_{2} & \bt_{12}  & \bt_{13 }&   \\
   & \al_{2}  & 0 &  \bigzero     \\
   & & \al_{3} &  &
    \end{bmatrix},$$
     where $c_2  =  c_1 + \al_{12}$.

 In the $4$'th step, using  a  binary gap form
  $\smat{  \al_{13}}{ \bt_{14}}{}{ \al_{4}},$
we take
  $$ \zt_4 = \footnotesize \begin{bmatrix}
  \al_{13} & 0 & 0 &  \bt_{14}  \\
  & 0 & 0 &  0  & \bigzero & \\
&  & 0 & 0 &     & \\
   &  & & \al_{4}  &  \\
    \end{bmatrix},$$
     and  obtain
  $$ q_4 = q_3   + \zt_4  = \footnotesize \begin{bmatrix}
  c_{2} & \bt_{12}  & \bt_{13 }& \bt_{14}&   \\
   & \al_{2}  & 0 & 0 & \bigzero &  \\
   & & \al_{3} & 0 & \\
  &  & & \al_{4} &
    \end{bmatrix},$$
     where  $c_3  =  c_2 + \al_{13}$.

 In the final   step, using  a  binary gap form
  $\smat{  \al_{1 (n-1)} }{ \bt_{1 n} } {} { \al_{n}},$
we arrive at
  \begin{equation}\label{eq:b6.3}
 q_n  = \footnotesize \begin{bmatrix}
                c_{n-1} & \bt_{12} & \bt_{13} & \cdots & \bt_{1n} \\
                & \al_2 & 0& \cdots & 0 \\
                &  & \ddots &  & \vdots \\
              &&  & \ddots & 0 \\
                         &      && & \al_n  \\
              \end{bmatrix},
\end{equation}
where  $c_{n-1} = \al_{10} + \cdots + \al_{1(n-1)}$. Note that all diagonal entries  are in
$U \sm \00$.

\end{construction}

\begin{defn}\label{def:b6.7} We say that a gap-zero form $q$ on $V$ is \textbf{full}, if
every off-diagonal entry $\bt_{ij}$ of its trigonal scheme
 is in $R \sm U$.
%, i.e., $\bt_{ij}$ is a bridge in $(R,U)$
%from $\al_i$ to $\al_j$.
\end{defn}
We describe two ways to construct full gap-zero forms on $V = R e_1 + \cdots + R e_n$.
\begin{construction}\label{cons:b6.8}
We iterate Construction \ref{cons:b6.5}, starting with the form \eqref{eq:b6.3}. We omit the first row and the first column, and, by Construction \ref{cons:b6.5}, obtain from the diagonal form $\dg(\al_2, \dots, \al_n)$
 a gap-zero form which, together with the first row  and column, gives
 $$ \footnotesize
     \begin{bmatrix}
                c_1 & \bt_{12} & \bt_{13} & \cdots & \bt_{1n} \\
                & c_2 & \bt_{23} & \cdots & \bt_{2n} \\
                &  & \al_{3} &  & \vdots \\
              &&  & \ddots  & \bt_{(n-1)n} \\
                         &      && & \al_n  \\
              \end{bmatrix}.
$$
Proceeding  in this way,  after $n-1$ steps we obtain  a full gap-zero form
  $$ \footnotesize
     \begin{bmatrix}
                c_1 & \bt_{12} & \bt_{13} & \cdots & \bt_{1n} \\
                & c_2 & \bt_{23} & \cdots & \bt_{2n} \\
                &  & \ddots &  & \vdots \\
              &&  & c_{n-1} & \bt_{(n-1)n} \\
                         &      && & \al_n  \\
              \end{bmatrix}.
$$
Altogether, this form depends on the choice of $\al_{10} \in U$ in  Construction \ref{cons:b6.5}
and on the  $\frac{n(n-1)}2$ binary gap forms.
\end{construction}

\begin{construction}\label{cons:b6.9}
We choose a $U$-valued diagonal form
$ q_0 = \dg(\al_{10}, \al_{20}, \dots, \al_{n0}),$
and for every pair of indices $(i,j)$, $1 \leq i < j \leq n$, a binary gap form
$ \eta_{ij} = \smat{ \gm_{ij} }{ \bt_{ij}} {} { \dl_{ij} }$, where
$\gm_{ij},  \dl_{ij} \in U$ and $\bt_{ij} \in R \sm U$
              (and so $\gm_{ij} \neq 0$,$  \dl_{ij} \neq 0$).
              We inflate $\eta_{ij}$ to a gap-zero form with a trigonal scheme
              whose entries are all zero except at the positions
              $(i,i)$, $(j,j)$, $(i,j)$, where the  entries
              $\gm_{ij}$,  $\dl_{ij}$, and $\bt_{ij}$ respectively, are retained in their  place.
               Then,
                \begin{equation}\label{eq:b6.4}
                q = q_0 + \sum_{1 \leq i < j \leq n} \zt_{ij}
 \end{equation}
               is a full gap-zero form with trigonal scheme
  $$
    q =  \footnotesize \begin{bmatrix}
                \tlal_1 & \bt_{12} & \bt_{13} & \cdots & \bt_{1n} \\
                & \tlal_2 & \bt_{23} & \cdots & \bt_{2n} \\
                &  & \ddots &  & \vdots \\
              &&  & \tlal_{n-1} & \bt_{(n-1)n} \\
                         &      && & \tlal_n  \\
              \end{bmatrix} .
$$
Here all $\tlal_i \in U \sm \00$ and all $\bt_{ij} \in R \sm U$. The $\tlal_{i} $'s  can be computed explicitly  from the entries  $\al_{10},  \gm_{ij}, \dl_{ij} \in U$ with
$\gm_{ij} \neq 0$ and  $\dl_{ij} \neq 0$.
\end{construction}

Construction~\ref{cons:b6.5} provides a gap-zero trail \eqref{eq:b6.2}. We next describe a general method for constructing gap-zero trails based on \textbf{splicing} two gap-zero trails. \begin{construction}\label{cons:b6.10}
Given two gap-zero trails
$$
     \xymatrix@R=0.9em@C=1.7em{
  P: \ q_0 \ar@{->}[r]^{\quad  \zt_1} & q_1 \ar@{->}[r]^{\zt_2} &  \cdots \ar@{->}[r]^{\zt_k} & q_k} ,
  \qquad \xymatrix@R=0.9em@C=1.7em{
    P': \ q_k' \ar@{->}[r]^{\quad  \zt_{k+1}} & q_{k+1}' \ar@{->}[r]^{\zt_{k+2}} & \cdots \ar@{->}[r]^{\zt_{k+\ell}} & q_{k+ \ell}' ,
   }
 $$
 of length $k \geq 1$ and $\ell \geq 1$  in $\QF(V,U)$, respectively,  we use Lemma \ref{lem:b6.3}(b) repeatedly.
 Choose two forms $\gm, \dl \in \QF(V,U)$ such that  $q_k + \gm = q_k' + \dl$. This is always possible; for example, we may take $\gm = q'_k$, $\dl = q_k$. After adding $\gm$ to every
 $q_i $ in $P$ and $\dl $ to every  $q_j'$
 in $P'$, we  combine the trails $P + \gm$ and $P' + \dl$ to a gap-zero trail
$$
     \xymatrix@R=0.9em@C=1.5em{
  (P + \gm) \circ (P' + \dl): q_0 + \gm\ar@{->}[r]^{\qquad \quad   \zt_1} & q_1 + \gm \ar@{->}[r]^{\zt_2} & \cdots \ar@{->}[r]^{\hskip -25pt \zt_k} & q_k + \gm  =
   q_k' + \dl \ar@{->}[r]^{ \quad  \zt_{k+1}} & q_{k+1}' +\dl \ar@{->}[r]^{\quad \zt_{k+2}} & \cdots \ar@{->}[r]^{\zt_{k+\ell}} & q_{k+\ell}' .
   }
 $$
We call the pair $(\gm, \dl)$ a \textbf{clutch} for splicing $P$ and $P'$.
  \end{construction}

\begin{construction}\label{cons:b6.11}
The splicing process can be iterated in the obvious way. Given
 a sequence  $(\zt_1,  \dots , \zt_r)$ of binary  gap-zero forms on $V$ (for $U$), we  choose forms
$q_{00}, \dots, q_{(r-1)( r-1)}$ such that
$$ q_{00} + \zt_1 = q_{01}, \quad  q_{11} + \zt_2 = q_{12}, \quad  \dots, \quad q_{(r-1)(r-1)} + \zt_r = q_{(r-1)r}$$
with all $q_{ij}$ in $\QF(V,U)$. Splicing these forms together, we obtain a gap-zero trail
in $\QF(V,U)$
of length $r$ with bridges $\zt_1, \dots, \zt_r$.
\end{construction}

\begin{thm}\label{thm:b6.12}
Let $
     \xymatrix@R=0.9em@C=1.7em{
  P:  q_0 \ar@{->}[r]^{\quad  \zt_1} & q_1 \ar@{->}[r]^{\zt_2} & \cdots \ar@{->}[r]^{\zt_m} & q_m ,
}$ be a gap-zero trail of length  $m \geq 1$, the  $\zt_i$ need not be distinct.
Then every gap-zero trail $Q$ starting at $q_m$ and using only bridges from the set
$S = \{\zt_1, \dots, \zt_m \}$ can be spliced with $P$ to form a trail
$P \circ Q$ by clutches  $(\gm, \dl) = (0,0)$.
\end{thm}
\begin{proof}
 Since
  $ q_0 \ds{\leq_{\QF(V)}} q_1 \ds{\leq_{\QF(V)}}  \cdots \ds{\leq_{\QF(V)}} q_m$
 for every $\zt \in S$,  we have a gap-zero trail
$ \xymatrix@R=0.9em@C=1.1em{
  q_m \ar@{->}[r]^{  \zt\quad } & q_m + \zt
}$ of length $1$. These trails can be spliced together by clutches  $(0,0)$ to gap-zero trails of arbitrary length.
\end{proof}

In what follows, as in Definition \ref{def:b6.3}.(a),  we denote a walk in $\QF(V,U)$ by an increasing sequence of elements in
$\QF(V,U)$.
\begin{thm}\label{thm:b6.14}
Let $S$ be a finite set of binary gap-zero forms on $V$ and
$q_0 \in \QF(V,U)$ be  a $U$-valued form such that
$q_0 + \zt \in \QF(V,U)$ for every $\zt \in S$.\footnote{N.B. We encountered such a situation in Theorem \ref{thm:b6.12},  there with $q_m$ in place of  $q_0$.}
For any finite sequence $(\zt_1,  \dots, \zt_r)$ in $S$, the walk
\begin{equation}\label{eq:b6.5}
  q_0 \ds < q_0 + \zt_1  \ds < q_0 + \zt_1 +\zt_2 \ds< \cdots < q_0 + \zt_1 + \cdots \zt_r
\end{equation}
is a gap-zero trail. These are precisely the gap-zero trails starting at $q_0$ whose bridges belong to $S$.
\end{thm}
\begin{proof}
We have forms $q_i \in \QF(V,U)$
 defined inductively by
 $$ q_i = q_{i-1} + \zt_i, \qquad 1 \leq i \leq r.$$
 These forms can be spliced together using clutches $(0,0)$,  by the same arguments as in the proof of Theorem ~\ref{thm:b6.12}.
\end{proof}

\begin{example}\label{exmp:b6.15}
  We look for gap-zero trails of length $3$ when $S$ has $2$ or $3$ elements.

  \begin{enumerate}\ealph
    \item Let $S = \{ \zt_1, \zt_2\}$. Starting from $q_0$ as in Theorem~ \ref{thm:b6.14}, we have one gap-zero trail to $q_0 + 2 \zt_1$
        with sequence of bridges $(\zt_1, \zt_1) $, and $3$ gap-zero trails from $q_0$ to $q_0 + 2 \zt_1 + \zt_2$ with sequences of bridges
        $(\zt_1, \zt_1,\zt_2) $, $(\zt_1, \zt_2,\zt_1) $, $(\zt_2, \zt_1,\zt_1) $
        given by the diagrams

        $$
     
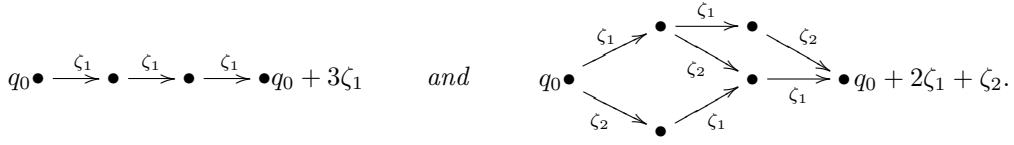
\xymatrix@R=0.9em@C=1.7em{ \\
   q_0 \bullet \ar@{->}[r]^{\quad  \zt_1} & \bullet \ar@{->}[r]^{\zt_1} &\bullet \ar@{->}[r]^{\hskip -15 pt \zt_1} & \bullet q_0 + 3\zt_1 &   \text{and} \qquad
}
     \xymatrix@R=0.7em@C=2.3em{
     &  \bullet \ar@{->}[r]^{  \zt_1}  \ar@{->}[rd]_{  \zt_2} & \bullet \ar@{->}[rd]^{  \zt_2} \\
   q_0 \bullet \ar@{->}[ru]^{\quad  \zt_1}  \ar@{->}[rd]_{\zt_2}  &  &\bullet \ar@{->}[r]_{\zt_1  } & \bullet & \hskip -10mm q_0 + 2\zt_1 + \zt_2 .\\
   &\bullet  \ar@{->}[ru]_{\zt_1}& &
}$$
The other gap-zero trails run from $q_0$ to $q_0 + 3 \zt_2$ and to $q_0 +  \zt_1 + 2 \zt_2$. They
are obtained from these trails by switching $\zt_1$ and $\zt_2$.
    \item
    Let $S = \{ \zt_1, \zt_2, \zt_3\}$.
     There are $6$ gap-zero trails of length $3$, which  use all three gap-zero forms
     $\zt_1, \zt_2, \zt_3$ as bridges. These are the trails from $q_0$ to
     $q_0 + \zt_1 + \zt_2 + \zt_3$ with sequences of bridges
      $(\zt_1, \zt_2,\zt_3) $,  $(\zt_1, \zt_3,\zt_2) $,  $(\zt_2, \zt_3,\zt_1) $,
       $(\zt_2, \zt_1,\zt_3) $,  $(\zt_3, \zt_1,\zt_2) $,  $(\zt_3, \zt_2,\zt_1) $.

 $$
     
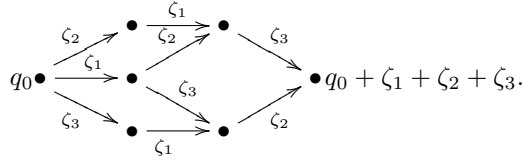
\xymatrix@R=0.7em@C=2.3em{
     &  \bullet \ar@{->}[r]^{  \zt_1 }  & \bullet \ar@{->}[rd]^{  \zt_3} \\
   q_0 \bullet \ar@{->}[ru]^{\zt_2}  \ar@{->}[rd]_{\zt_3} \ar@{->}[r]^{\quad \zt_1}   &\bullet \ar@{->}[ru]^{\zt_2} \ar@{->}[rd]^{\zt_3  }& & \bullet & \hskip -10mm q_0 + \zt_1 + \zt_2 +\zt_3.\\
   &\bullet  \ar@{->}[r]_{\zt_1 \quad}& \bullet \ar@{->}[ru]_{\zt_2}&
}$$

     \end{enumerate}
\end{example}
From Construction~\ref{cons:b6.5} onward, we have aimed to replace zero off-diagonal entries in the trigonal scheme of a gap-zero form by nonzero entries. There is a somewhat reverse process for an arbitrary $U$-valued quadratic pair on $V$.

\begin{prop}\label{prop:b6.15}
Let $(q,b)$ be a $U$-valued quadratic pair on $V$ with trigonal scheme \eqref{eq:a5.1}.
Given $r,s \in \N$, where $1 \leq r < s \leq n$,  replacing the entry $\bt_{rs}$ by zero yields  the scheme of a $U$-valued quadratic pair on $V$.
\end{prop}

\begin{proof}
  By formula \eqref{eq:a5.2}, we have
  \begin{equation}\label{eq:srt}
  q(\lm_1 e_1 + \cdots + \lm_n e_n) = \sum_{i=1}^{n} \lm_i ^2 \al_i + \sum_{i <j} \lm_i \lm_j \bt_{ij} \in U.
    \tag{$*$}
  \end{equation}
  Setting  $\lm_r = \lm_s =0$, we obtain
  $$
  \sum_{i\neq r,s } \lm_i ^2 \al_i + \sum'_{i <j} \lm_i \lm_j \bt_{ij} \in U,
  $$
  where the term
  $\lm_s \lm_r \bt_{rs}$ is omitted from  the second sum. We then add the sum
  $\lm_r^2 \al_r + \lm_s^2 \al_s \in U$. Thus, in \eqref{eq:srt},  only the off-diagonal term
   $\lm_s \lm_r \bt_{rs}$   is replaced by zero, while the resulting value remains in~ $U$.
Note that, if $q$ is gap-zero, then the new form is again gap-zero.
\end{proof}
\section{Extension of forms from $V$ to $S^{-1}V$}\label{sec:a6}

Let $R$ be a \ub semiring without zero divisors, i.e., the set $R \sm \00 $ is closed under multiplication (and hence is a monoid under multiplication).
\begin{construction}\label{cons:a6.1}
   Given a subset $S \subseteq R \sm \00$ that  is closed under multiplication (so $S \cup \{1 \}$ is a submonoid of $R \sm \00$),  we have an extension $S^{-1} R \supset R$ of the semiring $R$, consisting of formal  quotients $\frac xs$ (a very well known construction  in the case of rings rather than semirings), with the rules
  \begin{equation}\label{eq:a6.1}
  x = \frac {x s}s \  (\text{and so $x = \frac x1$, if $1  \in S$}), \quad \frac xs = \frac{tx}{ts}, \quad
    \frac{x_1}{s_1} \cdot
  \frac{x_2}{s_2} = \frac{x_1  x_2}{s_1 s_2 },
  \quad \frac{x_1}{s} +
  \frac{x_2}{t} = \frac{tx_1 + sx_2}{st}.
 \end{equation}
   The semiring $S^{-1}R$ is again \ub, and the ordering $\leq_{S^{-1}R}$ extends $\leq_R$.
  When $S = R \sm \00$, we obtain the unique maximal  extension,  denoted by $\Quot(R)$; that is,  the semiring  of all formal quotients of $R$.
\end{construction}

This construction is  compatible with \as submonoids of $(R,+)$ and $(S^{-1}R, +)$.
\begin{prop}\label{prop:a5.9} $ $
\begin{enumerate}\ealph
  \item If $U$ is an \as submonoid of $(R,+)$, then
  $$ S^{-1}U = \big \{ \textstyle{\frac xs} \ds | x \in U, s \in S \big \}$$
  is an \as submonoid of $(S^{-1}R,+)$.
  \item If $U'$ is an \as  submonoid of $(S^{-1}R,+)$, then $U = U' \cap R $ is an
  \as submonoid of $(R,+)$, and $ S^{-1}U$ is an  \as submonoid of $(S^{-1}R,+)$.
  % with    $S^{-1}U \subset U'$

  \item If $1 \in S$, then $U \subseteq  S^{-1}U$, and
  $S^{-1}U $ is  \as in $(S^{-1}R,+)$. Moreover, $tU \subseteq U$ for every $t\in S$.
\end{enumerate}
\end{prop}
\begin{proof}
  (a): Given  $u' , v' \in S^{-1}U$ and $z'_1, z_2' \in S^{-1}R$,  assume that $u' + z_1' = u' + z_2' =  v'$. We may write $u' = \frac us$, $v' = \frac vs$, $z_1' = \frac {z_1} s$,
  $z_2' = \frac {z_2} s$, with $u, v\in U$, $z_1, z_2 \in R$, and a common denominator $s$. We then obtain from
  $$ \frac us + \frac {z_1} s = \frac us + \frac {z_2} s = \frac vs,$$
  that  $u+ z_1 = u+ z_2 = v$. Since $U$ is \as in $R$, this implies that
  $z_1 = z_2$, and jence  $z_1' = z_2'$.
  \pSkip
  (b): $U = U' \cap R$ is an  \as submonoid of $S^{-1}R$, since $U$ is a submonoid of $U'$, and so
  $U$ is an  \as submonoid of $R$. By (a), it follows that  $S^{-1}U$  is \as in  $S^{-1}R$.
  \pSkip
  (c): If $1 \in S$, then $u= \frac u1 \in S^{-1}U$ for every $u \in U$.  If $t \in S$, then $tU' = \frac t1 U' \subseteq U'$, and thus
  $tU \subseteq R \cap U' = U$.
\end{proof}

\begin{problem}\label{prob:a5.10}
  If $U'$ is an \as submonoid of $S^{-1}R$ and  $U =  U' \cap R$, when do  both
  $U'$ and $S^{-1}U$ lie in a single maximal \as submonoid $U''$ of $S^{-1}R$?
  (Recall Corollary \ref{cor:1.4}).
\end{problem}

We exhibit a system of \as submonoids of $(\iS R, + )$ with a good control over the gap elements.

\begin{prop}\label{prob:a5.4}
  Let  $U$ be an \as submonoid of $(R,+)$, and let $S$ be a subset of $R \sm \00$ that is  closed under multiplication.
  \begin{enumerate} \ealph
    \item For every $s \in S$, the set $\frac 1 s R $ is a submonoid of $(\iS R, +)$. When  $s U \subseteq U $, the set
        $$\textstyle{\frac 1 s} U = \big\{ \textstyle{\frac us} \ds | u \in U \big \}$$
        is \as submonoid of $(\iS R, +)$ and also of $(\frac 1s R,  +)$.
    \item For any $z \in R$m the element  $\frac zs$ is a gap element of $(\frac 1s R, \frac 1 s U)$ if and only if $z$
    is a gap element of $(R,U)$.
  \end{enumerate}
\end{prop}
\begin{proof}
  (a): This is clear by Proportion \ref{prop:a5.9}(a) and rules \eqref{eq:a6.1}.
  \pSkip
  (b): Assume that $u,\tlu \in U$, $z \in R$, and $z' \in \frac 1s S$. Then,
  $$ u + s z' = \tlu \dss{\Leftrightarrow} \frac u s + z' = \frac{\tlu} s.$$
  Thus, the ``gap set'' $H_R(u, \tlu)$ in \eqref{eq:1.1} is a singleton $\{ z \}$ if and only if $H_{\frac 1s R} (\frac u s, \frac{\tlu}  s)$ is a singleton $\{ \frac zs\} .$
 \end{proof}

 \begin{remark}
 It follows by  rules
 \eqref{eq:a6.1} that, if $s,t \in S$, $sU \subseteq U$, and $t U \subseteq U$, then $st U \subseteq U $ and
 $\frac 1 s U  + \frac 1t U\subseteq \frac1 {st} U $.
 \end{remark}
We return to quadratic forms over semirings. Let again $V$ be a free $R$-module with basis $e_1, \dots, e_n$ and let $(q,b)$ be a quadratic pair on $V$. In direct analogy with Construction~
\ref{cons:a6.1}, we obtain a free $\iS R$-module $\iS V$ consisting of quotients $\frac xs$, where  $x  \in V$ and   $s \in S$,  with rules as in \eqref{eq:a6.1}, and, furthermore, a quadratic pair $(\iS q, \iS b)$ on $\iS V$. More explicitly, we view $V$ as a subset of $\iS V$ by identifying  $x = \frac{xs} s$ for every $x \in V$ and  $s \in S$, and then have
\begin{equation}\label{eq:a6.2}
  (\iS q) \bigg( \frac xt  \bigg) = \frac{q(x)}{t^2}, \qquad
  (\iS q , \iS b )\bigg( \frac xt, \frac yt   \bigg) = \frac{b(x,y)}{t^2},
  \end{equation}
for $x,y \in V$ and  $t \in S$. The trigonal scheme \eqref{eq:a5.1} representing $(q.b)$ can be read in $\iS R$, and then represents   $( \iS q, \iS b)$,
\begin{equation}\label{eq:a6.3}
  (\iS q, \iS b) \hteq \footnotesize \begin{bmatrix}
                \al_1 & \bt_{12} & \bt_{13} & \cdots & \bt_{1n} \\
                & \al_1 & \bt_{23} & \cdots & \bt_{2n} \\
                &  & \ddots &  & \vdots \\
              &&  & \ddots & \bt_{(n-1)n} \\
                         &      && & \al_n  \\
              \end{bmatrix}
\end{equation}
with $\al_i = q(e_i)$ and  $\bt_{ij} = b(e_i, e_j)$.

\section{Sums of gap elements}\label{sec:a8}

Let $U$ be an \as submonoid of $(R,+)$, equipped by the partial ordering $\leq_ R$ on $R$, usually  denoted by $\leq$.
\begin{defn}\label{def:a8.1} $  $
\begin{enumerate}
\ealph
\item
We call an element $u\in U$ \textbf{initial}, if there exists $z\in R \sm U$ such that
$u + z \in U$, and  $u$ \textbf{terminal}, if there exists $u_0 \in U$ and $z \in R \sm U$ with $u_0 + z = u$. In other words, $u$ is initial, if there is a gap walk
$ \xymatrix@R=0.9em@C=1.1em{
  u \ar@{->}[r]^{z \ } & u_1
}$
of length $1$ starting at $u$, and $u$ is terminal,
if there is a gap walk
$ \xymatrix@R=0.9em@C=1.1em{
  u_0 \ar@{->}[r]^{z} & u
}$.

\item We say that an element $x \in R $ is a \textbf{gap-sum of length} $r \geq 1$ in $R$, if
there exists a gap walk
$
     \xymatrix@R=0.9em@C=1.7em{
  P:  u_0 \ar@{->}[r]^{\quad  z_1} & u_1 \ar@{->}[r]^{z_2} & \cdots \ar@{->}[r]^{z_r} & u_r
}$
of length $r  $ with $ x = z_1 + \cdots + z_r$, cf. Definition  \ref{def:b6.3}.
Observe that  $x$ is then a bridge from $u_0$ to $u_r$.
% \item
We denote the set of all gap-sums in $R$ with respect to  $U$ by  $\GS(R,U)$.
\end{enumerate}
\end{defn}

\begin{remark}\label{rem:a8.2} $ $
\begin{enumerate}\eroman
  \item A lonely element of $U$ (cf. Definition \ref{def:a4.1}(b)) is neither initial nor terminal. However,  an element $u\in $ that  is not lonely can be both initial and  terminal. This happens when there exists  a gap walk $ \xymatrix@R=0.9em@C=1.1em{
  u_0 \ar@{->}[r]^{z_1} & u \ar@{->}[r]^{z_2} & u_2
}$ of length $2$ with $u$ in its middle.
  \item It can well happen that a gap-sum $x$ is an element of $U$.

\end{enumerate}

\end{remark}
\begin{prop}\label{prop:a8.3} The set
$\GS(R,U)$ is closed under addition, and thus is a subsemigroup of $(R,+)$. More precisely, if $x$ and $y$ are gap-sums of lengths $r$ and $s$, respectively, then $x+y$ is a gap-sum of length
$r+s$.
\end{prop}

\begin{proof}
We have increasing sequences
$u_0 < u_1 < \cdots < u_r$
and
$v_0 < v_1 < \cdots < v_s$ in $U$, with gap bridges
$z_1, \dots, z_r$ and  $z_{r+1}, \dots, z_{r+s}$, respectively, and gap-sums $x$ and $y$, respectively.
These sequences give rise to the  sequences
$u_0 + v_0 < u_1 + v_0  < \cdots < u_r + v_0$ and
$u_r + v_0 < u_r + v_1  < \cdots < u_r+ v_s$.
with the same gap bridges
$z_1, \dots, z_r$ and  $z_{r+1}, \dots, z_{r+s}$.
These sequences combine to form   the  sequence
\begin{equation}\label{eq:a8.0}
u_0 + v_0 < u_1 + v_0  < \cdots < u_r + v_0 < u_r + v_1  < \cdots < u_r+ v_s,
\end{equation}
in $U$,
with bridges
$z_1, \dots, z_r$, $z_{r+1}, \dots, z_{r+s}$ and gap-sum $x + y$.
\end{proof}
Given an element  $u \in U$, we  focus on  understanding the set  $\dwu  \cap (R \sm U)$ of gap elements in the downset ~$\dwu$. More precisely, we search for gap-walks
\begin{equation}\label{eq:a8.1}
     \xymatrix@R=0.9em@C=1.7em{
  P: \ u_0 \ar@{->}[r]^{\quad  z_1} & u_1 \ar@{->}[r]^{z_2} & u_2 \ar@{->}[r] & \cdots \ar@{->}[r]^{z_k} & u_k = u ,
}
\end{equation}
ending at $u$. Clearly, this is of interest only if $u$ is terminal. The following question is therefore natural.
Are there good cases in which the supremum of the lengths of these walks for a given $ u \in U$ is finite?
We then call this supremum  the \textbf{gap height}  of~ $u$, denoted by $h(u)$. If $u$ is not terminal, we assign
$h(u) = 0$.

In \S\ref{sec:6}, we already encountered a broad class of monoids $(R, U)$ with finite gap height, as follows.
\begin{thm}\label{thm:a8.4}
  Let $\NN^r$ be the free monoid on $r \geq 2$ generators (cf. Definition \ref{eq:6.1}),  and let
  $w: \NN^r \to \NN$ be the standard weight function on $\NN^r$ (cf. \eqref{eq:6.2}).
Let $R = \NN^r / E$, where  $E$ is a nontrivial weight-compatible additive equivalence on $\NN^r$ (cf. Theorem \ref{thm:6.9}, in particular  Example~ \ref{exmp:6.11}).
 Let  $d$ denote the largest number such that $E$ is trivial below height $d$, i.e., all
 $E$-equivalence classes of weight $< d$ are singletons, and assume that $d \geq 2$. If $U$ is an \as submonoid of $R$ with $U \neq R$, then for every $u \in U$, every gap walk $P$ ending at $u$ (cf.~ \eqref{eq:a8.1}) has a finite gap height $h(u) \leq w(u)$.
 \end{thm}

\begin{proof}
  In the setup  of \eqref{eq:a8.1}
the weights of the $u_i$'s  increase  by at least $1$ at each step $ \xymatrix@R=0.9em@C=1.1em{
  u_{i-1} \ar@{->}[r]^{z_i} & u_i
}$, which forces $ k \leq w(u)$. (The assumptions $r \geq 2$ and $d \geq 2 $ are  made to only avoid trivialities.)
\end{proof}

The following holds for any additive monoid with an \as \; submonoid $U$.
\begin{thm}\label{thm:a8.5}
Let $u,u' \in U$. If $u + u'$ has finite gap height, then both
$u$ and $u'$ have finite gap heights, and
$$ h(u) + h(u') \leq h(u +u '). $$
\end{thm}

\begin{proof}
  Any two gap walks ending at $u$ and $u'$,  of lengths $r$ and $s$, respectively, can be combined to form a gap walk of length $r+s$ ending at $u + u '$, as in the proof of Proposition \ref{prop:a8.3}.
\end{proof}

Thus, the set $\FGH(R,U)$ of elements of $U$ with finite gap height in $R$ is a submonoid of $(U, +)$.
In what follows, \textbf{we assume that every $u\in U$ has finite gap height in $(R,+)$}, so that  $U = \FGH(R,U)$.

Let $\SG$ be a finite subset of $\Gap(R,U)$, where $(R,+)$
an additive submonoid and $U$ an \as~submonoid.
We tacitly assume that $\SG \neq \emptyset$.
We have a general way to obtain subpairs
$(R',U')$ of $(R,U)$ in which  every $u' \in U' $ has finite gap height.
We define $Z(\SG)$ to be the subsemigroup  of $R \sm \00$ consisting of all finite  sums of elements of $\SG$, and put $Z(S)_0= Z(S)\cup \00$.

We introduce the submonoids $R_S = U + Z(S)_0$ of $(R,+)$ and
$U_S = [U \cap Z(S)] \cup \00$ of~ $(U,+)$, and have the diagram

$$
   \xymatrix@R=2em@C=1.2em{
U \ar@{^{(}->}[rr]  && U + Z(S)_0   \\
U_S \ar@{^{(}->}[rr]  \ar@{_{(}->}[u]  && Z(S)_0   \ar@{_{(}->}[u]
}
$$
of submonoids of $(R,+)$.
Note that, if $T$ is another finite subset of $\Gap(R,U)$, then
\begin{equation}\label{eqL}
  R_S + R_T \subseteq R_{S\cup T}, \qquad U_S + U_T \subseteq U_{S \cup T}.
\end{equation}
We remain  with a fixed finite set
$S \subseteq \Gap(R,U)$, and  often write $Z(S)= Z$ for short.

Assume   without essential  loss of generality  that
$\SG \neq \emptyset$ and
all its  elements  pairwise distinct, i.e.,
$\SG = \{s_1, \dots, s_p \}$  with $s_i \neq s_j$  for $ i\neq j$ and $\ p \geq ~1  $.
Then, $Z= (\NN s_1 + \cdots + \NN s_p ) \sm \00$.

\begin{defn}\label{def:a8.8}
Every $z \in Z$ has a \textbf{presentation}
\begin{equation}\label{eq:a8.3}
z= n_1 s_1 + \cdots + n_p s_p \text{ with }n_i \in \NN \text{ and }  \ \sum_{i= 1}^p  n_i >0.
\end{equation}
We call
$n_1 + \cdots + n_p$
 the \textbf{length} of the sum $z$ in the presentation \eqref{eq:a8.3}. \end{defn}
%the set $\{s_i \ds | n_i \neq 0 \} $ the \textbf{support} of $z$ in $S$, and

The  presentation may not be unique, however  $z$ has only finitely many presentations, since the set $\SG$ is finite.
Note that $z$ is a gap element of $(R,U)$. Indeed, if $u_i + s_i = \tlu_i$ with
$u_i, \tlu_i \in U$ for   $1 \leq i \leq p$, then
 \begin{equation}\label{eq:a8.9}
  n_1 u_1 + \cdots + n_p u_p + z =
 n_1 \tlu_1  + \cdots + n_p \tlu_p.
  \end{equation}

\begin{defn}\label{def:a8.11}
We call a gap walk in $(R,U)$ an \textbf{$\SG$-walk}, if all its gaps are belong to $\SG$, i.e., are sums of elements of~ $\SG$.
\end{defn}

\begin{thm}\label{thm:a8.9}
Let $u_0, u \in U_\SG$ and $z \in Z$ be such that  $u_0 + z = u$. Then, every $\SG$-walk
  $$
   \xymatrix@R=0.9em@C=1.7em{
   u_0 \ar@{->}[r]^{\quad  z_1} & u_1 \ar@{->}[r]^{z_2} & u_2 \ar@{->}[r] & \cdots \ar@{->}[r]^{z_k} & u_k = u ,
}$$
in $(R_\SG, U_S)$ has length $k \leq r(u_0,u)$ for some
$r(u_0,u) \in \NN$ depending only on $u_0$ and $ u$.
\end{thm}

\begin{proof}
 There is a unique $z \in Z$ for which  $u_0 + z = u$. We refine the given gap walk to a gap walk
  $$
   \xymatrix@R=0.9em@C=1.7em{
   \ u_0 = u'_0 \ar@{->}[r]^{\quad  z'_1} & u'_1 \ar@{->}[r]^{z'_2} & u'_2 \ar@{->}[r] & \cdots \ar@{->}[r]^{z'_\ell} & u'_\ell = u ,
}$$
with all $z'_j \in \SG$ and  $\ell \geq k$. (Typically, $\ell$ much larger  than $k$.)
This provides a presentation
$
z = z'_1 +  \cdots + z'_\ell$ of $z$ with $u_0 + z = u$. Since, for a given $u$, there are only finitely many such presentations $r = r(u_0,u) \in \N$ of length
$\ell \leq r$, where $r$ depends only on $u_0$ and ~$u$. Thus,
$k \leq r(u_0, u)$.
 \end{proof}

\begin{cor}\label{cor:a8.12}
Let
  $
   \xymatrix@R=0.9em@C=1.7em{
  P:   u_0 \ar@{->}[r]^{\quad  z_1} & u_1 \ar@{->}[r]^{z_2} & \cdots \ar@{->}[r]^{z_k\quad} & u_k = u
}$
be  an $\SG$-walk in $(R,U)$. Then every $\SG$-walk \\
  $
   \xymatrix@R=0.9em@C=1.7em{
  P':  u_0 \ar@{->}[r]^{\quad  z'_1} & u'_1 \ar@{->}[r]^{z'_2} & \cdots \ar@{->}[r]^{z'_\ell\qquad } & u_\ell' =  u_k = u
}$ that
refines  $P$ has length
$ \ell \leq r(u_0, u_1) + \cdots + r(u_{k-1}, u).$

\end{cor}

\begin{proof}
  Apply Theorem \ref{thm:a8.9} to each of the  gap walks
  $ \xymatrix@R=0.9em@C=1.1em{
  u_{i-1} \ar@{->}[r]^{z_i} & u_i
}$ for
  $1 \leq i \leq k$.
\end{proof}
Assume now that $R$ is a semiring (rather than merely an additive monoid) without zero-divisors, i.e., $R \sm \00$ is closed under multiplication.
As before, let $U$ be an \as submonoid of $(R, +)$. \{In contrast to \S\ref{sec:a6}, we now use the letter ``$t$'' instead of
``$s$'' to avoid a conflict with the notation above.\}

\begin{defn}\label{def:a8.9}
Let $\Om(R,U)$ denote the set of all $t \in R$ such that
$$ \forall x \in R: \ tx \in U \dss \Leftrightarrow x \in U, $$
equivalently: $tU \subseteq U$ and $t(R \sm U) \subseteq R \sm U$.
\end{defn}

Note that $\Om(R,U)$ is closed under multiplication and thus   is a multiplicative submonoid of $R \sm \00$. In what follows, we write $\Om$ for
$ \Om (R,U)$ for short.
\begin{lem}\label{lem:a8.10}
For any gap element $z$ of $(R,U)$ and any $t \in \Om$, the element $tz$ is a gap element of $(R,U)$.
\end{lem}
\begin{proof}
  If $u_0 + z = u_1$ with $u_0,u_1 \in U$ and $z \in R \sm U$, then
  $t u_0 + t z = t u_1$, where  $t u_0, t u_1 \in U$ and $tz \in R \sm U$.
\end{proof}

We have at hand a  \textbf{gap height function}
$$ h: U \to \NN \text { with } h^{-1}(0) =\00.$$
 In this setup, Theorem \ref{thm:a8.5} reads simply as:
   \begin{equation*}\label{eq:a8.8}
h(u) + h(u') \leq h(u +u')
\end{equation*}
for all $u \in U$.
We also have the following fact.
\begin{thm}\label{thm:a8.11} $h(u) \leq h(tu)$
 for every  $u \in U$ and $t \in \Om(R,U)$.
 %
% If $su$ has finite gap height in $(R,U)$, then $u $ has
%   finite gap height in $(R,U)$ and $h_R(u) \leq h_R(su)$.
%%  \item  Let $u \in \Om^{-1} U$. Then $u'$ has  finite gap height in $(\Om^{-1}R,\Om^{-1}U)$ iff $u' = \frac ut$ for some $u \in U$, $t \in \Om$ with $h(u) < \infty$, and
%      $h_{\Om^{-1}}(u) = h_R(tu)$.
\end{thm}
\begin{proof}
  If $     \xymatrix@R=0.9em@C=1.7em{
  P:  u_0 \ar@{->}[r]^{\quad  z_1} & u_1 \ar@{->}[r]^{z_2} &  \cdots \ar@{->}[r]^{z_k\quad } & u_k = u ,
}
$ is a gap walk in $(R,U)$ of length~ $k$, then \\
$      \xymatrix@R=0.9em@C=1.7em{
  tP:  t u_0 \ar@{->}[r]^{\quad  t z_1} & tu_1 \ar@{->}[r]^{t z_2}  & \cdots \ar@{->}[r]^{tz_k \quad } & tu_k = tu
}
$
is also a gap walk by Lemma \ref{lem:a8.10}.
\end{proof}

\section{$\SG$-sums and $\SG$-walks in $(R,U)$}\label{sec:a9}
We proceed by studying $\SG$-sums and $\SG$-walks in detail, cf.
Definition \ref{def:a8.11}. Let $U$ be a cancellative submonoid  of $(R,+)$.

\begin{prop}\label{prop:a9.1}
 If $     \xymatrix@R=0.9em@C=1.7em{
  P: u_0 \ar@{->}[r]^{\quad  z_1} & u_1 \ar@{->}[r]^{z_2} & u_2 \ar@{->}[r] & \cdots \ar@{->}[r]^{z_k\quad } & u_k
}
$ is an $\SG$-walk in $(R,U)$, then,
for every $u' \in U$,   the sequence
 $$     \xymatrix@R=0.9em@C=1.7em{
  P + u' :  u_0 + u' \ar@{->}[r]^{\quad  z_1} & u_1 + u' \ar@{->}[r]^{z_2} & u_2 + u'\ar@{->}[r] & \cdots \ar@{->}[r]^{z_k\quad } & u_k + u'
}
$$ is again an $\SG$-walk in $(R,U)$. It has the same length $k$ as $P$ and  the same
bridges $z_i \in Z$.
\end{prop}

\begin{proof}
  Clear, since $u_{i-1} + z_i = u_i$ implies
  $u_{i-1} + u' + z_i = u_i + u'$.
\end{proof}

\begin{defn}\label{def:a9.2}
Given $u_0 \leq_R u$ in $U$, we define $h_\SG(u_0,u)$ to be  the maximum of the lengths of all $\SG$-walks from~ $u_0$ to $u$. When $\SG$ is fixed, we  write
 $h_\SG(u_0,u) = h(u_0,u)$.
\end{defn}
This maximum is finite by Theorem~\ref{thm:a8.11}, since there are only finitely many such $\SG$-walks.
% Now a key  notion for the whole section.
\begin{defn}\label{def:a9.3}
The \textbf{stability region}  of $\SG$ in $U$, denoted by $\stab_S(R,U)$, is the set of all $u \in U$ such that  $u+s \in U$ for every $s \in \SG$.

\end{defn}
\begin{thm}\label{thm:a9.3} The stability region
$\stab_S(R,U)$ is a  nonempty upset in the cancellative monoid $U $, i.e.,  if $ u\in \stab_S(R,U)$, then $ u  + u' \in \stab_S(R,U)$ for all
$ u'
 \in U$.
\end{thm}
\begin{proof}
Choose
$u_i, u'_i \in U$ with
$u_i + s_i = \tlu_i$ for $1 \leq i \leq p$. Then,
$ u = u_1 + \cdots+ u_p + u' \in U$ for any
$u' \in U \cup \00$,
and
$$ u' + u +s_i = u' +\tlu_i + \sum_{j \neq i} u_j \in U.$$
This implies  that $ u\in \stab_S(R,U)$  (take $u' =0$),  and hence that  $ u' + u \in \stab_S(R,U)$.
\end{proof}

\begin{defn}\label{def:a9.4}
We call an $\SG$-walk $P$ \textbf{maximal}, if there is no $\SG$-walk $P' \neq P$ that refines~ $P$.
\end{defn}
This notion of maximality  strongly depends on the choice of $\SG$. Often,  an $\SG$-walk can be refined further by enlarging  $\SG$.
\begin{thm}\label{thm:a9.5}
 If $     \xymatrix@R=0.9em@C=1.7em{
  P: u_0 \ar@{->}[r]^{\quad  z_1} & u_1 \ar@{->}[r]^{z_2} & \cdots \ar@{->}[r]^{z_k\quad } & u_k = u
}
$ is a maximal  $\SG$-walk  from $u_0$ to $u$ in $(R,U)$, then
every
 $ \xymatrix@R=0.9em@C=1.1em{
  u_{i-1} \ar@{->}[r]^{z_i} & u_i
}$ is a
maximal $\SG$-walk from $u_{i-1}$ to $u_i$,
and
\begin{equation*}\label{eq:a9.1}
  h(u_0, u ) \ds \geq h(u_0, u_1 ) + h(u_1, u_2 ) + \cdots + h(u_{k-1}, u_k ).
\end{equation*}
\end{thm}
\begin{proof}
Every bridge $z_i$ in $P$ is an element of $\SG$ rather than merely of $Z$. Hence,
 $ \xymatrix@R=0.9em@C=1.1em{
  u_{i-1} \ar@{->}[r]^{z_i} & u_i
}$  has the single  bridge $z_i \in \SG$, and  is therefore trivially maximal.
\end{proof}

\begin{schol}\label{sch:a9.3}
Let $u_0,u \in U$ such that $u_0 < _R u$.
\begin{enumerate} \ealph
  \item $h_\SG(u_0,u) = 0$ if and only if the unique $x \in R$ such that  $u_0 + x = u $ is not in $Z$.
  \item Let $h_\SG(u_0,u) > 0$. Every maximal $\SG$-walks from $u_0$ to has a length determined by a presentation of the gap
      $z \in Z$ satisfying $u_0 + z = u$, as follows.  We successively add   elements $s_i \in \SG$ to $u_0$ occurring in the presentation of $z$ until we reach $u$. This can be added in any order. Every  such construction  gives us a maximal $\SG$-walk  from $u_0$ to $u$.

\end{enumerate}
\end{schol}

For elements $u$ in the submonoid $U_\SG$ of $U$ defined in \eqref{eq:a8.0}, the gap height
$h(u)$ is related to the numbers $h_\SG(u_0, u)$ as follows.
\begin{prop}\label{prop:a9.7} Let
$u \in U_\SG$. Then
$ h(u) \geq \sup \{ h_\SG(u_0, u) \ds | u_0 \in U, u_0 < u \}.$
\end{prop}

\begin{proof}
  Every gap walk in $(R_\SG, U_\SG)$ ending at $u$ starts at some
  $u_0 < u$.
\end{proof}
%
%\begin{defn}\label{def:a9.8}
%We call a finite set $\SG \subset \Gap(R,U)$ \textbf{minimal}, if no element of $\SG$ is a sum of other elements of $\SG$, i.e., if $\SG = \{s_1, \dots, s_p \}$ and all $s_i$ are different,
%then
%$s_i \neq  \sum_ {j \in J} s_j$ for $1 \leq i \leq p$, $J \subset \{ 1, \dots, p\}$, $i \notin J$.
%\end{defn}
%
%In the following \emph{ unless otherwise is specified we  assume that $\SG$ is minimal}, but we do not exclude the case   that $\sum_{i \in J} s_i =
%\sum_{i \in K} s_i$ for
%some $J, K \subset \{ 1, \dots, p\}$ with $J \cap K = \emptyset$, $|J| > 1$, $|K| > 1$.
%Given  a finite subset $\SG'$ of $\Gap(R,U)$, a minimal set $\SG \subset \SG'$ can be obtained by omitting  elements of $\SG'$.

\begin{defn}\label{def:a9.10}
Let $u \in U$.
\begin{enumerate}\ealph
  \item  The  \textbf{$\SG$-star} of $u$ is the subset
      $$ \star_\SG(u) := (u +\SG) \cap U$$
      of $U$. The elements  of $ \star_\SG(u)$ give  the $\SG$-walks  of length $1$ that
      start  at $u$.
  \item  We say that  $u$ is \textbf{lonely} for $\SG$ (or $\SG$-\textbf{lonely} for short), if $\star_\SG(u) = \emptyset$, and denote the set of such $u \in U$ by $\Lo_S(R,U)$.
\end{enumerate}
\end{defn}

We examine the family of $\SG$-stars $\star_\SG(u) \subseteq U$, with  $u \in U$, and its impact on the stability regions $\stab (\SG',U)$, where $\emptyset \neq \SG' \subseteq \SG$.

%The stability regions of the nonempty subsets $T$ of $S$ are governed by the $S$-stars $\star_S(u)$, $u \in U$ as follows.
%\begin{remark}\label{rem:a9.10}
%Clearly
%$\star_T(u) \subset \star_S(u) \cap T$, and so
%$$ \star_T(u) = T \dss\Leftrightarrow \star_S(u) \supset T.$$
%Thu $\stab(T,U)$ is the intersection of all $S$-stars $\star_S(u)$, $u \in U$, containing $T$.
%\end{remark}

\begin{lem}\label{lem:a9.10}
$\star_\SG(u_1) \cup \star_\SG (u_2) \subseteq \star_\SG (u_1+ u_2)$ for any $u_1, u_2 \in U$.
%  \begin{enumerate}\ealph
 %   \item  $\star_\SG(u_1) \subset \star_\SG (u_1+ u_2)$,
 %   \item $\star_\SG(u_1) + \star_\SG  (u_2) \subset \star_\SG (u_1+ u_2)$,
    %\item

 % \end{enumerate}
\end{lem}

\begin{proof}
   Since $u_1 + s = \tlu_1 \in U$, we have  $u_1 + u_2 + s = \tlu_1 + u_2 \in U $
   whence $\star_S(u_1) + u_2 \subseteq \star_S(u_1 + u_2)$. By symmetry,  we obtain that
   $\star_S(u_2) + u_1 \subseteq \star_S(u_1 + u_2)$.
   %
%   . This proves (a) and (c).
%  If $u_1 + s_1 = \tlu_1 \in U$ and $u_2 + s_2 = \tlu_2 \in U$, then
%  $(u_1 +u_2) + (s_1 + s_2) = \tlu_1 + \tlu_2 \in U$. This proves (b).
\end{proof}

\begin{thm}\label{thm:a9.11}
$\stab_S(R, U)$ is a subsemigroup of $(U,+)$, consisting of cosets of $U$ in $R$.
%The set of $S$-stars
%$\{ \star_\SG(u) \ds | u \in U \}$ is an upward directed family of subsemigroups of $(U \sm \00,+)$, and gives an upward directed family
%$\{ \star_S(u) \cup \00 \ds | u \in U\} $
%of submonoids of $(U,+)$. It has the unique maximal element $U$.
\end{thm}

\begin{proof}
Lemma \ref{lem:a9.10} implies that
$\stab_S(R,U) + U \subseteq \stab_S(R,U)$. Since $0 \in U$, we have
  $\stab_S(R,U) + U =  \stab_S(R,U)$, and hence
  $\stab_S(R,U) + \stab_S(R,U) \subseteq \stab_S(R,U) + U = \stab_S(R,U)$.
  %This is an immediate consequence of Lemma \ref{lem:a9.10}.
\end{proof}
These cosets are disjoint from the set $\Lo_S(R,U)$ of $S$-lonely elements of $U$, as observed earlier.

\begin{remark}\label{rem:a9.12} Let $S_1$ and $S_2$ be  nonempty subsets of $S$.
\begin{enumerate}\ealph
  \item $\stab_{S_1}(R,U) \cap  \stab_{S_2}(R,U) = \stab_{S_1 \cup S_2}(R,U) $.

  \item Trivially,
  $ u + (S_1 \cup S_2) = (u +S_1) \cup (u + S_2)$, and thus
  $$ u + (S_1 \cup S_2) \subseteq U \dss{\Leftrightarrow} (u + S_1) \subseteq U \text{ and }
  (u + S_2) \subseteq U.$$
  In particular,    $S_1 \subseteq S_2 \ds \Rightarrow \stab_{S_1}(R,U) \supseteq   \stab_{S_2}(R,U) $.
\item The maximal elements of  the family
$\{ \stab_{S'}(R,U) \mid \emptyset \neq S' \subseteq S,\} $   are among the submonoids
$ \stab_{s}(R,U) := \stab_{\{ s \}}(R,U) = \{ u\in U \mid u+s \in U\}$
with $s$ range over the set $S$.

\item
$\stab_{S}(R,U) =  \bigcap _{s\in S}\stab_{\{s\}}(R,U)$ .

\end{enumerate}

\end{remark}
We establish a filtration of the subsemigroup $\stab_{S}(R,U)$  of $(U,+)$, partially ordered by $\leq_R$. %First notice the following.

\begin{defn}
We say that a finite subset $S$ of $\Gap(R,U)$ is \textbf{reduced}, if no element $s \in S$ is a sum of elements of  $S \sm \{ s \}$. In other words, every $s \in S$ is an atom of $(Z(S),+)$.
\end{defn}
Note that from any finite subset of $\Gap(R,U)$,  we can obtain a reduced set $S'$ with
$Z(S') = Z(S)$ by  successively omitting ``superfluous'' elements.
%  until we have got a reduced set.

\begin{remark}
For any $u \in \stab_S(R,U)$, there exists a minimal element $u_0$ in $\stab_S(R,U)$ such that
$u_0 \leq u$, as follows from Theorem \ref{thm:a8.5}.
\end{remark}

We now define the filtration of $\stab_S(R,U)$ as follows.

\begin{defn} Assume that $S$ is reduced.
\begin{enumerate} \ealph
  \item $\stab_S(R,U)_0$ is the set of minimal elements of $\stab_S(R,U)$. For
  $m \in \NN$, we define inductively
  $\stab_S(R,U) _{m+1} = \stab_S(R,U)_m + S$. Since
  $u' + S \subset U$ for every
  $u' \in \stab_S(R,U)$, it follows that
  $\stab_S(R,U)_m = \stab_S(R,U)_0 + Z_m(S) \subset \stab_S(R,U)$.

  \item We say that a maximal $S$-walk
  $     \xymatrix@R=0.9em@C=1.7em{
  P:  u_0 \ar@{->}[r]^{\quad  s_1} & u_1 \ar@{->}[r]^{s_2} & \cdots \ar@{->}[r]^{s_m\quad } & u_m
}
$ in $\stab_S(R,U)$ is \textbf{rooted}, if
$u_0 \in \stab_S(R,U)$, and we then call $u_0$ the \textbf{root} of $P$.

\end{enumerate}
\end{defn}

The following is now obvious.

\begin{thm} If $S$ is reduced, then a rooted maximal
  $S$-walk
  $     \xymatrix@R=0.9em@C=1.7em{
  P: u_0 \ar@{->}[r]^{\quad  s_1} & u_1 \ar@{->}[r]^{s_2} &  \cdots \ar@{->}[r]^{s_r\quad } & u_r
}
$ of length $r$ stays in $\stab_S(R,U)_m$ if and only if
$r \leq m$.
\end{thm}%
%$$ --------------------------$$
%\begin{prop}\label{prop{a9.14}}
%For any $u \in \stab_{S}(R,U)$ which in not minimal in $\stab_{S}(R,U)$ with respect to
%$\leq_R$, there is a \textbf{maximal $S$-walk} $     \xymatrix@R=0.9em@C=1.7em{
%  P: \ u_0 \ar@{->}[r]^{\quad  s_1} & u_1 \ar@{->}[r]^{s_2} & u_2 \ar@{->}[r] & \cdots \ar@{->}[r]^{s_r\quad } & u_r = u
%}
%$     in $\stab_{S}(R,U)$, i.e., $P$ is an $S$-walk in $U$ (Definition \ref{def:a8.11})
%ending at $u$ with all bridges in $S$ and $u_0$ minimal in $\stab_{S}(R,U)$.
%\end{prop}
%
%\begin{proof}
%  Clear by Corollary \ref{cor:a8.12}.
%\end{proof}
%
%\begin{defn}\label{def:a9.15} Let $\stab_{S}(R,U)_0$ to be the set of minimal elements of $\stab_{S}(R,U)$. For $m \in \NN$ we define inductively the set
%$\stab_{S}(R,U)_{m+1} :=  \stab_{S}(R,U)_m + S$.
%\end{defn}
%Since $u' + S \subset U$ for every
%$u' \in \stab_{S}(R,U)$, it is evident that
%$\stab_{S}(R,U)_m \subset \stab_{S}(R,U)$ for any $n \in \NN$, and
%$$\stab_{S}(R,U)_m = \stab_{S}(R,U)_0 + Z_m(S),$$
%where $Z_m(S) = \{ n_1 s_1 + \cdots + n_p s_p \mid \sum_{1}^{p} n_i \leq m\} $.
%The following is then obvious
%\begin{prop}\label{prop{a9.16}}
%A maximal $S$-walk starting at $\stab_{S}(R,U)_0$ stays in $\stab_{S}(R,U)_m$ iff it has length
%$\leq m$.
%\end{prop}
%
%$$ -------------------- $$
We may also consider $S$-walks in $\stab_{S}(R,U)$ that  do not start at a minimal element.

\begin{defn}\label{def:a9.17} Let
$u_1, u_2 \in \stab_{S}(R,U)$ with $u_1 <_R u_2 $, and let their difference be
$z$, % (Definition ~\ref{defn:1.1}),
i.e., $u_2 = u_1 + z$ for a unique $z = \sum_{1}^{p} n_i s_i \in Z(S)$. We call $z$ the \textbf{stable distance} from $u_1$ to $u_2$. We  extend this  to the case $u_1 \leq u_2$ by setting $d^+(u_1, u_2) =0 $, if $u_1 = u_2$.
\end{defn}

\begin{remark}\label{rem:a9.18} If $u_1 \leq_R u_2 \leq _R u_3$ and $u_1 \in \stab_{S}(R,U) $, then also $u_2$ and $u_3$ belong to $\stab_{S}(R,U)$, and
$ d^+ (u_1 ,u_2) + d^+(u_2,u_3) = d^+(u_1, u_3)$.
\end{remark}

Stable distances can be defined more generally for elements of $U$ that are not necessarily in $\stab_{S}(R,U)$.
\begin{remark}\label{def:a9.19} Let
$u_1, u_2 \in U$ such that $u_1 <_R u_2$,  and suppose that $u_2 = u_1 + z $ for a (unique) $z \in R \sm U$.
Given   $w \in \stab_{S}(R,U)$, both  $u_1 +w$ and $u_2 + w$ belong to  $\stab_{S}(R,U)$. Hence,
$u_2 + w = (u_1 + w) + z$,  so $z = \sum_i n_i s_i \in Z(S) $ and
$d^+(u_1  + w, u_2 + w ) = z$. We define
$d^+(u_1 , u_2 ) = d^+(u_1  + w, u_2 + w )$.
\end{remark}
\noindent
\emph{Comment.} This is an element of $Z(S)$, independent of the choice of $w$. There may be no maximal $S$-walk from $u_1$ to $u_2$. However, maximal $S$-walks from $u_1+w$ to $u_2 + w$
and from $u_1+w'$ to $u_2 + w'$ with $w,w' \in \stab_{S}(R,U)$ can both be ``lifted'' to the same walk from  $u_1+w+ w'$ to $u_2 + w + w'$ by adding $w$ and $w'$ respectively.
We  discuss a natural problem concerning sums of elements of $\SG$ in $U$,
which we label MSU (``minimal $\SG$-sums in~ $U$'').
\pSkip
\emph{Problem MSU}. Which elements $u \in U$ can be expressed as sums of elements of $\SG$ (not necessarily distinct) such that every proper partial sum lies in $R \sm U $?

\begin{defn}\label{def:a9.20}
We call such an element $u$ a \textbf{minimal gap-sum in $U$} (for $S$);  more precisely, we say that~ $u$ is a \textbf{minimal $S$-gap-sum in $(R,U)$}.
\end{defn}

\begin{examp}\label{exmp:a9.20}
  Let $R = (\NN, + )$, $U = \{0 , d, d+1, d+2,  \dots \} $ for some $d \geq 2$, and $\SG = \{ 1 \} $. Then,
  $\underbrace{1 + \cdots + 1}_{d} \in U$, but $k d \notin U$ for $1 \leq k \leq d- 1$.
\end{examp}

%We slightly modify this example to point an intricacy in Definition \ref{def:a9.8} of minimality of sets $\SG \subset \Gap(R,U)$.
\begin{examp}\label{exmp:a9.21}
  Let $R = (\NN, + )$, $U = \{0 , d, d+k, d+k+1,  \dots \} $ for some
  $1 < k <d$, where $d \geq 3$, and  $\SG = \{ 1,k  \} $. Then, $\SG= \Gap(R,U)$ and the  sum
  $\underbrace{1 + \cdots + 1}_{d}+k$ is in $ U$, but the proper partial sums are not in ~$U$.
\end{examp}

These examples can be extended to the additive monoid $R = \NN^r / E$, where $E$ is a weight-compatible additive equivalence relation on $\NN^r$, as described in \S\ref{sec:6}
(Theorems \ref{thm:6.8} and \ref{thm:6.9}), and is trivial for all weights $< d$.
%We now allow $d \geq 2$ (instead of $d \geq 3$).
Let~ $T$ be  the set of standard generators
$\{t_1, \dots, t_r \} $ of~ $\NN^r$, and let $B_e = \iw(e) \sm T$ denote the associated big cell, cf.~ \eqref{eq:6.9}. \{Here we use the letter ``$T$'' instead of ``$\SG$'', to avoid confusion with the notation used  above.\}
Recall from~ \S\ref{sec:6} that the $E$-classes in $\iw(d)$ may form  an arbitrary
nontrivial partition of $\iw(d)$, i.e., not all $E$-classes in $\iw(d)$ are singletons.

\begin{examp}\label{exmp:a9.22}
In this setting, consider the sets  $\SG = \{ [t_i] \in T \ds |  1  \leq i \leq r\} $ and
$U = \{ x \in R \ds | w(x) \geq ~d\}  \cup~ \00$. Then, $U$ is an \as submonoid of $(R,+)$, and every $u \in U$ of weight $d$ can be expressed as a sum od elements of ~$\SG$ such that all proper partial sums lie in $R \sm U$. No other elements of (U) have this property.
\end{examp}

Here is a class of examples of a somewhat different kind, based on a monoid introduced in \S\ref{sec:1}.
\begin{examp}\label{exmp:a9.23} Let $R_d = \NN^2 / E_d$, as introduced in \S\ref{sec:1}, cf.  \eqref{eq:1.8}, and let $$ R^+_d = \{ [n_1 t_1 + n_2 t_2] \ds| n_1 > n_2\}, $$
cf. \eqref{eq:1.9}. By Proposition \ref{prop:1.5}, the \as submonoids of $R^+_d$ are precisely the submonoids of ~$\NN t$. If $U$ is such a submonoid and $U \neq \00$, then $U = Hx$, where $x$ is the element of minimal weight in $U \sm \00$.  In particular,
$U = Hx \cong H$, where $H$ is a submonoid of $\NN$. Then, there is
  a unique weight-compatible bijection $H \isoto Hx$. For various choices of $\SG$ in $Hx = U$,  the element $x$ is  minimal, and can be written as  a sum of $\SG$-gap elements whose proper partial sums are all lie in $R^+_d \sm U$.
\end{examp}
%\begin{proof}
 % This is immediate from the construction of $R$ and $U$ in \S\ref{sec:6}.\end{proof}

The submonoids $[U \Ng  s]_0$ constructed in \S\ref{sec:2} provide examples in which  the MSU-problem has a negative solution. Recall that
$\Idm(R) = \{ s \in R \ds | s + s = s\} $ is an additive subset of $(R,+)$, which, for any \as submonoid~ $U$ of $(R,+)$, gives rise to a new \as submonoid
$$ [U \Ng  s]_0 = [U \Ng  s] \cup \00 $$
containing $U$, cf. \eqref{eq:2.6}.
These submonoids form  an upward-directed system of \as submonoids of $R$, and their union $\tlU$ is again an \as submonoid of $(R,+)$ containing $U$. The set
$$ \SG = (\Idm(R)) \sm \00$$
is an additive subset of $R \sm \tlU$, since $s+ s = s + 0$ for every  $s \in \SG$. Indeed,  if~ $s$ belongs to the cancellative monoid~ $\tlU$,  this would imply  $s =0 $. Hence,  no $u \in \tlU$ can be expressed as a sum of elements of~ $\SG$. The same  holds for every monoid $[U \Ng  s]_0 \supset U$.

We return to minimal gap-sums in general.

\begin{thm}\label{thm:a9.25} Assume that $S$ is a finite subset of
$\Gap(R,U)$, and let $t \in S$ be a minimal gap-sum in $U$ for~ $R$. Then, there exist elements
$s_1, \dots, s_k \in S$ such that  $t= s_1 + \cdots + s_k \in U$, while none of the proper partial sum of
$s_1 + \cdots + s_k $ belongs to  $U$. We choose some $u_0 \in \stab_S(R,U)$.
\begin{enumerate}\ealph
  \item Any two $S$-walks  $     \xymatrix@R=0.9em@C=1.7em{
  P: u_0 \ar@{->}[r]^{\quad  s_1} & u_1 \ar@{->}[r]^{s_2}  & \cdots \ar@{->}[r]^{s_k\quad } & u_k
}
$
and
 $     \xymatrix@R=0.9em@C=1.7em{
  P': u_0 \ar@{->}[r]^{\quad  s'_1} & u_1 \ar@{->}[r]^{s'_2} &  \cdots \ar@{->}[r]^{s'_\ell \quad } & u_\ell
}
$
have the same length $k = \ell $ and the same gap-sum $t$, i.e.,
$(s'_1, \dots, s'_\ell)$ is a permutation of
$s_1, \dots, s_k$ and $u_k = u_\ell = u_0 +t$.
  \item For every $s_i$, where $1 \leq i \leq k =\ell$, there exists  a unique element
  $\check{s}_i \in Z(S)$ such that $s_i + \check{s}_i = t $.
\end{enumerate}
\end{thm}
\begin{proof}
This is immediate by Theorem \ref{thm:a4.8} and the previous definitions.
\end{proof}
We name $\check{s}_i$ the \textbf{complement} of $s_i$ in $t$ (for $S$); it  is independent of the choice of $u_0$ in $\stab_S(R,U)$.

Assume that  $\xymatrix@R=0.9em@C=1.7em{
  P: u_0 \ar@{->}[r]^{\quad  s_1} & u_1 \ar@{->}[r]^{s_2} & \cdots \ar@{->}[r]^{s_r\quad } & u_r
}
$
is a maximal gap walk in $\stab_S(R,U)$,  i.e., all bridges are element of $S$.  Assume also that $S$ is reduced. Among the $u_i$, we define inductively the\textbf{ blue nodes}~ $u_i$ of $P$ as follows.
\begin{defn}\label{def:a9.26} The
node $u_0$ is declared blue. If $u_k$ is blue, then the first node $u_\ell$ with $\ell > k$ and  $u_\ell \in U$  is blue as well.  In other terms, $u_0 \in U$, $u_k \in U$, and all $u_i$ with
$k < i < \ell $ do not belong to~ $U$.
\end{defn}
Thus, we have associated to $P$ a walk
 \begin{equation}\label{eq:9.2}
    \xymatrix@R=0.9em@C=1.7em{
  \sig(P): u_0 = u_{h(0)} \ar@{->}[r]^{\quad  t_1} & u_{h(1)} \ar@{->}[r]^{t_2} & u_{h(2)} \ar@{->}[r] & \cdots \ar@{->}[r]^{t_k\quad } & u_{h(k)},
}
 \end{equation}
where all $t_i$ are minimal gap-sums as studied in Theorem \ref{thm:a9.25}. Thus,
$$ u_{h(i)} = u_{h(i-1)} + t_i \qquad (1 \leq i \leq k)$$
$t_i \in Z(S) \cap U$, but no proper subsum of $t_i$ belongs to  $U$. We call
$\sig(P)$ the \textbf{blue walk of $P$}.
The following properties are immediate from   Definition \ref{def:a9.17}.

\begin{remark}\label{rem:a9.26}  $ $
\begin{enumerate}\ealph
  \item The length of $\sig(P)$ is the stable distance $d^+(u_0, u_{h(k)})$.
  \item $d^+(u_0, u_{h(k)}) = 0 $ if and only if $u_0$ is the only blue node of $P$.
  \item Either $u_r$ is blue, or $u_{h(k)}$ is the last node $u_s$ with $u_s \in U$ and $s < r$.
  \item Assume that  $u < _R v$ are nodes of $P$, and let $P[u,v]$ denote the part of $P$ starting at $u$ and ending at $v$. Then, in this notation,
      $$ \sig(P)[u_{h(k)}, u_{h(\ell)}] = \sig(P[u_{h(k)},u_{h(\ell)}]). $$
\end{enumerate}
\end{remark}

%\begin{defn}\label{def:a9.28}
%We say that two maximal $S$-walks $P$ and $P'$ in $\stab_S(R,U)$ are \textbf{confluent}, if they have at least two blue nodes in common.
%\end{defn}
%
%
%
%\begin{thm}\label{thm:a9.29}
%If $P$ and $P'$ are confluent, then $\sig(P)$ and $\sig(P')$  have same length, and we may write
%$     \xymatrix@R=0.9em@C=1.7em{
%  \sig(P): u_{h(P,0)} \ar@{->}[r]^{\quad  t_1} & u_{h(P,1)} \ar@{->}[r]^{t_2}  & \cdots \ar@{->}[r]^{t_k\quad } & u_{h(P,k)}
%}
%$ and \\
%$     \xymatrix@R=0.9em@C=1.7em{
%  \sig(P'): u_{h(P',0)} \ar@{->}[r]^{\quad  t'_1} & u_{h(P',1)} \ar@{->}[r]^{t'_2}  & \cdots \ar@{->}[r]^{t'_k\quad } & u_{h(P',k)}
%}
%$
%with
%$u_{h(P,i)} = u_{h(P',i)}$ and bridges $t_i = t'_i \in Z(S) \cap U$.
%The length $k$ is the stable distance $k = d^+(u_{h(P,0)},u_{h(P,k)})
%= d^+(u_{h(P',0)},u_{h(P',k)})
% $.
%\end{thm}
%\begin{proof}
%  This is evident from Theorem \ref{thm:a9.25}, if $\sig(P)$ or $\sig(P')$ has length one, and then follows by induction from Remark \ref{rem:a9.26}.(d) if $\sig(P)$ or $\sig(P')$ has any length $k$.
%\end{proof}
%
%
%\begin{defn}\label{def:a9.30}
%We name the part $[u_{h(P,0)}, u_{h(P,k)}]$ of $P$ the \textbf{confluence  interval} of $P$ and $P'$ on $P$.
%\end{defn}
%This is the subwalk of $P$ starting and ending with a blue node, which contains all common blue nodes
%of $P$ and $P'$.

Exploiting  Theorem \ref{thm:a9.25}(a), we obtain the following result.

\begin{cor}\label{cor:a9.31}
If $P$ starts and ends at distinct blue nodes, then every other such $S$-walks~ $P'$ with $\sig(P) = \sig(P')$ is obtained  from $P$ by arbitrary permuting the  bridges between each pair of two consecutive blue nodes $u_{h(P,s)}$ and $u_{h(P, s+1)}$.
\end{cor}

We express this fact in another way. As usual, let $\tS(n)$ denote the symmetric group of all  permutations of $\{ 1, \dots, n\}$. Assume that $P$ is a maximal $S$-walk in $\stab_S(R,U)$ that ends (and of course also starts) at a blue node of
$$     \xymatrix@R=0.9em@C=1.7em{
  P: u_0  = u_{h(P,0)} \ar@{->}[r]^{\qquad  \quad    s_1} & u_1 \ar@{->}[r]^{s_2}  & \cdots \ar@{->}[r]^{s_k\quad } & u_{h(P,k)}= u_r.
}
$$
For consecutive blue nodes
$$  u_{h(P, i-1)},  u_{h(P, i)}, \qquad i = 1,\dots, k, $$
the symmetric group
$$  \tS\big(h(P, i) -1 - h(P, i-1)\big) = \tS(d^+ ( u_{h(P, i-1)},  u_{h(P, i)})-1  )$$
acts naturally on the set  of nodes of $P$ strictly between these blue nodes. Hence, the
direct product of these symmetric groups, for  $(i = 1, \dots , k)$, acts on the set of nodes of $P$ that  are not blue. We extend this action to the set of all nodes of $P$ by fixing the blue nodes.

\begin{defn}\label{def:a9.30} We denote the image of this group in the symmetric group of all nodes of~ $P$ by $\Gm(P)$, and call it \textbf{the tolerance group of} $P$.
\end{defn}
%The orbit of $\Gm(P)$ is the set of all maximal $S$-walks $P'$ which have the same blue walk as~ $P$, $\sig(P) = \sig(P')$.

The following is now obvious.

\begin{prop}\label{prop:a9.31} The orbit
$\{ \gm(P) \ds  | \gm \in \Gm(P) \}$ is the set of all maximal $S$-walks $P'$ in $\stab_S(R,U)$ that start and end at blue nodes and
have the same blue walk as $P$, i.e., $\sig(P) = \sig(P')$.
\end{prop}

\begin{remark}\label{rem:a9.32}
The set of nodes of $P$ is totally ordered by the upper-bound ordering $\leq_R$. The group
$\Gm(P)$ fixes the blue nodes of $P$, but ignores the ordering between
consecutive blue nodes of $P$. If $u <_R v$ and $u$ or $v$ is blue, then  $\gm(u) <_R \gm(v)$ for every
$\gm \in \Gm(P)$.
\end{remark}

\section{Confluences, grids, and translates of blue $S$-walks}\label{sec:b9}
Assume  that $\sig(P)$ and $\sig(P')$ are   $S$-walks with blue ends in $\stab_S(R,U)$, as considered earlier, and that there is a sequence $A$ of consecutive blue nodes, contained in both  $\sig(P)$ and $\sig(P')$.  We say that
$\sig(P)$ and $\sig(P')$  are \textbf{confluent}, and that $A$ is a \textbf{confluence interval} of
$\sig(P)$ and $\sig(P')$.

\begin{remark}\label{rem:b9.3} $ $
\begin{enumerate}\ealph
  \item
If $B$ is a second  confluence interval of $\sig(P)$ and $\sig(P')$, and
$A \cap B \neq \emptyset$, then clearly
$A \cup B$ is a confluence interval of $\sig(P)$ and $\sig(P')$. Thus, every
confluence interval is contained in a unique \textbf{maximal confluence interval}.
It may well happen that $\sig(P)$ and $\sig(P')$  have several maximal confluence intervals,
which are then disjoint.
We illustrate such a scenario by the following diagram
   $$    \xymatrixrowsep{3mm}
\xymatrixcolsep{6mm}
    
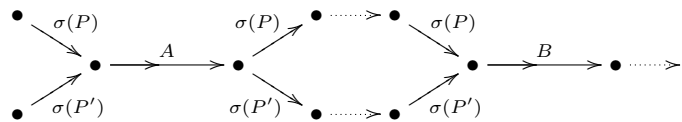
\xymatrix@R=0.8em@C=1.8em{
    \bullet  \ar@{->}[rd]^{\sig(P)} & && & \bullet \ar@{.>}[r] & \bullet \ar@{->}[rd]^{\sig(P)} \\
    & \bullet  \ar@{->}[r] \ar@{->}[rr]^{A} &   & \bullet \ar@{->}[ru]^{\sig(P)} \ar@{->}[rd]_{\sig(P')} & & & \bullet  \ar@{->}[r] \ar@{->}[rr]^B & & \bullet \ar@{.>}[r] & \\
    \bullet\ar@{->}[ru]_{\sig(P')} &&& & \bullet \ar@{.>}[r] & \bullet \ar@{->}[ru]_{\sig(P')}
     }
$$

  \item Consequently,  $\sig(P)$ and $\sig(P')$ are confluent if and only if these blue walks have at least two consecutive nodes in common, which then form a \textbf{minimal confluence interval} of  $\sig(P)$ and $\sig(P')$.
\end{enumerate}
\end{remark}

We   construct a ``rectangular grid of blue walks'', which  provides a large variety  of blue confluences.  We start with an obvious fact.

\begin{lem}[Translation Lemma]\label{lem:b9.4}
  Let $P$ be an  $S$-walk of length $k \geq 1 $ with blue end,  and let
  $t$ be a minimal gap-sum in $U$ for $S$. Then,
  $P '  = P +t $ is  an  $S$-walk  with blue end, and
  the walks walks from $u_{h(P,0)}$ to
$u_{h(P',k)}$
  in the  diagram
   $$    \xymatrixrowsep{3mm}
\xymatrixcolsep{6mm}
    \xymatrix@R=1.3em@C=1.5em{
     u_{h(P',0)} \ar@{->}[r] & \cdots \ar@{->}[r] &  u_{h(P',k)} \\
     u_{h(P,0)} \ar@{->}[r] \ar@{->}[u] & \cdots \ar@{->}[r] &  u_{h(P,k)} \ar@{->}[u]. \\
   }
$$ have the same gap-sum and the same length.
Thus, it represents two blue $S$-walks with the same gap-sum and length, cf.~\eqref{eq:9.2}.

%of blue $S$-walks, cf.~\eqref{eq:9.2}.
\end{lem}
%\begin{proof}
%  Clear by Theorem \ref{thm:a9.29}, since both $S$-walks from $ u_{h(P,0)}$ to
%   $u_{h(P',k)}$ have the same gap-sum.
%\end{proof}

\begin{construction}\label{cons:b9.3}
  Given two blue $S$-walks
$$
\xymatrixcolsep{6mm}
    \xymatrix@R=1.3em@C=1.5em{ \sig(Q):
     u_{0,0} \ar@{->}[r]^{\quad t_{1,0}} & u_{1,0} \ar@{->}[r]^{t_{1,0}} & \cdots \ar@{->}[r]^{t_{r,0}} &  u_{r,0}}, \qquad \xymatrix@R=1.3em@C=1.5em{
       \sig(P):   u_{0,0} \ar@{->}[r]^{\quad t_{0,1}} & u_{0,1} \ar@{->}[r]^{t_{0,2}} & \cdots \ar@{->}[r]^{t_{0,s}} &  u_{0,s},
     }
$$
that start at the same node $u_{0,0} \in \stab_S(R,U)$,    by iterative use of the translation lemma~ \ref{lem:b9.4} we obtain a rectangular \textbf{grid of blue $S$-walks} of size $r \times s$, which consists of meshes
   \begin{equation}\label{eq:b9.1}
   \xymatrixcolsep{6mm}
    \xymatrix@R=1.3em@C=1.5em{
     u_{i+1,j,} \ar@{->}[rr]^{t_{0,j+1}} &  &  u_{i+1,j+1} \\
     u_{i,j} \ar@{->}[rr]^{t_{0,j+1}} \ar@{->}[u]^{t_{i+1,0}} &  &  u_{i,j+1}  \ar@{->}[u]^{t_{i+1,0}}\\
     }
     \end{equation}
of blue $S$-walks of length one ($0 \leq i < r$, $0 \leq j < s$).
 \end{construction}

 Here, any two blue $S$-walks
from a fixed node $u_{i,j}$ to a fixed node $u_{k,\ell}$ with $i \leq k$, $j \leq \ell $, that run along horizontal and vertical edges, have the same gap-sum and length.
 We call this diagram a \textbf{blue grid $G$ of $S$-walks based on $\sig(P)$ and $\sig(Q)$}.

\begin{defn}\label{def:b9.4}
We say that a pair of $S$-walks $(\gm, \dl )$ in $G$ is a \textbf{pen}\footnote{A memory aid: think of $(\gm,\dl)$ as a fenced area in a meadow.}, if the following conditions  hold:
\begin{enumerate}\ealph
  \item $\gm $ and $\dl$ start and end at the same nodes, and thus have the same length and the same gap-sum,
  \item their length is $> 1$,
  \item $\gm$ and $\dl$ have no other blue nodes in common.
\end{enumerate}
\end{defn}
If this holds, then, perhaps after  interchanging, $\gm$ and $\dl$, the first bridge of
$\gm$ goes east (i.e., is horizontal), while the first bridge of $\dl$
goes north. We call $\gm$ the \textbf{lower part} and  $\dl$ the \textbf{upper part} of the pen.
The following is now obvious.

\begin{prop}\label{prop:b9.4}
Any pair of blue walks in $G$ with the same start and end consists of a succession of maximal confluence intervals and pens.
\end{prop}

We illustrate  such a scenario by the  diagram
$$
   
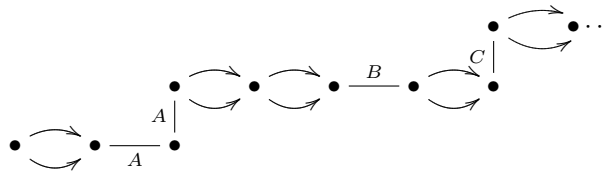
\xymatrix@R=1.2em@C=1.9em{
   &&&&&& \bullet \ar@/_0.7pc/[r]<0ex>
   \ar@/^0.5pc/[r]<0ex> &  \bullet \cdots
   \\
      &&   \bullet \ar@/_0.7pc/[r]<0ex>
   \ar@/^0.5pc/[r]<0ex> & \bullet \ar@/_0.7pc/[r]<0ex>
   \ar@/^0.5pc/[r]<0ex> &  \bullet \ar@{-}[r]^B & \bullet \ar@/_0.7pc/[r]<0ex>
   \ar@/^0.5pc/[r]<0ex> &  \bullet \ar@{-}[u]^C
   \\
   \bullet \ar@/_0.7pc/[r]<0ex>
   \ar@/^0.5pc/[r]<0ex> &  \bullet \ar@{-}[r]_A &  \bullet \ar@{-}[u]^A
}
$$
without going into details about the pens.
The pens may have very different shapes. We illustrate one  such pen
      \begin{equation}\label{eq:b9.2}    \xymatrixrowsep{3mm}
\xymatrixcolsep{6mm}
    
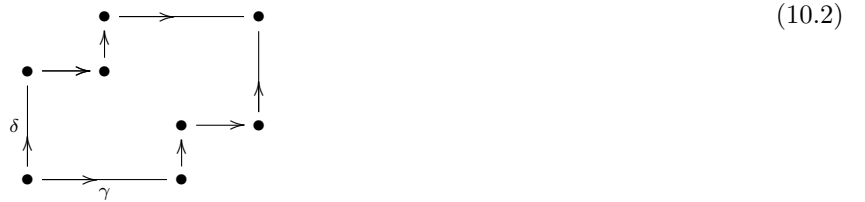
\xymatrix@R=1em@C=1.8em{
    &\bullet \ar@{-}[rr] \ar@{->}[r] & & \bullet &&
    \\
    \bullet \ar@{->}[r] \ar@{->}[r] &\bullet  \ar@{->}[u] &&&&
        \\ && \bullet \ar@{->}[r] & \bullet \ar@{->}[u] \ar@{-}[uu] &
        \\
    \bullet \ar@{-}[rr]_{\gm} \ar@{->}[r] \ar@{->}[u] \ar@{-}[uu]^{\dl} &  &\bullet   \ar@{->}[u] && &
     }
\end{equation}

Assume that $G$ is a blue grid of size $r \times s$ ($r \geq 1$, $s \geq 1$),
based on $\sig(P)$ and $\sig(Q)$.
We look for pens in $G$ that are as
``narrow'' as possible. We rely on the following fact.

\begin{remark}\label{rem:b9.6} The grid
$G$ can be identified with the set of its nodes, equipped with  the partial ordering
$\leq_R$, restricted to $G$, on a set of
$(r+1) (s+1)$ nodes.
Then, an $S$-walk $\gm$ in $G$ is a totally ordered subset of $G$ in which all bridges have length one. Explicitly, if $u \in \gm$, then the next node east or north of $u$ is $u +t$ for  some minimal gap-sum $t$.
\end{remark}
The following is now obvious.
\begin{prop}\label{prop:b9.7}
Let  $\gm'$ be a walk in $G$ from $u_{k,0}$ to $u_{\ell,s}$, with
$0 \leq k \leq \ell < s$,  starting with a bridge east.
\begin{enumerate} \ealph
  \item The nodes one step north to  $\gm'$ form  a walk $\dl'$ from
  $u_{k+1,0}$ to $u_{\ell+1,s}$.
  \item The walks $\gm = \gm' \cup \{u_{\ell+1,s}\} $ and
  $\dl =\{ u_{k,0}\} \cup \dl' $ form a pen $(\gm, \dl)$, where all bridges have length one.
\end{enumerate}
\end{prop}

\begin{defn}\label{def:b9.8} We call $(\gm, \dl)$ the \textbf{narrow channel} with \textbf{bottom} $\gm'$ and \textbf{ceiling} $\dl'$, and refer to these pens as
\textbf{narrow channels above} $\sig(P)$.
\end{defn}
Here is an illustration  such a channel
   \begin{equation}\label{eq:b9.3}
  \xymatrixrowsep{3mm}
\xymatrixcolsep{6mm}
    
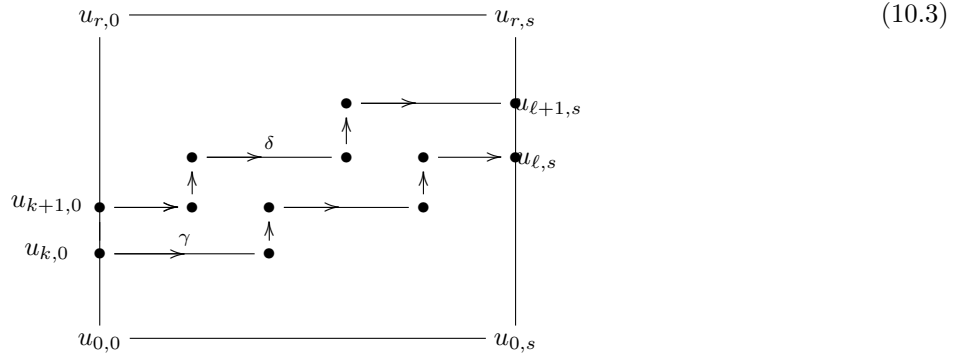
\xymatrix@R=0.5em@C=0.1em{
     \\ \\
    \\
    \\ \\  \\
  {u_{k+1,0}}    \\
  {u_{k,0}}
     \\
     \\
    & } \hskip -1.5em
         \xymatrix@R=0.7em@C=1.8em{
    u_{r,0} \ar@{-}[rrrrr]  && &&& u_{r,s}
    \\
    \\
   &&& \bullet \ar@{->}[r] \ar@{-}[rr] & & {\bullet} &  \hskip -14mm {u_{\ell+1,s}}
    \\
    &\bullet \ar@{-}[rr]^{\dl} \ar@{->}[r] & & \bullet \ar@{->}[u]  & \bullet \ar@{->}[r] &  {\bullet} &   \hskip -17mm {u_{\ell,s}}
    \\
   {\bullet} \ar@{->}[r] \ar@{->}[r] &\bullet  \ar@{->}[u] & \bullet \ar@{->}[r]
     \ar@{-}[rr]& & \bullet \ar@{->}[u]  &
        \\
 {\bullet}  \ar@{-}[rr]^{\gm} \ar@{->}[r] \ar@{-}[u] &  &\bullet   \ar@{->}[u] && &
     \\
     \\
     u_{0,0} \ar@{-}[uuuuuuu] \ar@{-}[rrrrr] && &&& \ar@{-}[uuuuuuu] u_{0,s}
     }
   \end{equation}
We introduce additional terminology for walks in the blue grid $G$. This will help us analyze  pens.

\begin{defn}\label{def:b9.9} We say that a walk  $\gm$ in $G$ is \textbf{transient} (in $G$), if $\gm$ starts at  a node ~$u_{k,0}$ on the vertical line $\sig(Q)$
and ends with  a node $u_{\ell,s}$ at the last vertical line of  $G$ (i.e., the vertical line starting at  $u_{0,s}$).
\end{defn}
For example, the walks $\gm'$ and $\dl'$ in the above diagram are transient. The walks $\gm$ and $\dl$, which form the lower and upper parts of the depicted pen, are also transient.  We denote the set of transient walks in $G$ by $\bTrns(G)$.

%In consequence of Remark \ref{rem:b9.6} we observe the following.
\begin{remark}\label{rem:b9.10} It follows by  Remark \ref{rem:b9.6} that, if $\gm$  is a transient walk in $G$, then the set $\gm \cap L$ of nodes of $\gm$ on any vertical line $L$ is an \textbf{interval} on $L$, i.e., a convex subset of the finite totally ordered set~ $L$.
\end{remark}

\begin{defn}\label{def:b9.11}
$ $
\begin{enumerate} \ealph
  \item  Given blue walks $\gm_1, \gm_2 \in \Trns(G)$, we say that
  $\gm_2$ \textbf{dominates} $\gm_1$, and write $\gm_1 \leq  \gm_2$, if, on no vertical line $L$ of $G$, a node of $\gm_1 \cap L $ is strictly above a node of $\gm_2\cap L$. In other words,  the interval $\gm_1 \cap L$ lies below the interval
  $\gm_2 \cap L$ except that the last node of $\gm_1$ on $L$ may coincide with the first node of $\gm_2$ on ~$L$.
  \item We say that
  $\gm_2$ \textbf{strictly dominates} $\gm_1$, and write $\gm_1 <  \gm_2$, if,
  for every vertical line~ $L$, the nodes of  $\gm_1$ on $L$ lie below the nodes of $\gm_2$ on $L$.
\end{enumerate}
\end{defn}

It follows then that
\begin{align*}
  \gm_1 \leq \gm_2, \gm_2\leq \gm_3 & \Rightarrow   \gm_1 \leq \gm_3, &
  \gm_1 \leq \gm_2, \gm_2 <  \gm_3 & \Rightarrow  \gm_1 < \gm_3, \\
    \gm_1 < \gm_2, \gm_2\leq \gm_3 & \Rightarrow  \gm_1 < \gm_3, &
  \gm_1 \leq \gm_2, \gm_2\leq \gm_1 & \Leftrightarrow  \gm_1 = \gm_2.
\end{align*}
The next proposition is evident from the definitions and the nature of blue grids.

\begin{prop}\label{prop:b9.13} Let $\gm_1$, $\gm_2$, and  $\dl$ be transient walks  in $G$ with $\gm_1 < \dl$ and $\gm_2 < \dl$.
\begin{enumerate}\ealph
  \item There  exists a unique $\eta \in \Trns(G)$, denoted as
  $\eta = \gm_1 \vee \gm_2$, such that $\eta \leq \zeta$ for any $\zeta < \dl$ satisfying
  $\gm_1 \leq \zeta$ and $\gm_2 \leq \zeta$.
  \item The walk $\eta $ is dominated by the bottom $\eps'$ of the narrow channel $(\eps,\dl)$ with ceiling $\dl'$, and hence  $\eta \leq \eps < \dl'$.
\end{enumerate}
\end{prop}

Assume that $A$ is a nonempty set of the blue grid $G$, viewed as a partially ordered subset of $G$ under the restriction of $\leq_R$ to $G$.

\begin{defn}\label{def:b9.14}
$ $\begin{enumerate} \ealph
     \item The \textbf{core} $\crA$ of $A$ is the union of all meshes \eqref{eq:b9.1} contained in $A$, in other words the union of all sets
         $\{u_{i,j}, u_{i+11,j}, u_{i,j+1}, u_{i+1 ,j+1} \} \subset A$.
     \item
      The \textbf{NE-completion} $\olA$ of $A$ (short for ``north east completion'') is the smallest set $B \subset G$ such that  $B \supset A$ and $B = \crB$, which
      consists of all nodes that can be reached  from a node $a \in A $ by a blue walk of length $1$ or a ``mixed''  blue walk of length $2$ (i.e., north-east, or east-north).

   \end{enumerate}

\end{defn}

It may happen that $B = \emptyset$; however, if $B \neq \emptyset$, then $B$ is clearly unique. %We give some examples.

\begin{examples}
  \label{exmp:b9.15} $ $
  \begin{enumerate} \ealph
    \item Let $A = \sig(P) \cup \sig(Q)$. Then, $\crA = \emptyset$, and
    $\olA$ is the union of the two rectangles
    $\{u_{0,0}, u_{r,0}, u_{0,1}, u_{r,1} \}$
and
    $\{u_{0,0}, u_{1,0}, u_{0,s}, u_{1, s } \}$.

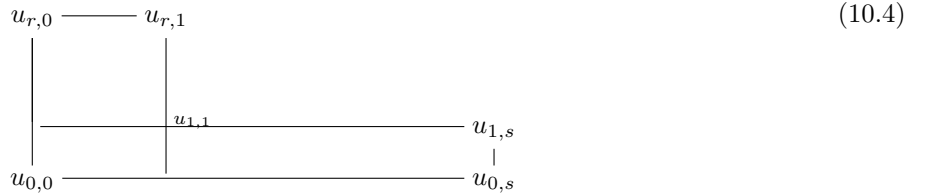
   \begin{equation}\label{eq:b9.4}
  \xymatrixrowsep{3mm}
\xymatrixcolsep{6mm}
         \xymatrix@R=0.5em@C=1.8em{
    u_{r,0} \ar@{-}[r]  & u_{r,1} & &&&
    \\
    \\
     \\
     \ar@{-}[uuu] \ar@{-}[rrrrr] & ^{ \qquad u_{1,1}} & &&&  u_{1,s}
     \\
     u_{0,0} \ar@{-}[uuuu] \ar@{-}[rrrrr] &  \ar@{-}[uuuu] & &&& \ar@{-}[u] u_{0,s}
     }
   \end{equation}
For every node $u \neq u_{1,1}$ in this union there is a unique shortest walk from $A$ to
$u$ of length $1$, while  for $u = u_{1,1}$ there are two shortest walks from $A $ to
$B$, both of length $1$.
\item If the set $A$ is a singleton  $\{ u_{i,j}\} $ with $0 \leq i <  r$ and  $0 \leq j < s$, then
    $M = \{u_{i,j}, u_{i+1,j}, u_{i,j+1}, u_{i+1, j+1 } \}$. The same holds, if $A$ is a proper subset of $M$ containing ~$u_{i,j}$. However, if $u_{i,j} \notin A$, then $\olA = \emptyset$.

    \item Assume that $A$ is is the set of nodes of a walk $\gm' $  from $u_{k,0}$ to $u_{\ell, s}$ with
    $0 \leq k \leq \ell < s$. Then, $\olA$ is the narrow channel $(\gm,\dl)$ with bottom
    $\gm'$ and ceiling $\dl'$ above $\sig(P)$, cf. Definition~\ref{def:b9.8}, as depicted in Diagram \eqref{eq:b9.3}.
  \end{enumerate}
\end{examples}

\begin{prop}\label{prop:b9.16}
Let $(A_\lm \mid \lm \in \Lm)$ be a family of subsets of $G$ such that,
for each $\lm \in \Lm$,  the NE-completion~ $\olA_\lm$ exists.
Then, $\bigcup _{\lm \in \Lm } A_\lm$ has the NE-completion
$$\overline{\bigcup _{\lm \in \Lm } A_\lm} = \bigcup _{\lm \in \Lm } \olA_\lm.$$

\end{prop}

\begin{proof}
  $B=  \bigcup _{\lm \in \Lm } \olA_\lm$ is a union of the meshes of $G$. If $x \in B$, then
  $x \in \olA_\lm$ for some $\lm \in \Lm$. Hence, there exists a node $a \in A_\lm \subset B$
  with a blue walk of length $1$ or a mixed blue walk of length~ $2$ from $a$ to $x$.
\end{proof}
We call the subwalks of length $1$ of a blue walk $\gm$
the \textbf{basic subwalks} of $\gm$.  Subwalks of $A = \crA$ are mostly ``thin'' with respect to NE-completion.

\begin{prop}\label{prop:b9.17}
Let   $\gm$ be a union of walks in the blue grid $G$ of size $ r\times s $ (cf.  Construction~\ref{cons:b9.3}), and let ~$A$ be  a subset of $G$ with the following properties.
\begin{enumerate}\ealph
  \item $A$ is a union of  meshes of $G$, $A = \crA$.
  \item For every horizontal basic subwalk $\dl$ of $\gm$ in $A$, there exists a mesh
  $M \subset A$ whose north horizonal border is $\dl$.
  \item For every vertical basic subwalk $\eps$ in $A$, there exists a mesh $N \subset A$
      whose east vertical border is $\eps$.
\end{enumerate}
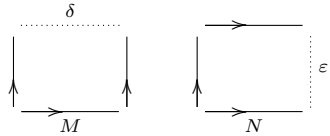
$$
   \xymatrixcolsep{6mm}
    \xymatrix@R=1em@C=1.5em{
      \ar@{..}[rr]^{\dl} &  &  \\
 && \\
      \ar@{-}[rr]_{M} \ar@{->}[r] \ar@{->}[u]  \ar@{-}[uu]&  &   \ar@{->}[u]
      \ar@{-}[uu] \\
     }
     \qquad
       \xymatrix@R=1em@C=1.5em{
      \ar@{-}[rr] \ar@{->}[r] &  &  \\
 && \\
      \ar@{-}[rr]_{N} \ar@{->}[r] \ar@{->}[u]  \ar@{-}[uu]&  &
      \ar@{..}[uu]_\eps \\
     }
$$

Then, $A\sm \gm$ has the NE-completion $(A \sm \gm)^- = A$.\footnote{A memory aid: think  of  $A$ as  a union of carpets $A_\lm$ with tears $A_\lm \sm (A \cap \gm) = A_\lm \sm \gm$. }

\end{prop}
\begin{proof}
  Every node of $\gm$ in $A$ can be reached from a node in in ~$A \sm \gm$ by a walk of length $1$ north or east direction.
\end{proof}

%Using Proposition \ref{prop:b9.17} we can repair a family of  tears in the union of carpets
%$A$ (see footnote~4) which roughly are apart by at least two nodes horizontally and vertically.

%We give two examples.

\begin{example}Let $0 < c < r$, $0 < d < s$, $\gm_1 = [u_{c,0}, u_{c,s}]$
$\gm_2 = [u_{0,d}, u_{r,d}]$, $\gm_3 = [u_{r,0}, u_{r,s}]$, and $\gm_4 = [u_{0,s}, u_{r,s}]$.
Then, $G \sm (\gm_1 \cup \gm_2 \cup \gm_3 \cup \gm_4)$ has the NE-completion $A = G$.

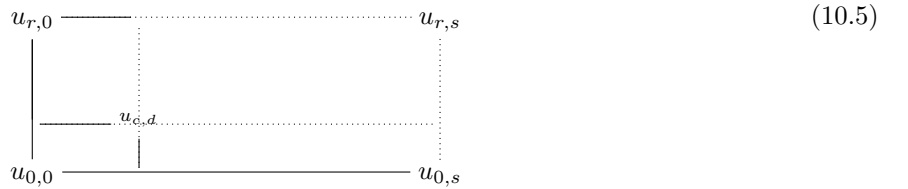
  \begin{equation}\label{eq:b9.4}
  \xymatrixrowsep{3mm}
\xymatrixcolsep{6mm}
         \xymatrix@R=0.5em@C=1.8em{
    u_{r,0} \ar@{.}[rrrrr] \ar@{-}[r]  & & &&&  u_{r,s}
    \\
    \\
     \\
     \ar@{-}[uuu] \ar@{.}[rrrrr] \ar@{-}[r] & ^{  u_{c,d}} & &&&
     \\
     u_{0,0} \ar@{-}[uuuu]  \ar@{-}[rrrrr] &  \ar@{.}[uuuu]  \ar@{-}[u]& &&& \ar@{.}[uuuu] u_{0,s}
     }
   \end{equation}

\end{example}
\begin{example}
  Let $A = (\gm_0,\gm_1) \cup \cdots \cup (\gm_{d-1},\gm_d)$ with narrow channels
  $(\gm_j,\gm_{j+1})$ above $\sig(P)$. Choose a set
  $J \subset \{ 1, \dots,d \}$ that does \textbf{not} contain any consecutive numbers
  $k, k+1$. By the construction of narrow channels (Definition \ref{def:b9.8}),
   $A$ is an  NE-completion of
   its subset $A \sm \bigcup_{j \in J} \gm_j$.
   \end{example}

By using NE-completions, we can also repair ``holes'' in a blue carpet.
% in a symmetric way.
\begin{construction}\label{const:b9.19}
  Let $G$ be a blue grid of size $r \times s$. Choose numbers $c,d, k,\ell$ such that
  $0 < c < c + k \leq r$ and  $0 < d< d +\ell \leq s$, to obtain a rectangular $k \times \ell $ blue subgrid $B \subset G$   with corners
  $u_{c,d}$, $u_{c+k,d}$, $u_{c,d+\ell}$, and $u_{c+k,d+\ell }$. We call  $A = G \sm B$ a \textbf{carpet in $G$
 with hole $B$}.
 Let $B'$ be the $(k-1) \times (\ell -1)$ subrectangle of $B$ with corners
 $u_{c+1,d+1}$, $u_{c+k,d+1}$, $u_{c+1,d+\ell}$, and $u_{c+k,d+\ell }$.  Then, $A = \crA$ has the NE-completion $\olA = G \sm B'$, except when  $k =1$ or $\ell =1$, in which case $\olA = G$, cf. the diagram

    \begin{equation}\label{eq:b9.5}
  \xymatrixrowsep{3mm}
\xymatrixcolsep{6mm}
         
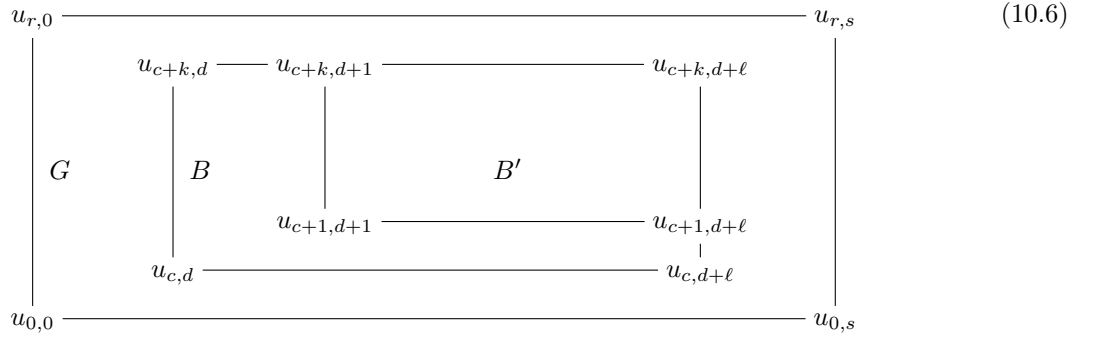
\xymatrix@R=0.5em@C=1.8em{
    u_{r,0} \ar@{-}[rrrrrr]  && &&& & u_{r,s}
    \\
     &    u_{c+k,d} \ar@{-}[r] & u_{c+k,d+1} \ar@{-}[rrr]  && &u_{c+k,d+\ell}
    \\
    \\
    \\ \qquad  G  & \qquad B & & \qquad B'
        \\
    &  &  u_{c+1,d+1} \ar@{-}[rrr] \ar@{-}[uuuu]  & && u_{c+1,d+\ell} \ar@{-}[uuuu]
     \\
    &    u_{c,d} \ar@{-}[rrrr] \ar@{-}[uuuuu]  && && u_{c,d+\ell} \ar@{-}[u]
     \\
     u_{0,0} \ar@{-}[uuuuuuu] \ar@{-}[rrrrrr] && &&& & \ar@{-}[uuuuuuu] u_{0,s}
     }
   \end{equation}

   In detail, for a node $x$ in the interval
    $[u_{c+1,d+1}, u_{c+k,d+1}]$
there exists  a unique node $a_1$ in
$[u_{c,d}, u_{c+k,d}]$ with horizontal basic walk from $a_1$ to $x$, and  for $y$ in
$[u_{c+1,d+1}, u_{c+1,d+\ell}]$
there exists  a unique vertical basic walk from a node  $a_2$ in  $[u_{c,d}, u_{c+1,d+\ell}]$
to $y$. Iterating this completion procedure, the hole vanishes after $k$ steps if $k \leq \ell $, and after $\ell$ steps if $\ell \leq k$. We  have ``repaired'' the carpet $A$ with  hole $B$.
\end{construction}

\section{Replacing $\SG$ by a larger finite subset $T$ of $\Gap(R,U)$}\label{sec:b10}
Let $(R,U)$ be an additive submonoid  $(R,+)$ with an \as\ submonoid $U$, Assume that every $ u \in U$ has finite gap height in $(R,U)$, where $U = \FGH(R,U)$. Given finite subsets $S$ and $T$ of $\Gap(R,U)$, we look for relations between the blue $S$-walks and blue $T$-walks in
 $\stab_S(R,U)$ and $\stab_T(R,U)$. We focus first on the case  $S \subset T$, in which
 $\stab_T(R,U) \subset \stab_S(R,U)$, cf. Remark \ref{rem:a9.12}. Any blue $T$-walk $\gm$ can be shifted  to a blue $S$-walk $\gm + w$ in $\stab_T(R,U)$ by adding a suitable  element $w \in U$ to all nodes.
 Two such walks $\gm + w_1$, $\gm + w_2$ can be shifted  to the same blue walk
 $\gm + w_1 + w _2 $ by adding $w_2$ and $w_1$, respectively, cf. the comment following Definition~ \ref{def:a9.19}.

 We first analyze what happens to a blue walk $\gm : \xymatrix@R=0.9em@C=1.1em{
  u_{0} \ar@{->}[r]^{s} & u_1
}$ of length $1$; that is, $\gm$ is  a set of two nodes $u_0$ and $u_1$ with
$u_0 + s = u_1$, where  $s$ a minimal element of $Z(S) \cap U$. Then,
$s\in Z(T) \cap  U$; however, often $s$ is not minimal in $Z(T) \cap U$. We choose a decomposition
$s = t_1 + \cdots + t_r$ with minimal elements in $Z(T) \cap U$, and obtain a blue walk
$
\xymatrix@R=0.9em@C=1.7em{u_0 \ar@{->}[r]^{\quad  t_1} & v_1 \ar@{->}[r]^{t_2} & \cdots \ar@{->}[r]^{t_r\quad } & v_r = u_1
}
$
which, after addition of a suitable $w \in U$, remains in $\stab_T(R,U)$
(e.g., $w = v_1 + \cdots + v_r$).

Although the walk $ \xymatrix@R=0.9em@C=1.1em{
  u_{0} \ar@{->}[r]^{s} & u_1
}$  is  simple, the new walk
$$ \gm':
\xymatrix@R=0.9em@C=1.7em{
(u_0 \ar@{->}[r]^{\quad  t_1} & v_1 \ar@{->}[r]^{t_2} & v_2 \ar@{->}[r] & \cdots \ar@{->}[r]^{t_r\quad } & v_r = u_k) + w
}
$$
can be more intricate, and there are  several choices for $\gm'$. This procedure for  producing  $T$-walks from
 $ \xymatrix@R=0.9em@C=1.1em{
  u_{0} \ar@{->}[r]^{s} & u_1
}$ readily generalizes to any blue $S$-walk
$
\xymatrix@R=0.9em@C=1.7em{
u_0 \ar@{->}[r]^{\quad  s_1} & u_1 \ar@{->}[r]^{s_2} &  \cdots \ar@{->}[r]^{s_m\quad } &  u_m
}
$
by trivially splicing the walks
 $ \xymatrix@R=0.9em@C=1.1em{
  u_{0} \ar@{->}[r]^{s_1} & u_1
}
, \dots,$ $
  \xymatrix@R=0.9em@C=1.1em{
  u_{m-1} \ar@{->}[r]^{s_m} & u_m
}$.
Working modulo the addition of constant elements $w \in U$, we may ask whether a given blue $T$-walk can be obtained from a blue $S$-walk of much shorter length, say of length $1$.

Assume that $S \subset T \subset \Gap(R,U)$. If $\gm$ is a minimal blue $S$-walk in
$\stab_S(R,U)$, we may regard $\gm$ as  a blue $T$-walk $\gm_T$; however, $\gm_T$
may have decomposable  bridges.

\begin{defn} Given a blue $T$-walk $\dl$ in $\stab_T(R,T)$, we say that $\dl$ is a \textbf{derivate} of $\gm$, and equivalently, that $\gm$ is a \textbf{nurse} of
$\dl$, if there exists  $w \in U$ with
$\dl= \gm_T + w$.
\end{defn}
 In this terminology, the walk $\gm'$ above is a derivate of the walk  $ \xymatrix@R=0.9em@C=1.1em{
  u_{0} \ar@{->}[r]^{s} & u_1
}$.
The Translation Lemma~\ref{lem:b9.4}, which is central to our constructions in this section, generalizes as follows.
\begin{prop}\label{prop::b10.2}
Let
$
\xymatrix@R=0.9em@C=1.7em{
\gm: u_0 \ar@{->}[r]^{\quad  t_1} & u_1 \ar@{->}[r]^{t_2} &  \cdots \ar@{->}[r]^{t_r\quad } &  u_r
}
$
be a blue $S$-walk
 in $(R,U)$, and let $z \in Z(S) \cap U$. The translation
$$\gm + z:
\xymatrix@R=0.9em@C=1.7em{
u_0 +z \ar@{->}[r]^{\quad  t_1} & u_1 +z  \ar@{->}[r]^{t_2} & u_2 +z \ar@{->}[r] & \cdots \ar@{->}[r]^{t_r\quad } &  u_r +z,
}
$$
of $\gm$ by $z$
 is again a blue $S$-walk
 in $(R,U)$.
\end{prop}
\begin{proof}
  Note that
  $u_i = u_{i-1} + t_i$, for  $1 \leq i \leq r$, and hence $u_i + z = u_{i-1} + z + t_i$.
\end{proof}

\begin{remark}\label{rem:b10.3}\footnote{We identify a blue $S$-walk in $(R,U)$ with the totally ordered set of its nodes in $U$ with respect to the partial ordering $\leq_R$, now also when $\gm$ is not maximal, cf. Definition \ref{def:a4.3} and \S\ref{sec:a9}.}  $ $
\begin{enumerate} \ealph
  \item $(\gm+ z_1) + z_2 = \gm + (z_1 + z_2)$ for any $z_1,z_2 \in Z(S) \cap U$.
  \item If $\gm$ is a maximal blue $S$-walk (Definition \ref{def:a9.4}), then $\gm+z$ is maximal and has the same length and  gaps as $\gm$.
  \item If $z \in \stab_S(R,U)$, then
  $\gm + z \subset \stab_S(R,U)$
  for any $\gm \subset \stab_S(R,U)$.

  \item
  If $\gm \subset \stab_S(R,U)$, then $\gm + z \in \stab_S(R,U)$ for any $z \in Z(S)\cap U$.

  \item If $\dl$  is a derivate of $\gm$, then $\dl + z $ is a derivate of $\gm + z$
  for any $z \in Z(S)\cap U$.
\end{enumerate}
\end{remark}

\section{$S$-walks with nodes outside the stability region}\label{sec:b10}

We start with an additive \as monoid $(R,U)$, where $U$ has the upper-bound partial  ordering $\leq_ U$, which we often write simply as  $\leq$.
Recall that the monoid $(U,+)$ is cancellative, cf. Definition \ref{defn:1.1}.
A finite set $S \subset \Gap(R,U)$ will be specified only later.

Given $z \in \Gap(R,U)$, we define the sets
$$ M (z) := (U:z)  = \{ u \in U \mid z+u \in U\}
\dss{\text{and}}
\tlM (z) := M(z) + z = (U: z) + z \subset U. $$
These sets are upsets  in the monoid $(U,+) $ with respect to the ordering~ $\leq_U$, and thus  are unions of cosets of $U$ in the semigroup
$(U \sm\00, + )$.

\begin{lem}\label{lem:b10.1}  $ $

\begin{enumerate}\ealph
  \item If $x \in U$, $w \in U$, and $w \leq x$, then $(x:w) + w = x$.
  \item If $V \subseteq U$ is nonempty, $w \in U$, and $w \leq V$, then $(V:w) + w = V$.
\end{enumerate}

\end{lem}
\begin{proof}
  (a): Clear, since $(x:w ) = (\{ x\} : w) $ is the set of all $y \in U$ with $x = w+y$.
  \pSkip
  (b): $(V:w )= \bigcup_{x \in V}(x:w)$, whence $(V:w) + w = \bigcup_{x \in V} (x:w) + w =
  \bigcup_{x \in V} \{ x \} = V$.
\end{proof}

\begin{thm}\label{thm:b10.2} Assume that $z \in \Gap(R,U)$, $w \in U$, and $w \leq (U:z)$. Then,
$M(z+w) + w = M(z)$.

\end{thm}
\begin{proof}  Lemma \ref{lem:b10.1}  implies  that $((U:z):w) +w = (U:z)$. Observe then  that
$$(U:z+w) = \{ x \in U \mid x+w \in (U:z)\}
= \{ x \in U \mid x+w+z \in U\} = (U:z)+w.$$ \vskip -20pt
\end{proof}

It follows by Lemma \ref{lem:a4.7} that the map
$$ \xymatrix@R=0.9em@C=1.7em{
M(z) \ar@{->}[r]^{  +z } & \tlM(z)}, \qquad u \mapsto u +z = \tlu,
$$
is injective. Note that $u < _R \tlu$;  however,  $u$ and $\tlu$ are incomparable  in the ordering $\leq_U$.
\begin{cor}\label{cor:b10.3}
If $w \leq (U:z)$, then we have a commuting triangle of injective maps between upsets in $(U,+)$
$$ \xymatrix@R=0.9em@C=1.7em{
M(z) \ar@{->}[rd]^{  +z \quad } &
\\ & \quad \tlM(z) = \tlM(z+w)\\
M(z+w) \ar@{->}[ru]_{  +(z+w) }  \ar@{->}[uu]_{  +w } . &  }$$

\end{cor}
\begin{proof}
This is evident, since
$x+(z+w) = (x+w) +z$
for any $x \in U$.
\end{proof}

Assume that $S$ is a reduced finite subset of $\Gap(R,U)$, and define
$$\Gap_S(R,U) := \Gap(R,U) \cap Z(S).$$
As in \S\ref{sec:b10}, we further assume that every $u\in U$ has a finite gap height in ~
$(R,+)$. This is a mild restriction that arises,  in particular,  if we replace $R$ by its submonoid $\FGH(R,U)$ consisting of the elements of finite gap height in $R$, cf. \S\ref{sec:a8}.
\begin{remark}\label{rem:b10.4} $ $
\begin{enumerate} \ealph
  \item If $x \in R$ and $U+x \subset U$, then  $x \in U$, since $0 \in U$.
  \item If  $W$ is a coset of $U$ in $U \sm \00$, where $W = w +U$, then  $w$ is the unique minimal element of $W$ with respect to both $\leq_U$ and  $\leq_R$.
   \end{enumerate}
\end{remark}

   If $z \in R$, $w \in U$, and $w+z \in U$, then $z \in \Gap_S(R,U)$. If such an element $z$ exists, we say that  $W$ \textbf{has gaps}. Then, there exists a smallest such gap $z_1$, and $z_1 \in S$.
These gaps in $W$ arise from the maximal blue walks
$\xymatrix@R=0.9em@C=1.7em{
w_0 \ar@{->}[r]^{  z_1} &
w_1}
$
of height $1$ in $W$.

We  introduce a ``lowering'' of an $S$-walk
$
\xymatrix@R=0.9em@C=1.7em{\gm:
u_0 \ar@{->}[r]^{\quad  z_1} & u_1 \ar@{->}[r]^{z_2} &  \cdots \ar@{->}[r]^{z_r\quad } & u_r
}
$
of length $r$, with $z_i \in Z(S)$, by a suitable element $w \in W$. We assume that $\gm$ runs in an upset~ $M$ of $U$ (with respect to $\leq_U$). Our main interest is in the set $M = \stab_S(R,U)$; nevertheless, the construction works for any upset $M$.

\begin{prop}\label{prop:b10.5}
  Assume that $w \in U$ and  $w \leq u_0$, and hence there  is a unique $p_0 \in U$ such that $p_0 +w = u$, since $U$ is cancellative  and subtractive in $R$. Assume further that
  \begin{equation}\label{eq:b10.1}
    p_0+ z_1 + \cdots + z_i \in U \ \  \text{ for } i =1, \dots, r.
  \end{equation}
  Then, there exists an $S$-walk
  $ \gm':
\xymatrix@R=1.9em@C=1.7em{
p_0 \ar@{->}[r]^{\quad  z_1} & p_1 \ar@{->}[r]^{z_2} & \cdots \ar@{->}[r]^{z_r\quad } & p_r}$
such that the following diagram commutes:
 $$
\xymatrix@R=1.9em@C=1.7em{
u_0 \ar@{->}[r]^{\quad  z_1} & u_1 \ar@{->}[r]^{z_2} & u_2 \ar@{->}[r] & \cdots \ar@{->}[r]^{z_r\quad } & u_r
\\
p_0 \ar@{->}[r]^{\quad  z_1} \ar@{->}[u]_w & p_1 \ar@{->}[r]^{z_2} \ar@{->}[u]_w & p_2 \ar@{->}[r]\ar@{->}[u]_w & \cdots \ar@{->}[r]^{z_r\quad } & p_r \ar@{->}[u]_w .
}$$

\end{prop}
\begin{proof}
  It suffices to verify the assertion for $r =1$, and then the claim follows by iteration. Let
  $p_1 := p_0 + z_1$. Then, $z_1 + w = w + z_1$, and thus the square
  $$ \xymatrix@R=1.5em@C=2.7em{
u_0 \ar@{->}[r]^{  z_1 \quad } & u_1 \\
p_0 \ar@{->}[r]^{  z_1 \quad } \ar@{->}[u]_{w} & p_1 \ar@{->}[u]_{w}
  }$$
commutes.  Proceed to the next walk $ \xymatrix@R=1.5em@C=1.7em{
u_1 \ar@{->}[r]^{  z_2 \quad } & u_2}$ and so on.
\end{proof}

\begin{defn}\label{def:b10.6}
We call the $S$-walk $\gm'$ the \textbf{push-down} of $\gm$ by $w$, and write
$\gm' = p_w(\gm)$, provided that $w \leq u_0$. We also define the push-down of a walk $\{ u_0\} $
of length zero by
$p_w(\{ u_0 \}) = \{ p_0\} $. We denote  the subwalk of $\gm$
from $u_k$ to $u_\ell$ by $\gm|_{[k,\ell]}$, for $1 \leq k \leq \ell \leq r$, often writing $\gm \leq \ell$, if $k=1$, and $\gm \geq k$, if $\ell = r$.
\end{defn}

\begin{remark}\label{rem:b10.7} $ $
\begin{enumerate}\ealph
  \item $p_w(\gm) + w = \gm$,
  \item$p_w(\gm|_{[k,\ell]}) = p_w(\gm)|_{[k,\ell]}$,
  \item $p_{w_1 + w_2}(\gm) = p_{w_1}(p_{w_2}(\gm))
= p_{w_2}(p_{w_1}(\gm))$
 \end{enumerate}
\end{remark}

We  have a version of  duality among push-downs of $\gm$. Assume that
$u_0 = p_0 + w$ with~ $p_0$ and  $w$ in $U$.
 We choose a decomposition $w = w_1 + w_2$, and then have the push-downs $p_w(\gm)$, $p_{w_1}(\gm)$, and $p_{w_2}(\gm)$, where
$  p_{w_1}(p_{w_2}(\gm))
= p_{w_2}(p_{w_1}(\gm)) = p_{w_1 + w_2}(\gm)$ by Remark ~\ref{rem:b10.7}. We set
$p^*_{w_1}(\gm) = p_{w_2}(\gm)$, and  consequently set
$p^*_{w_2}(\gm) = p_{w_1}(\gm)$.
Then,
$(p^*_{w_1}(\gm))^* = p_{w_1}(\gm)$ and $(p^*_{w_2}(\gm))^* = p_{w_2}(\gm)$.

Assuming that $(R,U)$ is an arbitrary \as\ monoid, we say that $u$ \textbf{contains} $s$, if $m >0$,  and that $u$ is \textbf{alien} to $s$ otherwise, and write  $m = \mu_s(u)$.
The set of elements in $U$ that are lonely for $\N s$ is denoted by $\Lo_s(U)$.

\begin{lem}\label{lem:b10.10}
If  $u \in U \sm \Lo_s(U)$, then
  \begin{equation}\label{eq:b10.2}
   u = ms + u_0
\end{equation}
for a unique maximal $m \in \NN$, denoted
  %\begin{equation}\label{eq:b10.3}
   $m = \mu_s(u)$,
%\end{equation}
 and some $u_0 \in U$ ($u_0$ is not necessarily unique).
\end{lem}

\begin{proof}
  This is clear, since we always  assume that $\leq_R$ is a partial ordering on the set $R$.
  ($R$ is ``upper bound''.)
\end{proof}
We say that an element $u$ of $U \sm \Lo_s(U)$ \textbf{contains} $s$, if $m > 0$ in \eqref{eq:b10.2}, and that $u$ is \textbf{alien} to $s$, if $m=0$. If $u \in U \sm \Lo_s(U)$, we say for short that $u$ is \textbf{$s$-lonely}.
We aim  to understand the number
$\mu_s(u+v)$ for two elements $u,v $ of $U$ that are not lonely.

\begin{prop}\label{prop:b10.12} $ $
\begin{enumerate} \ealph
  \item $\mu_s(u+v) \geq \mu_s(u) + \mu_s(v)$
  for any $s \in \Gap(R,U)$ and $u,v \in  U \sm \Lo_s(U)$.
  \item  If the  sum
  $u + v$ is alien to $s$ whenever $u,v\in U\sm\Lo_s(U)$ are alien to $s$,  then  $\mu_s(u+v) = \mu_s(u) + \mu_s(v)$  for all
  $u,v \in  U \sm \Lo_s(U)$.
 \end{enumerate}
\end{prop}

\begin{proof} Note that
  $u = ms +u_0$ and $v = ns +v_0$,  where $u_0, v_0$ are alien to $s$ and $m = \mu_s(u)$,
  $n = \mu_s(v)$. Since $u_0$ and  $v_0$ are not lonely for $\N s$, it is clear that
  $u_0 + v_0$ is  not lonely for $\N s$. Furthermore,
    \begin{equation}\label{eq:b10.4}
 u +v = (m+n)s + u_0 + v_0.\end{equation}
  This implies that $m + n \leq \mu_s (u+v)$, which proves part (a).
    If $u + v $ is alien to $s$, then \eqref{eq:b10.4} implies  that
  $\mu_s(u +v) = m+n $, which proves part (b).
\end{proof}

We present an example in which the sum of two aliens is almost never an alien.

\begin{examp} Let $R = \NN / E$, where  $r \geq 2$ and $E$ a nonempty additive equivalence relation compatible with the standard weight function $w: \NN^r \to \NN$, as decribed in Theorem~\ref{thm:6.9}. Note that $d \geq 2$ is the minimal weight such that there exist elements of $w^{-1}(d)$ that are not singletons. Pick an extremal generator
$s_i = [t_i]_E$. Then \textbf{no} $E$-equivalence class of weight  $e > d$,  except
$ e[t_i]_E = e s_i$,    is alien to $s_i$.  Thus, no sum $x+y$, where  $x,y \notin \NN s_i$ and $w(x) + w(y) > d$,  is alien to $s_i$.
\end{examp}
\section{Ancestors of blue walks}\label{sec:d10}
It is natural to view a blue walk
$ \xymatrix@R=1.5em@C=1.7em{u_0 \ar@{->}[r]^{  z_1  } & u_1}$ of length one as a model of a gap in a glacier, and to look for traces of
%$ \xymatrix@R=1.5em@C=1.7em{
%u_0 \ar@{->}[r]^{  z_1  } & u_1
%}$
this gap in the geological past.
We pursue this idea by introducing a commutative diagram
\begin{equation}\label{eq:d10.1}
 \xymatrix@R=1.5em@C=2.7em{
u_0 \ar@{->}[r]^{  z_1 \quad } & u_1 \\
p_0 \ar@{->}[r]^{  z_1 \quad } \ar@{->}[u]_{w} & p_1 \ar@{->}[u]_{w}
  }
\end{equation}
of blue walks, i.e., where
$ \xymatrix@R=1.5em@C=1.5em{
p_0 \ar@{->}[r]^{  z_1  } & p_1
}$
is a push-down of
$ \xymatrix@R=1.5em@C=1.5em{
u_0 \ar@{->}[r]^{  z_1  } & u_1
}$.
  We call  $ \xymatrix@R=1.5em@C=1.5em{
p_0 \ar@{->}[r]^{  z_1  } & p_1
}$
an \textbf{ancestor}  of the gap
$ \xymatrix@R=1.5em@C=1.7em{
u_0 \ar@{->}[r]^{  z_1  } & u_1
}$, and call \eqref{eq:d10.1}
an  \textbf{ancestor diagram}.

Searching for ancestors of $ \xymatrix@R=1.5em@C=1.5em{
u_0 \ar@{->}[r]^{  z_1  } & u_1
}$ means choosing  elements $w \in U$ and $p_0 \in U$ with
$p_0 + w = u_0$ and testing  whether $p_1 := p_0 + z_1$ is in $U$.\footnote{Think of a drill into the glacier at $u_0$.} In this case, we have found an ancestor~\eqref{eq:d10.1}
of $ \xymatrix@R=1.5em@C=1.5em{
u_0 \ar@{->}[r]^{  z_1  } & u_1
}$ at   the ``time''~ $w$ in the past,
  $$ \xymatrix@R=1.5em@C=2.7em{
u_0 \ar@{->}[r]^{  z_1 \quad } & u_1 \\
p_0 \ar@{->}[r]^{  z_1 \quad } \ar@{->}[u]_{w} & p_1 \ar@{->}[u]_{w} .
  }$$
  This diagram is uniquely determined by the node $p_0$ and $w$, since
  $(U,+)$ is cancellative. Often, the label $w$ on the vertical arrows can be omitted.

  We describe this situation by using a partially ordered index subset $A$ and families
  $(p_{0,\al} \ds \mid \al \in A)$ and  $(w_\al \ds \mid  \al \in A)$ of elements of $U$ such that %\begin{equation*}\label{eq:d10.2}
    $u_0 = p_{0,\al} + w_\al.$
  %\end{equation*}
  for every $\al$.
  In addition,
  \begin{equation*}\label{eq:d10.3}
  w_\al = w_\bt \dss{\Leftrightarrow}
  p_{0,\al} = p_{0,\bt}
  \dss{\Leftrightarrow} \al= \bt
  \end{equation*}
  (faithful indexing) and $\al < \bt$ if and only if there exists  a  unique  $w_{\al,\bt} \in U$ such that
  %\begin{equation*}\label{eq:d10.4}
  $p_{0,\bt} + w_{\al, \bt} = p_{0,\al}$.
%\end{equation*}
For $\al < \bt$,
this gives the following commutative diagram
\begin{equation}\label{eq:d10.5} \xymatrix@R=1.5em@C=1.7em{
u_0 \ar@{->}[rr]^{  z_1 \quad } & & u_1 \\
p_{0, \al} \ar@{->}[rr]^{  z_1  } \ar@{->}[u] & & p_{1,\al} \ar@{->}[u]
 \\
& p_{0, \bt} \ar@{->}[rr]^{  z_1  } \ar@/^3.5pc/[uul]<0ex>,
\ar@{->}[ul] & & p_{1,\bt}
\ar@/_1.8pc/[uul]<0ex> \ar@{->}[ul] .
  }\end{equation}
Assuming, as before, that every $u \in U$ has finite gap height in $(R,+)$,  for every
$\al \in A$ there exists  a minimal element of $\al^\downarrow$. This gives rise to a \textbf{first} (or \textbf{primordial}) \textbf{ancestor}.

Our theory of ancestors of blue walks of length one extends to a theory of ancestors of a blue walk
  $$ \gm:
\xymatrix@R=1.9em@C=1.7em{
u_0 \ar@{->}[r]^{  z_1} & u_1 \ar@{->}[r]^{z_2} & u_2 \ar@{->}[r] & \cdots \ar@{->}[r]^{z_r } & u_r
}$$
of a fixed length $r > 1$ as follows.
Given $p_0 \in U$ and $w \in U$ with
$p_0+w = u_0$, we obtain ancestor diagrams
\begin{equation}\label{eq:d10.6} \xymatrix@R=1.5em@C=2.7em{
u_0 \ar@{->}[r]^{  z_1 \quad } & u_1 \\
p_0 \ar@{->}[r]^{  z_1 \quad } \ar@{->}[u]_{w} & p_1 \ar@{->}[u]_{w}, &
  }
\xymatrix@R=1.5em@C=2.7em{
u_1 \ar@{->}[r]^{  z_2 \quad } & u_2 \\
p_1 \ar@{->}[r]^{  z_2 \quad } \ar@{->}[u]_{w} & p_2 \ar@{->}[u]_{w}, & \dots  &
  }
  \xymatrix@R=1.5em@C=2.7em{
u_{k-1} \ar@{->}[r]^{  z_k \quad } & u_{k} \\
p_{k-1} \ar@{->}[r]^{  z_k \quad } \ar@{->}[u]_{w} & p_k \ar@{->}[u]_{w} .
  }
  \end{equation}
as long as $p_{i-1} + z_i \in U$ for $i = 1, \dots,  k$. Hence, we  obtain a commutative diagram of blue walks, in which $\gm'_{\leq k}$ is a push-down of $\gm_{ \leq k}$:
\begin{equation}\label{eq:d10.7} \xymatrix@R=1.5em@C=2.7em{
u_0 \ar@{->}[r]^{  z_1 \quad } & u_1 \ar@{->}[r]^{  z_2 \quad } & \cdots &
\ar@{->}[r]^{  z_{k-1} \quad } &
u_{k-1} \ar@{->}[r]^{  z_k \quad } & u_{k} \\
p_0 \ar@{->}[r]^{  z_1 \quad } \ar@{->}[u]_{w} & p_1 \ar@{->}[u]_{w}  \ar@{->}[r]^{  z_2 \quad } & \cdots &
\ar@{->}[r]^{  z_{k-1} \quad }
& p_{k-1} \ar@{->}[r]^{  z_k \quad } \ar@{->}[u]_{w} & p_k \ar@{->}[u]_{w} .
  }
  \end{equation}
We call the push-down $\gm'_{ \leq k}$ an \textbf{ancestor of $\gm_{ \leq k}$}, and, when  $k =r$,
an \textbf{ancestor} of $\gm$. These ancestors  are uniquely determined by $p_0$ and $w$.
 We refer to the diagrams in \eqref{eq:d10.7} as \textbf{ancestor diagrams}.
%We also consider \textbf{ancestor diagrams}, namely \eqref{eq:d10.7} for $\gm_{ \leq k}$ and $\gm$, respectively.

\begin{remark}\label{rem:d10.1}
$\gm'_{ \leq k}$ is an ancestor of  $\gm_{ \leq k}$ if and only if
$ p_0 + (z_1 + \cdots + z_i) = p_i \in U$
for $i = 1, \dots,k$, in other words, if and only if, for $i = 1, \dots,k$,
$$ \xymatrix@R=1.5em@C=1.7em{
 p_0 + (z_1 + \cdots + z_{i-1}) \ar@{->}[r]^{  z_i  } &  p_0 + (z_1 + \cdots + z_i)
}
$$ is an ancestor of
$$ \xymatrix@R=1.5em@C=1.7em{
 u_0 + (z_1 + \cdots + z_{i-1}) \ar@{->}[r]^{  z_i  } &  u_0 + (z_1 + \cdots + z_i)
}.
$$
\end{remark}

We introduce the sum of ancestor diagrams as follows.

\begin{defn}\label{def:d10.2}
The \textbf{sum} of two ancestor diagrams
\begin{equation}\label{eq:d10.8} \xymatrix@R=1.5em@C=2.7em{
u'_0 \ar@{->}[r]^{  z' \quad } & u'_1 \\
p'_0 \ar@{->}[r]^{  z' \quad } \ar@{->}[u] & p'_1 \ar@{->}[u], &
  }
\xymatrix@R=1.5em@C=2.7em{
u''_0 \ar@{->}[r]^{  z'' \quad } & u''_1 \\
p''_0 \ar@{->}[r]^{  z'' \quad } \ar@{->}[u] & p''_1 \ar@{->}[u] &
  }
  \end{equation}
  is the diagram
\begin{equation}\label{eq:d10.9} \xymatrix@R=1.5em@C=2.7em{
u'_0 + u''_0 \ar@{->}[r]^{  z'+z''  } & u'_1 + u''_1 \\
p'_0 + p''_0 \ar@{->}[r]^{  z' + z''   } \ar@{->}[u] & p'_1 + p''_1 . \ar@{->}[u]  &
  }
  \end{equation}
\end{defn}
This sum remains unchanged if we interchange the two diagrams in \eqref{eq:d10.8}. The following is readily verified.

\begin{prop}\label{prop:d10.3} The right ancestor diagram in \eqref{eq:d10.8} is related to the sum diagram by the following commutative diagram:
\begin{equation*}\label{eq:d10.9} \xymatrix@R=1.5em@C=2.7em{
u'_0 + u''_0 \ar@{->}[r]^{  z'+z''  } & u'_1 + u''_1 \\
p'_0 + p''_0 \ar@{->}[r]^{  z' + z''   } \ar@{->}[u]_{w} & p'_1 + p''_1\ar@{->}[u]_{w} &
\\
u''_0 \ar@{->}[r]^{  z'' \quad }
\ar@/^2.8pc/[uu]<0ex>
& u''_1  \ar@/_2.8pc/[uu]<0ex> \\
p''_0 \ar@{->}[r]^{  z'' \quad } \ar@{->}[u]
\ar@/^2.8pc/[uu]<0ex>
& p''_1 \ar@{->}[u] \ar@/_2.8pc/[uu]<0ex> &
  }
  \end{equation*}
\end{prop}

We obtain an analogous diagram using the left ancestor diagram in \eqref{eq:d10.8}. The concept of ancestors and ancestor diagrams can be generalized based on the following fact.

\begin{thm}\label{thm:d10.4} Assume that $ \xymatrix@R=1.5em@C=1.7em{
p_0 \ar@{->}[r]^{  y } & p_1
}$ and $ \xymatrix@R=1.5em@C=1.7em{
q_0 \ar@{->}[r]^{  z } & q_1
}$
  are blue walks of length one, and let  $w_0$ and $w_1$ be elements of $U$ such that the diagram
  \begin{equation}\label{eq:d10.10}
 \xymatrix@R=1.5em@C=2.7em{
q_0 \ar@{->}[r]^{  z \quad } & q_1 \\
p_0 \ar@{->}[r]^{  y \quad } \ar@{->}[u]_{w_0} & p_1 \ar@{->}[u]_{w_1}
  }
\end{equation}
commutes. Then,
$ z + w_0 = y + w_1.$
\end{thm}

\begin{proof} The sums $z+w_0$ and $y+w_1$ are elements of $\Gap(R,U)$ that are bridges from $p_0$ to $q_1$. Since $U$ is an \as~submonoid of $R$, there is at most one bridge from $p_0$ to $q_1$.
\end{proof}

We call diagram \eqref{eq:d10.10} a \textbf{mixed ancestor diagram}, and $ \xymatrix@R=1.5em@C=1.7em{
p_0 \ar@{->}[r]^{  y } & p_1
}$ a \textbf{mixed ancestor} of $ \xymatrix@R=1.5em@C=1.7em{
q_0 \ar@{->}[r]^{  z } & q_1
}$. The term ``mixed'' alludes to the fact that the two gap elements $z$ and $y$ are involved, rather than only one gap element, $z$, as in the case of ancestors.

\begin{remark}\label{rem:d10.15} $ $

\begin{enumerate}\ealph
  \item A mixed ancestor diagram \eqref{eq:d10.10} and a second such diagram
\begin{equation}\label{eq:d10.11} \xymatrix@R=1.5em@C=2.7em{
p_0 \ar@{->}[r]^{ z'  \quad } & p_1  \\
p'_0 \ar@{->}[r]^{  y'  \quad } \ar@{->}[u]^{w'_0} & p'_1 \ar@{->}[u]_{w'_1}
  }\end{equation}
can be combined to a mixed ancestor diagram
\begin{equation}\label{eq:d10.12} \xymatrix@R=1.5em@C=2.7em{
q_0 \ar@{->}[r]^{ z  \quad } & q_1  \\
p'_0 \ar@{->}[r]^{  y'  \quad } \ar@{->}[u]^{w_0 + w'_0} & p'_1 \ar@{->}[u]_{w_1+ w'_1}  .
  }\end{equation}
\item  A ``horizonal'' family of diagrams
\begin{equation*} \xymatrix@R=1.5em@C=2.7em{
u_0 \ar@{->}[r]^{  z_1 \quad } & u_1 \\
p_0 \ar@{->}[r]^{  y_1 \quad } \ar@{->}[u]_{w_0} & p_1 \ar@{->}[u]_{w_1}, &
  }
\xymatrix@R=1.5em@C=2.7em{
u_1 \ar@{->}[r]^{  z_2 \quad } & u_2 \\
p_1 \ar@{->}[r]^{  y_2 \quad } \ar@{->}[u]_{w_1} & p_2 \ar@{->}[u]_{w_2}, & \dots  &
  }
  \xymatrix@R=1.5em@C=2.7em{
u_{r-1} \ar@{->}[r]^{  z_r \quad } & u_{r} \\
p_{r-1} \ar@{->}[r]^{  y_r \quad } \ar@{->}[u]_{w_{r-1}} & p_r \ar@{->}[u]_{w_r}
  }
  \end{equation*}
  can be combined to a commutative  diagram
  \begin{equation}\label{eq:d10.13} \xymatrix@R=1.5em@C=2.7em{
u_0 \ar@{->}[r]^{  z_1 \quad } & u_1 \ar@{->}[r]^{  z_2 \quad } & \cdots &
\ar@{->}[r]^{  z_{r-1} \quad } &
u_{r-1} \ar@{->}[r]^{  z_r \quad } & u_{r} \\
p_0 \ar@{->}[r]^{  y_1 \quad } \ar@{->}[u]_{w_0} & p_1 \ar@{->}[u]_{w_1}  \ar@{->}[r]^{  y_2 \quad } & \cdots &
\ar@{->}[r]^{  y_{r-1} \quad }
& p_{r-1} \ar@{->}[r]^{  y_r \quad } \ar@{->}[u]_{w_{r-1}} & p_r \ar@{->}[u]_{w_r}
  }
  \end{equation}
with a blue walk $\gm$ in the upper horizonal line and a blue walk $\gm'$ in the lower
horizonal line.
We call $\gm'$ a \textbf{mixed ancestor of $\gm$}, and the diagram \eqref{eq:d10.13} a
\textbf{mixed ancestor diagram of blue walks}.

\end{enumerate}

\end{remark}

We introduce a ``splitting'' of gap elements of $(R,U)$.

\begin{defn}\label{def:d10.16} \ealph $ $
\begin{enumerate}
  \item We call an element $z \in \Gap(R,U)$ \textbf{pure}, if it has  \textbf{no}
  decomposition $z = z' + v$ with $v \in U \sm \00$. In other words, $z$ is minimal with respect to $\leq_U$. We then also say that~ $z$ is a \textbf{tight gap element} of~ $(R,U)$.
  \item We call
  $
\xymatrix@R=1.9em@C=1.6em{
\gm: u_0 \ar@{->}[r]^{  z_1} & u_1 \ar@{->}[r]^{z_2} &\cdots \ar@{->}[r]^{z_r } & u_r
}$  a \textbf{tight blue gap walk}, if  every $z_i$ is minimal with respect to ~$\leq_U$.

\end{enumerate}
\end{defn}
%Note that, if $\gm_1$ and $\gm_2$ are tight blue walks of same length , it may well happens
 % that $\gm_1 + \gm_2$ is not tight.

We have the following method for obtaining  tight blue walks that are mixed ancestors of a given blue walk $
\xymatrix@R=1.9em@C=1.7em{
\gm: u_0 \ar@{->}[r]^{  z_1} & u_1 \ar@{->}[r]^{z_2} &  \cdots \ar@{->}[r]^{z_r } & u_r
}$, assuming,  for simplicity,  that none of the bridges is tight. Note that, in the case $r =1$,
a tight ancestor is given by the diagram
$$
\xymatrix@R=1.5em@C=1.5em{
u_0 \ar@{->}[rr]^{  z_1 \quad }  \ar@{->}[rd]_{  z'_1  }  & & u_1 \\
& u'_1 \ar@{->}[ru]^{  v_1  } &
  }$$
 with $v_1 \in U \sm \00$. Starting
 with $\gm$, we choose a commutative diagram
 \begin{equation}\label{eq:d10.14}
   \xymatrix@R=1.5em@C=1.5em{
u_0 \ar@{->}[rr]^{  z_1 \quad }  \ar@{->}[rd]_{  z'_1  }  & & u_1  \ar@{->}[rr]^{  z_2 \quad }  \ar@{->}[rd]_{  z'_2  } && u_2  \ar@{->}[rr]^{  \quad } && \cdots \ar@{->}[rr] &&
u_{r-1} \ar@{->}[rr]^{  z_r \quad }  \ar@{->}[rd]_{  z'_r  } && u_r  \\
& u'_1 \ar@{->}[ru]^{  v_1  } & & u'_2 \ar@{->}[ru]^{  v_2  }
&& && &&  u'_r \ar@{->}[ru]^{  v_r  }
  }
 \end{equation}
  of walks in $(R,U)$
with  elements  $v_1, \dots, v_r \in U \sm \00$ such that the elements
$z'_1, \dots, z'_r$ are pure gaps.

\begin{lem}\label{lem:d10.17}
For any $i < r$ there exists a unique element $v_{i+1}\in U \sm \00$ such that the following  diagram  commutes
$$
\xymatrix@R=1.5em@C=2.7em{
u_i \ar@{->}[r]^{  z'_i \quad } & u_{i+1} \\
u'_i \ar@{->}[r]^{  z'_i \quad } \ar@{->}[u]_{v_i} & u'_{i+1} \ar@{->}[u]_{v_{i+1}} \ .  &
  }
  $$
\end{lem}
\begin{proof}
  Follows immediately from  Theorem \ref{thm:d10.4}.
\end{proof}

Using this lemma iteratively, we uniquely complete the diagram of walks \eqref{eq:d10.14} to a commuting diagram, as shown below for $r=4$. The diagram consists of $r$ tight blue walks associated with the blue subwalks $\gm_{\geq k}$ of~ $\gm$ for $k=0,\dots,r-1$, together with push-downs connecting them.

 
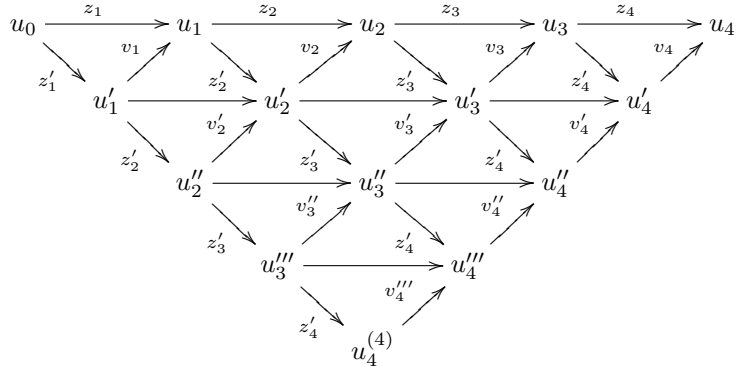
\begin{equation*}
   \xymatrix@R=1.5em@C=1.5em{
u_0 \ar@{->}[rr]^{  z_1 \quad }  \ar@{->}[rd]_{  z'_1  }  & & u_1  \ar@{->}[rr]^{  z_2 \quad }  \ar@{->}[rd]_{  z'_2  } && u_2  \ar@{->}[rr]^{  z_3 \quad }
\ar@{->}[rd]_{  z'_3  } &&
u_{3} \ar@{->}[rr]^{  z_4 \quad }  \ar@{->}[rd]_{  z'_4  } && u_4
 \\
& u'_1 \ar@{->}[ru]^{  v_1  } \ar@{->}[rd]_{  z'_2  }  \ar@{->}[rr] & &
u'_2 \ar@{->}[ru]^{  v_2  } \ar@{->}[rd]_{  z'_3  }  \ar@{->}[rr]
&& u'_3 \ar@{->}[ru]^{  v_3  } \ar@{->}[rd]_{  z'_4  }  \ar@{->}[rr] &&
u'_4 \ar@{->}[ru]^{  v_4  }
\\
  & & u''_2 \ar@{->}[ru]^{  v'_2  } \ar@{->}[rd]_{  z'_3  } \ar@{->}[rr]
&& u''_3 \ar@{->}[ru]^{  v'_3  }  \ar@{->}[rr]
\ar@{->}[rd]_{  z'_4  } &&  u''_4 \ar@{->}[ru]^{  v'_4  }
\\
  &
&& u'''_3  \ar@{->}[rr] \ar@{->}[ru]^{  v''_3  } \ar@{->}[rd]_{  z'_4  } &&  u'''_4 \ar@{->}[ru]^{  v''_4  }
\\
&&  &&  u^{(4)}_4 \ar@{->}[ru]^{  v'''_4  }
  }
 \end{equation*}
Here, the walks pointing southeast are tight, while those pointing northeast are in $U\sm\00$. The walks pointing east are mixed push-downs of $\gm$.

\section{Blue $S$-walks for $S$ a singleton, $S = \{ s \}$}\label{sec:e10}

Let $(R,+)$ be an \ub additive monoid and $U$ an \as submonoid of $R$.
Assume that $s \in R \sm U$ amd that there exist elements $u \in U$ such that
$u + ms \in U$ for some $m \in \N$. We denote the set of such elements $u$ by
$$B(U,s) := B(R,U,s) \footnote{We now abandon the assumption, made in  \S\ref{sec:a8},  that every $u \in  U\sm \00$ has finite gap height. } $$
Then,  for every $ u \in B(U,s)$, there exists a unique maximal blue $\{s\}$-walk (Definition \ref{def:a9.4})
\begin{equation}\label{eq:e10.1}
   \gm_{s,u}:
\xymatrix@R=1.9em@C=2.7em{
u= u_0 \ar@{->}[r]^{ \ m_1 s } & u_1 \ar@{->}[r]^{m_2 s } & u_2 \ar@{->}[r]^{m_3 s} & \cdots
}
\end{equation}
which is either infinite or  finite, in which the terminal node does not belong to $B(U,s)$. We call the $m_i s$ the
\textbf{blue bridges} of $\gm_{s,u}$ (or of the pair $(s,u)$), and call $m_1, m_2, \dots$ the \textbf{characteristic numbers} of $(s,u)$. We  write ``$s$-walk'' instead of $\{ s\} $-walk.

\begin{thm}\label{thm:a10.1}
  Every blue $s$-walk
 $\gm_{s,u}$ is tight.\end{thm}
\begin{proof}
  Suppose that $m_i s = v + ns$ with $v \in U \sm \00$ and $n \in \N$. Then,
  $v = ks$ for some $k > 0$, whence $k+n = m_i$, and hence $0 < n < m_i$. It follows that
  $ns \in U$,  contradicting the minimality of $m_i$.
\end{proof}

We use an agricultural metaphor to understand the $s$-walks $\gm_{s,u}$ more thoroughly. Recall that the set of nodes of $\gm_{s,u}$ is a totally ordered subset of $U$ with respect to the partial ordering $\leq_R$ and its restriction $\leq_U$ to~ $U$ (Definition \ref{def:a4.3}). We say that $B(U,s)$ is the \textbf{blue seed} for $s$ in the subsoil $U$ of $R$, and that $u\in B(U,s)$ is a \textbf{corn} of the seed. We view $\gm_{s,u}$ as the stalk growing from $u$ in $U$, thinking of $s$ as a fertilizer for $u$.

%We resort to a picture from agriculture to understand the $s$-walks $\gm_{s,u}$ more thoroughly. Recall that the set of nodes of $\gm_{s,u}$ is a totally ordered subset of $U$ under the partial ordering ~$\leq_R$ and under  its restriction $\leq_U$ to~ $U$
%(Definition \ref{def:a4.3}). We say that $B(U,s)$ is the \textbf{blue seed} for $s$ in the subsoil $U$ of $R$, and that $ u \in B(U,s)$ is a \textbf{corn} of the seed. We view $\gm_{s,u}$ as the stalk growing from $u$ in $U$, thinking of $s$ as a fertilizer for $u$.

\begin{defn}\label{defn:e10.2}

Given a second corn $v\in B(U,s)$, we say that $\gm_{s,v}$ is \textbf{confluent to} $\gm_{s,u}$,  if there exists a node~ $w$ of $\gm_{s,v}$ that is also a node of $\gm_{s,u}$ and if $\gm_{s,u}$ is finite, is different from its last node.
\end{defn}

The following theorem is now immediate  from  the definitions.

\begin{thm}\label{thm:e10.1} $ $
\begin{enumerate} \ealph
  \item If $\gm_{s,v}$ is confluent to  $\gm_{s,u}$, then $\gm_{s,u}$ is  confluent to  $\gm_{s,v}$ and  $\gm_{s,u} \cap \gm_{s,v} = \gm_{s,w}$
for
   a unique node  $w $
of~ $\gm_{s,u}$.

  \item If $\gm_{s,u}$ is finite, that $\gm_{s,v}$ is finite, and both blue $s$-walks end at the same node.
  Otherwise, $\gm_{s,u}$ and $\gm_{s,v}$ are infinite.

  \item $\gm_{s,w}$ is a \textbf{truncation} of both $\gm_{s,u}$ and $\gm_{s,v}$, consisting of all nodes $x$ of $\gm_{s,u}$ such that $x\geq_u w$.
 \end{enumerate}
\end{thm}
\noindent
If $\gm_{s,u}$ is confluent to $\gm_{s,v}$, then, in accordance with \S\ref{sec:a8}, we say that ``$\gm_{s,u}$ and $\gm_{s,v}$ are confluent.''

Let $\Gm \times \Dl$ be a rectangular blue grid of size $k \times \ell $ (cf.  Construction~\ref{cons:b9.3}), spanned by  a vertical blue $s$-walk~$\Gm$ of length $k$ and a horizontal blue $s'$-walk $\Dl$ of length $\ell$, both starting at
$u = u_{0,0} \in U$. The nodes of $\Gm$ are~ $u_{i,0}$, $0\leq i\leq k$, and those of $\Dl$ are $u_{0,j}$, $0\leq j\leq\ell$, with characteristic numbers
$m_1,\dots,m_k$ and $m'_1,\dots,m'_\ell$, respectively. The grid $\Gm\times\Dl$ has the meshes

  \begin{equation*}\label{eq:e10.}
   \xymatrixcolsep{6mm}
    \xymatrix@R=1.3em@C=1.5em{
     u_{i+1,j} \ar@{->}[rr]^{m'_{j+1}s' } &  &  u_{i+1,j+1} \\
     u_{i,j} \ar@{->}[rr]_{m'_{j+1} s'} \ar@{->}[u]^{m_{i+1}s } &  &  u_{i,j+1}  \ar@{->}[u]_{m_{i+1}s}\\
     }
     \end{equation*}
Note that $\Gm \times \Dl$ is uniquely determined by the walks
$ \xymatrix@R=1.5em@C=1.7em{
u \ar@{->}[r]^{  m_1 s } & u_{1,0}
}$
and
$ \xymatrix@R=1.5em@C=1.7em{
u \ar@{->}[r]^{  m'_1 s' } & u_{0,1}
}$
in the blue seeds for $s$ and $s'$, respectively.
We quote an obvious fact about walks in a grid.

\begin{lem}\label{lem:e10.4} If $\gm$ is a blue walk of length $d \geq 1$ in the grid
$\Gm \times \Dl$ starting at a node $u_{p,q}$ (north or east), then ~$\gm$ remains  in the subrectangle
      \begin{equation}\label{eq:e10.2}
   \xymatrixcolsep{6mm}
    \xymatrix@R=1.3em@C=1.5em{
     u_{k,q} \ar@{->}[rr] &  &  u_{k,\ell} \\
     u_{p,q} \ar@{->}[rr] \ar@{->}[u] &  &  u_{p,\ell}  \ar@{->}[u]\\
     }
     \end{equation}
of $\Gm \times \Dl$ of size $(k-p) \times (\ell -q)$, where $d \leq (k-p) + (\ell -q)$. If $d =1$, then
     \begin{equation}\label{eq:e10.3}
 \gm= \eps =  \xymatrix@R=1.5em@C=1.7em{
u_{p,q} \ar@{->}[r] & u_{p+1,q}
}
%\end{equation}
\dss{or}
 %    \begin{equation}\label{eq:e10.4}
 \gm= \eps =  \xymatrix@R=1.5em@C=1.7em{
u_{p,q} \ar@{->}[r] & u_{p,q+1}
}
\end{equation}
When $d > 1$, there exists a unique factorization $\gm = \eps \circ \gm'$, where $\gm'$ has  length $d-1$.

\end{lem}
We call \eqref{eq:e10.2} an NE-rectangle with \textbf{initial node} $u_{p,q}$, and call $\eps$ the \textbf{expel} of $\gm$ at $u_{p,q}$, or the expel of the  seeds
$(u_{p,q}, s)$ and   $(u_{p,q}, s')$ in \eqref{eq:e10.3}.
We introduce a property of pairs $(\gm_1, \gm_2)$ of walks in $\Gm \times \Dl$ by an inductive definition that explicitly excludes confluence of $\gm_1$ and $\gm_2$.

\begin{defn}\label{def:e10.5}
Let  $\gm_1$ and  $\gm_2$ be two walks of the same length $d$ in the grid $\Gm \times \Dl$. We say that  $\gm_1$ and $\gm_2$
are \textbf{NE-dual} (or the pair  $(\gm_1, \gm_2)$ is NE-dual), if the following holds.
If $d=1$, then  $\gm_1$ and $\gm_2$ start in different directions, one north and the other east. If $d > 1$, then  $\gm_1$ and $\gm_2$ start in different directions, and are therefore in the situation \eqref{eq:e10.3} for $\gm_1$
and $\gm_2$, or vice versa.  We stipulate that the walks $\gm'_1$, $\gm'_2$ of length $d-1$ are also NE-dual.

\end{defn}
\begin{thm}\label{thm:e10.6}
Two NE-dual walks $\gm_1$ and $\gm_2$ in $\Gm\times\Dl$ that start at a common node $u_{a,b}$ and end at a common node $u_{c,d}$, with no other common nodes in between, surround a pen (cf. Definition \ref{def:b9.4}). More precisely, if $\gm_1$ starts north, then $\gm_2$ starts east. The walk $\gm_1$ is the upper part of the pen, and $\gm_2$ is the lower part; both have length $(c-a)+(d-b)$.
\end{thm}
\begin{proof}
This is evident from the preceding discussion and Definition \ref{def:b9.4} of a pen.  \end{proof}

We call  $(\gm_1, \gm_2)$ a \textbf{self dual pen}.  By iteration, we obtain the following.
\begin{cor}\label{cor:e10.7}
If $\gm_1$ and $\gm_2$ are NE-dual walks in $\Gm\times\Dl$, both starting at a node $u$ of $\Gm\times\Dl$ and ending at a node $v$ of $\Gm\times\Dl$, then the pair $(\gm_1,\gm_2)$ surrounds a finite succession of pens. All of these pens are self-dual.
\end{cor}

The following general task seems to be relevant to much of what follows. Its answer depends on the submonoid $U$ of $R$.
 \begin{problem}\label{prob:e10.8}
    Let$\gm_1$ and $\gm_2$ be  NE-dual walks in $\Gm \times \Dl$ starting at the same node $u_{0,0}$ with expels (= basic walks)
   $ \xymatrix@R=1.5em@C=1.7em{
u_{0,0} \ar@{->}[r]^{  m_1 s } & u_{1,0}
}$
and
$ \xymatrix@R=1.5em@C=1.7em{
u_{0,0} \ar@{->}[r]^{  m'_1 s' } & u_{0,1}
}$. Decide, whether $\gm_1$ and $\gm_2$ meet again.
Refined version: Decide whether $\gm_1$ and $\gm_2$ have a common node other than $u_{0,0}$ and the last nodes of $\gm_1$ and $\gm_2$.
 \end{problem}

We  start by describing  arbitrary pens in $\Gm \times \Dl$. Recall that the grid $\Gm \times \Dl$ is a partially ordered set under~ $\leq _U$, with totally ordered $s$-walks in the north direction and  totally ordered $s'$-walks in the east direction. Assume that $(\gm_1, \gm_2) $ runs in the subrectangle  \eqref{eq:e10.2} of $\Gm \times \Dl$ with common nodes $u_{p,q} <_U u_{a,b}$  of $\gm_1$ and $\gm_2$. Then, the upper part~ $\gm_1$ is a totally ordered subset of the rectangle \eqref{eq:e10.2} of cardinality $(a-p) + (b-q) +1$, with first nodes
$u_{p,q}$, $u_{p+1,q}$ and last node $u_{a,b}$, while $\gm_2$ is a totally ordered subset of
\eqref{eq:e10.2} of same cardinality with first nodes
$u_{p,q}$, $u_{p,q+1}$ and  last node $u_{a,b}$. Both sets are maximal \as  walks from $u_{p,q}$ to $u_{a,b}$.

We say that $(\gm_1, \gm_2)$ is a \textbf{pen in $\Gm \times \Dl$ of type $(a-p, b-q)$)}, and call $\gm_1$ the \textbf{upper fence} and  $\gm_2$ the \textbf{lower fence}  of the pen. We sketch a pen of type $(3,4)$.
       \begin{equation}\label{eq:e10.5}    \xymatrixrowsep{3mm}
\xymatrixcolsep{6mm}
    \xymatrix@R=1.7em@C=1em{
   u_{p+3,q}
    \\
        \\
        \\ u_{p,q}
     }
\xymatrixcolsep{6mm}
    \xymatrix@R=1.5em@C=1.8em{
   \bullet \ar@{.}[rr] \ar@{.}[d]  & &\bullet  \ar@{->}[r] & \bullet  \ar@{->}[r] & \bullet     \\
    \bullet \ar@{->}[r] & \bullet  \ar@{->}[r]_{\gm_1} &\bullet  \ar@{->}[u] && \bullet \ar@{->}[u]
        \\  \bullet \ar@{->}[u] && &  \bullet \ar@{->}[r] &  \bullet \ar@{->}[u]
        \\
    \bullet \ar@{->}[r] \ar@{->}[u] & \ar@{->}[r]_{\gm_2}  \bullet &
    \ar@{->}[r] \bullet  &\bullet   \ar@{->}[u] & \bullet \ar@{.}[l] \ar@{.}[u]
     }
      \xymatrix@R=1.7em@C=1em{
   u_{p+3,q+4}
    \\
        \\
        \\ u_{p,q+4}
     }
\end{equation}
 The following is now immediate.
 \begin{thm}\label{thm:e10.9}
   The pens in the subrectangle $Q$ of $\Gm\times\Dl$ with corners $u_{p,q}$ and $u_{p+k,q+\ell}$ correspond uniquely to pairs $(\gm_1,\gm_2)$ of maximal totally ordered subsets of $Q$ that have exactly two common nodes, $u_{p,q}$ and $u_{a,b}$, where $a\leq p+k$ and $b\leq q+\ell$.
 \end{thm}

 \begin{proof}
Observe that such totally ordered subsets form subwalks in the grid $\Gm\times\Dl$ that fit into the description of pens above. \end{proof}

We return to the special case in which $(\gm_1,\gm_2)$ is a dual NE-pair of blue walks in $\Gm\times\Dl$ starting at a common node. Let $v$ be a common node of $\gm_1$ and $\gm_2$ different from the first and last common nodes of $\gm_1$ and $\gm_2$. By Corollary \ref{cor:e10.7}, there exists a self-dual pen $A$ ending at $v$ and a self-dual pen $B$ starting at $v$. Our goal is to simplify $A\cup B$ to a single self-dual pen by a \textbf{modification $(\gm'_1,\gm'_2)$ of $(\gm_1,\gm_2)$ at $v$} that eliminates the common node $v$. Let $x_j$ denote the last node of $\gm_j$ before $v$, and let $w_j$ denote the first node of $\gm_j$ after $v$ ($j=1,2$). Note that $x_j$ and $w_j$ are both \textbf{not} common nodes of $\gm_1$ and $\gm_2$.
%
% We return to the special case that $(\gm_1, \gm_2)$ is a dual NE-pair of blue walks in $\Gm \times \Dl$ starting at  some common node. Let $v$ be a common node of
% $\gm_1$ and $\gm_2$, different from the the first and last common nodes of $\gm_1$ and  $\gm_2$.
% By Corollary \ref{cor:e10.7}, there exists  a self-dual pen  $A$ ending at $v$
% and a selfdual pen $B$ staring at $v$. Our goal  is to simplify
% $A \cup B$ to one self-dual pen by a \textbf{modification $(\gm'_1, \gm'_2)$ of
% $(\gm_1, \gm_2)$ at $v$} which eliminates the common node ~$v$. Let ~ $x_j$
%denote the last node of $\gm_j$ before $v$ and $w_j$ denote the first node of $\gm_j$ after $v$ ($j = 1,2$). Note that $x_j$ and $w_j$ are both \textbf{not}
%common nodes of $\gm_1$ and  $\gm_2$.
Assume without loss of generality that $\gm_1$ contains the upper fence of $A$, and
$\gm_2$ contains the lower fence of $A$.

\begin{construction}\label{const:e10.10}
At $v$, $\gm_1$ changes direction from east to north, while $\gm_2$ changes direction from north to east. We obtain $\gm'_1$ by leaving $\gm_1$ at $x_1$, going one step north and then one step east to reach $w_1$. We then continue $\gm'_1$ along $\gm_1$ (see the sketch below).
Dually, we leave $\gm_2$ at $x_2$, going east and then north to reach $w_2$, and then continue $\gm'_2$ along $\gm_2$.
We call $(\gm'_1,\gm'_2)$ the \textbf{standard modification of $(\gm_1,\gm_2)$ at $v$}.
\end{construction}
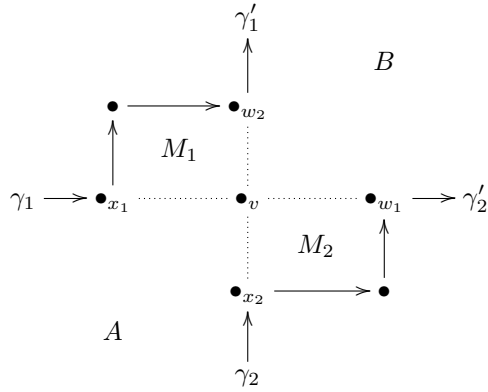
$$
\xymatrixcolsep{6mm}
    \xymatrix@R=0.3em@C=0.5em{
     & &    & &\gm'_1  & & &
     \\ & & &&& & B
\\
     & & \bullet   \ar@{->}[rr]   & &\bullet_{w_2} \ar@{->}[uu] & & &
\\ &&& M_1& &&&
\\
 \gm_1  \ar@{->}[rr] &  & \bullet_{x_1} \ar@{->}[uu]    \ar@{.}[rr]  & &\bullet_{v}
 \ar@{.}[uu] \ar@{.}[rr] & & \bullet_{w_1}   \ar@{->}[rr] & & \gm'_2
\\
 &&& && M_2& &&&
\\
 &  &   & &\bullet_{x_2}
 \ar@{.}[uu] \ar@{->}[rr] & & \bullet  \ar@{->}[uu] & &
\\
      & & A &&& &  \\
           & &    & &\gm_2  \ar@{->}[uu] & & &
   }
   $$
The self-dual pen given by $(\gm'_1,\gm'_2)$ is obtained by adding two meshes $M_1$ and $M_2$ of $\Gm\times\Dl$ to $A\cup B$.

Let $\Gm \times \Dl$ be  a $k \times k$-square with $k > 0$,
$\Gm = \{ u_{0,0}, u_{1,0}, \dots, u_{k,0} \} $, and
$\Dl = \{ u_{0,0}, u_{0,1}, \dots, u_{0,k} \} $.
We establish a further approach to NE-duality by using the involution
     \begin{equation}\label{eq:e10.6}
          \tau: \Gm \times \Dl \To \Gm \times \Dl, \qquad \tau(u_{p,q}) = u_{q,p},
\end{equation}
which switches the coordinates in $\Gm \times \Dl$.

 \begin{thm}\label{thm:e10.11} Given a walk $\gm$ in the grid
 $\Gm \times \Dl$ from $u_{0,0}$ to $u_{p,q}$, where  $ p \leq k$ and $ q \leq k$,
 the dual walk ~$\gm^\vee$  is obtained by applying $\tau$ to the set $\gm$, that is, $\gm^\vee = \tau(\gm)$. Both walks $\gm$ and $ \gm^\vee$ have length $p + q$.
\end{thm}

\begin{proof}
  We label the nodes of $\gm $ by  $u_{0,0} = u_0 < u_1 < \cdots < u_r = u_{p,q}$, and proceed  by induction on the length of~ $\gm$. For $r= 0$,  there is nothing to  do. Assume that  $r > 0$, and let $\gm'$ denote the subwalk of $\gm$ from $u_0$ to $u_{r-1}$ of length $r-1$. By the induction hypothesis, $\tau(\gm') = (\gm')^{\vee}$, and both $\gm'$ and $\tau(\gm')$ have length $p + q -1$.
  By the nature of the grid $\Gm \times \Dl$, there are exactly two cases:
  $u_{r-1}= u_{p-1,q}$ or $u_{r-1}= u_{p,q-1}$. Moreover,
  $(\gm')^\vee$ has nodes
  $$ u_{0,0} = v_0 = u_0 < v_1 < \cdots < v_{r-1}$$
  where $v_{r-1} = u_{q,p-1}$, if $u_{r-1} = u_{p-1,q}$, and
  $v_{r-1} = u_{q-1,p}$, if $u_{r-1} = u_{p,q-1}$. Thus, in the last step the walk ~ $\gm^\vee$ goes east, if $\gm$ goes north, and $\gm^\vee$ goes north, if $\gm$ goes east. This proves that
  $\tau(\gm) = \gm^\vee$, and that both $\gm$ and~ $\gm^\vee$ have length $p+q$.
\end{proof}

\begin{cor}\label{cor:e10.12} Let $\gm$  be a walk  from $u_{p,q}$ to $u_{a,b}$ in a  $k$-square $\Gm \times \Dl$,  where $ k \leq a$ and $ k \leq b$.  Then,  $\gm^\vee = \tau(\gm)$, and both  $\gm$ and  $\gm^\vee$ have length $a-p + b -q$.
\end{cor}

\begin{proof}
  Choose a walk $\eps$ from $u_{0,0}$ to $u_{p,q}$.
  Then, $\eps \circ \gm$ is a walk from $u_{0,0}$ to $u_{a,b}$,  $\tau(\eps)$
  is a walk  from $u_{0,0}$ to $u_{q,p}$, and
  $\tau(\eps \circ \gm)$ is a walk from $u_{0,0}$ to $u_{b,a}$. Clearly
  $\tau(\eps \circ \gm)= \tau(\eps) \circ \tau(\gm)$, $(\eps \circ \gm)^\vee = \eps^\vee \circ \gm^\vee$.
By Theorem~ \ref{thm:e10.11}, $\tau(\eps) = \eps^\vee$ and $\tau(\eps \circ \gm) = (\eps \circ \gm)^\vee = \eps^\vee \circ \gm ^\vee$. It follows that
$\tau(\gm) = \gm^\vee$, that  $\eps$ and $\eps^\vee $ both have length $p +q$,
and that $\eps \circ \gm$ and   $\eps^\vee \circ \gm^\vee$ both have length $a+b$. Thus, $\gm$ and
$\gm^\vee$ have length $(a-p)+ (b-q)$.
\end{proof}

We develop a way to ``narrow'' a given self-dual pen $A$ by omitting two of its dual meshes.
\begin{construction}\label{const:e10.13}
Let $\Gm \times \Dl$ be a self-dual $k\times k$
 square, where
 $\Gm = \{ u_{0,0}, u_{1,0}, \dots, u_{k,0} \} $ and
$\Dl = \{ u_{0,0}, u_{0,1}, \dots, u_{0,k} \} $, and  let
$A$ be a self-dual pen in $\Gm \times \Dl $ that  starts at $u_{0,0} $ and ends at $u_{k.k}$ (for simplicity).  Label the upper fence $\gm$ of $A$ in the usual way by
$$ \gm: u_0 = u_{0,0} < u_1 < \cdots < u_r = u_{k,k},$$
and obtain  the lower fence
$$ \gm^\vee = \tau(\gm): v_0 = u_{0,0} < v_1 < \cdots < v_r = u_{k,k}$$
of $A$, where $v_i = \tau(u_i)$.

Assume that $u_m$ is a node of $\gm$ with $0 < m  < k$, where the orientation of $\gm$ changes from north to east. We say that $u_m$ is a \textbf{north-east corner} of $\gm$. Then, there is a unique mesh~ $M$ in $\Gm \times \Dl$,
      $$ \begin{array}{c} \\ \\
            M :
          \end{array}
   \xymatrixcolsep{6mm}
    \xymatrix@R=1.3em@C=1.5em{
     u_{m} \ar@{->}[rr] &  &  u_{m+1} \\
     u_{m-1} \ar@{->}[rr] \ar@{->}[u] &  &  u'_{m}  \ar@{->}[u]
     }
 $$
and a new walk $\gm'$
obtained from $\gm$ by replacing $u_m$ with $u'_m$ and retaining all other nodes of $\gm$.
% where $u_m$ is replaced by $u'_m$, and all other nodes of $\gm $ are retained
Thus, $u'_m$ is an ``east-north corner'' of $\gm'$. In the dual walk $(\gm')^\vee = \tau(\gm')$, the node $v_m$ is replaced by $v'_m = \tau(u'_m)$, while all other nodes of ~$\gm'$ are retained in $(\gm')^\vee$. We further have the dual mesh
     $$\begin{array}{c} \\ \\
            \tau(M) :
          \end{array}
   \xymatrixcolsep{6mm}
    \xymatrix@R=1.3em@C=1.5em{
     v'_{m} \ar@{->}[rr] &  &  v_{m+1} \\
     v_{m-1} \ar@{->}[rr] \ar@{->}[u] &  &  v_{m}  \ar@{->}[u].
     }
 $$

It is then evident, that
$A' := \tau (A)$ is a self-dual pen with upper fence $\gm'$ and lower fence $(\gm')^\vee = \tau(\gm')$,
obtained from $A$ by removing exactly the two meshes $M$ and $\tau(M)$.
% where exactly two meshes $M$ and $\tau(M)$ of $A$ are taken out.
 We say that $A'$ is the \textbf{narrowing} of $A$ at the nodes $u_m$ and $v_m = \tau(u_m)$.
Note that  $\tau(u_m) = v_m$ is an east-north corner of $\tau(\gm)$.
\end{construction}

We present another way to produce self-dual pens.

\begin{construction}\label{const:e10.14}
Let $Q_1$ and $Q_2$ be overlapping $k\times k$-square and  $\ell \times \ell$-square in
$\Gm \times \Dl$.
Suppose that $Q_1$ starts at $u_{0,0}$ and ends at $u_{k,k}$, while $Q_2$ starts at $u_{p,p}$ and ends at $u_{a,a}$, where
$a = k+\ell - p$ and  $0 < p < k < a$.
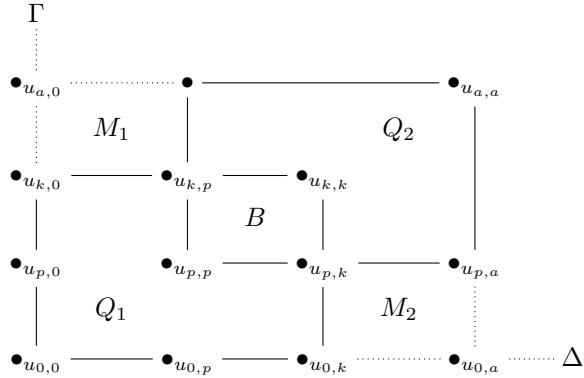
$$
\xymatrixcolsep{6mm}
    \xymatrix@R=0.3em@C=0.5em{
     & &  \Gm  & & & & &
     \\ & & &&& &
\\
     & & \bullet_{u_{a,0}}   \ar@{.}[rr] \ar@{.}[uu]  & &\bullet \ar@{-}[rrrr] & & & &
     \bullet_{u_{a,a}} &
\\ &&& M_1& &&& Q_2
\\
 &  & \bullet_{u_{k,0}} \ar@{.}[uu]    \ar@{-}[rr]  & &\bullet_{u_{k,p}}
 \ar@{-}[uu] \ar@{-}[rr] & & \bullet_{u_{k,k}}   & &
\\
 && &   && B & &&&
\\
 &  & \bullet_{u_{p,0}} \ar@{-}[uu] & &\bullet_{u_{p,p}}
 \ar@{-}[uu] \ar@{-}[rr] & & \bullet_{u_{p,k}}  \ar@{-}[uu]  \ar@{-}[rr]& &  \ar@{-}[uuuu] \bullet_{u_{p,a}}
\\
      & &  & Q_1 && && M_2  &  \\
           & & \ar@{-}[uu] \bullet_{u_{0,0}} \ar@{-}[rr]    & &\bullet_{u_{0,p}} \ar@{-}[rr]   & & \bullet_{u_{0,k}} \ar@{.}[rr] \ar@{-}[uu]  && \bullet_{u_{0,a}} \ar@{.}[rr]
            \ar@{.}[uu] & & \Dl
   }
   $$
   Then,  $Q_1 = \tau(Q_1)$ and $Q_2 = \tau(Q_2)$, and thus  $A = Q_1 \cup Q_2 = \tau(A)$. Hence,  $A$ is a self-dual pen, obtained from the square starting at $u_{0,0}$ and ending at $u_{a,a} $
   by omitting rectangles ~$M_1$ and $M_2 = \tau(M_1)$. Note that $Q_1 \cap Q_2$ is a square $B$ from $u_{p,p}$ and $u_{k,k}$.
   \end{construction}

We begin  a deeper study of self-dual pens in $\Gm \times \Dl$, of which Construction \ref{const:e10.14} may be regarded as a very special case. We denote a node $u_{a,b}$ in  $\Gm \times \Dl$ by the pair $(a,b)$, the ``coordinates'' of $u_{a,b}$. Recall that a pen $A$ in $\Gm \times \Dl$ is self-dual if and only if $\tau(A) = A$ for the involution $\tau: (a,b) \mapsto (b,a)$ on $\Gm \times \Dl$, cf. Theorem \ref{thm:e10.11} and Corollary \ref{cor:e10.12}.

\begin{thm}\label{thm:e10.15}
Let $A$ and $B$ be self-dual pens in $\Gm \times \Dl$ such that  $A \cap B$ is nonempty and is not a singleton, and neither $A \subset B$ nor $B \subset A$. (A pen is identified with the set of members  it contains.) Assume that $A$ starts at $(0,0)$ and ends at $(k,k)$, while $B$ starts at $(p,p) $ and ends at $(p+\ell, p + \ell)$  with
$0 < p < k < p+\ell = k +c$.  Then, $A \cup B$ is a self-dual pen with subpens  $A \sm B$, $A \cap B$, and $B \sm A$,  whose upper fence $\gm$ is obtained by
combining three walks $\gm_1, \gm_2,$ and  $ \gm_3$ which are the upper fences of
$A \sm B$, $A \cap B$, and $B \sm A$, respectively.

Conversely, suppose that a self-dual pen $A$ is partitioned as $A = A_1 \; \dot \cup \;  A_2 \; \dot \cup \;  A_3 $ into  three self-dual pend  such that $A_1$ starts at $(0,0) $ and ends at $(p,p)$,
 $A_2$ starts at $(p,p) $ and ends at $(k,k)$, and
  $A_3$ starts at $(k,k) $ and ends at $(p+ \ell,p + \ell) = (k+c,  k+c)$, with upper fences
  $\gm_1, \gm_2, \gm_3$. Then, there exists a unique subpen $B$ of $A$ such that $A_1 = A \sm B$ and
  $A_2 = B\sm A$ and the upper fences  $\gm_1, \gm_2, \gm_3$ combine to form the  upper fence $\gm$ of $A$.
\end{thm}

\begin{proof}
   The assertions follow immediately  from the definitions,  Construction \ref{const:e10.14}, Definition  \ref{def:e10.5},  and  Theorem~\ref{thm:e10.6}.
\end{proof}

\begin{comment}\label{comm:e10.16} %
The possibility of understanding the upper fence of a self-dual pen $A$ in $\Gm\times\Dl$ as a sequence $\gm_1,\gm_2,\gm_3$ of upper fences of three subpens of $A$ is less uniform (less ``dull'') than may be visible at first glance. Note that the labels of the two edges of a mesh from $(i,j)$ to $(i+1,j+1)$ are $m_{i+1}s$ and $m'{j+1}s'$, with characteristic numbers $m{i+1}$ and $m'_{j+1}$, which may vary considerably.      $$ \begin{array}{c}
                      \end{array}
   \xymatrixcolsep{6mm}
    \xymatrix@R=1.3em@C=1.5em{
     (i+1,j) \ar@{->}[rr]^{m'_{j + 1} s' } &  &  (i+1, j+1) \\
     (i,j) \ar@{->}[rr]^{m'_{j + 1} s' } \ar@{->}[u]^{m_{i+1} s}  &  &  (i, j+1)  \ar@{->}[u]^{m_{i+1} s}\\
     }
 $$

\end{comment}

\section{Idempotent gap walks in $(R,U)$}\label{sec:f10}

As before, we assume that $U$ is an \as\ submonoid of $R$ and that $R$ is upper bounded; cf. \S\ref{sec:1}. We always work with the \ub\ ordering $\leq_R$, which we abbreviate to $\leq$. Recall from \S\ref{sec:2} that $\Idm(R)$ is the set of fixed points in $R$ under the action of $\N$ on the semigroup $(R,+)$. We mention the following basic fact.

\begin{lem}\label{lem:f10.1}
  If $x,y \in R$ and $x+y = t \in \Idm(R)$, then $mx + ny = t$ for any $m,n \in \NN$ that are not both zero.
\end{lem}

\begin{proof}
  $t \leq mx+ ny \leq (m +n)(x+y) = (m+n) t  =t$, since $m+n >0$.
\end{proof}

In all that follows, we assume the following axiom.
\begin{axiom}\label{ax:f10.2}
$\Idm(R)$ is a well ordered set under $\leq _R$, i.e., every descending chain in $\Idm(R)$ terminates.
\end{axiom}
This is a mild condition that, in particular, holds if $R$ has only finitely many idempotents.

\begin{defn}\label{def:f10.3}
We call a gap walk in $(R,U)$ whose bridges lie in $\Idm(R)$ an \textbf{idempotent gap walk}, and we also say that its bridges are \textbf{idempotent gaps}.
\end{defn}

\begin{lem}\label{lem:f10.4}
 Given an idempotent gap
$ \xymatrix@R=1.5em@C=1.4em{
x \ar@{->}[r]^{t} & u
}$, with $t \in \Idm(R) \sm \00$ and  $x,u \in U$, any commutative diagram
\begin{equation}\label{eq:f10.1}
   \xymatrix@R=0.5em@C=2.7em{
u  & \\
& x'  \ar@{->}[lu]_{ t' }, &  \qquad \text{with }  t',h \in \Idm(R), \text{ and }  x' \in U, \\
x \ar@{->}[uu]^{t} \ar@{->}[ru]_{h}
}
\end{equation}
is uniquely  determined  by  $h$ and also  by $t'$.

\end{lem}

\begin{proof}
This is obvious since $(R,U)$ is almost subtractive.
\end{proof}

\begin{thm}\label{thm:f10.5}
 Given a gap
$ \xymatrix@R=1.2em@C=1.4em{
x \ar@{->}[r]^{t} & u
}$ with $x,u \in U$
and a bridge
$t \in \Idm(R) \sm \00$,  there is a \textbf{minimal} idempotent $t_1 < t$ and a gap walk
$ \xymatrix@R=1.5em@C=1.7em{
x \ar@{->}[r]^{h_1} & x_1  \ar@{->}[r]^{t_1} & u
}$
refining $ \xymatrix@R=1.5em@C=1.4em{
x \ar@{->}[r]^{t} & u
}$.
\end{thm}

\begin{proof}
The set of nonzero idempotents $t' \neq 0$ occurring in Diagram \ref{eq:f10.1} is a well ordered subset of $\Idm(R) \sm \00$, and hence has minimal elements $t_1$, typically more than one, satisfying $h_1 + t_1 =t$
(cf.  Theorem~ \ref{thm:a4.8}).
\end{proof}

\begin{construction}\label{cont:f10.6}
  Given a gap walk  $
\xymatrix@R=1.9em@C=1.4em{
\gm: u_0 \ar@{->}[r]^{  z_1} & u_1 \ar@{->}[r]^{z_2} &  \cdots \ar@{->}[r]^{z_r } & u_r
}$ of length $r \geq 1$ and a minimal nonzero idempotent gap
$ \xymatrix@R=1.5em@C=1.4em{
v_0 \ar@{->}[r]^{t} & u_0
}$, we  build  an idempotent gap walk
 $
\xymatrix@R=1.9em@C=1.7em{
\sig: v_0 \ar@{->}[r]^{  z_1} & v_1 \ar@{->}[r]^{z_2} &  \cdots \ar@{->}[r]^{z_r } & v_r
}$ below $\gm$ \footnote{As in \S\ref{sec:d10}, think of $t_0$ as a drill into a glacier at $u_0$ }.  We start with a commutative diagram
\begin{equation}\label{eq:f10.0}
  \xymatrix@R=1.2em@C=2.7em{
u_0  \ar@{->}[r]^{z_1} &  u_1 \\
&   \ar@{->}[u]^{ t_1 } {v_1}   \\
v_0 \ar@{->}[uu]^{t= t_0} \ar@{->}[r]^{z_1} & v_0 + z_1 \ar@{->}[u]^{ h_1 }
}
\end{equation}
Such a diagram exists by Theorem \ref{thm:f10.5}. We repeat this construction for the diagram
$$    \xymatrix@R=1.2em@C=2.7em{
u_1  \ar@{->}[r]^{z_2} &  u_2 \\
&    {v_2} \ar@{->}[u]^{ t_2 }  \\
v_1 \ar@{->}[uu]^{t_1} \ar@{->}[r]^{z_2} & v_1 + u_2 \ar@{->}[u]^{ h_2 }
}
$$
and iterate until we arrive at a walk
$    \xymatrix@R=1.2em@C=1.7em{
v_{r-1} + z_r \ar@{->}[r]^{\quad h_r} &  v_r  \ar@{->}[r]^{t_r} &  u_r }
$ in the last column with an idempotent $h_r$ and a minimal nonzero idempotent $t_r$.
In this way, we obtain an idempotent gap walk
 $
\xymatrix@R=1.9em@C=1.7em{
\sig: v_0 \ar@{->}[r]^{  z_1} & v_1 \ar@{->}[r]^{z_2} &\cdots \ar@{->}[r]^{z_r } & v_r
}$
which we sketch here for $r=4$.

\begin{equation}\label{eq:f10.2}
\xymatrix@R=1.3em@C=1.8em{
u_0 \ar@{->}[r]^{  z_1} & u_1 \ar@{->}[r]^{z_2} & u_2 \ar@{->}[r]^{z_3} & u_3 \ar@{->}[r]^{z_4 } & u_4 \\
 & & & & v_4  \ar@{->}[u]^{  t_4}  \\
  & & & v_3 \ar@{->}[r]^{z_4}  \ar@{->}[uu]^{t_3}  & v_3 + z_4  \ar@{->}[u]^{  h_4} \\
  & &  v_2 \ar@{->}[uuu]^{t_2} \ar@{->}[r]^{z_3} & v_2+ z_3   \ar@{->}[u]^{h_3}  &
\\
   &  v_1 \ar@{->}[uuuu]^{t_1} \ar@{->}[r]^{z_2} & v_1+ z_2   \ar@{->}[u]^{h_2}  &
\\ v_0  \ar@{->}[uuuuu]^{t = t_0} \ar@{->}[r]^{z_1} & v_0 + z_1 \ar@{->}[u]^{h_1}
}
\end{equation}

We call $\sig$ an \textbf{idempotent gap walk supporting} $\gm$ by
$ \xymatrix@R=1.5em@C=1.4em{
v_0 \ar@{->}[r]^{t} & u_0
}$,
and say that $\sig$ is a \textbf{$t$-support} of~ $\gm$ ($t$ is minimal $> 0$ in $\Idm(R)$).

\end{construction}

\begin{prop}\label{prop:f10.7}
If $\sig$ is an idempotent gap walk supporting $\gm$ by
$ \xymatrix@R=1.5em@C=1.4em{
v_0 \ar@{->}[r]^{t} & u_0
}$, then, for any $x \in U$, $\sig +x$ is an
idempotent gap walk supporting $\gm + x $ by
$ \xymatrix@R=1.5em@C=1.4em{
v_0 +x \ar@{->}[r]^t & u_0 +x.
}$
\end{prop}

\begin{proof}
  This is  obvious from Construction  ~\ref{cont:f10.6} and the definitions.
\end{proof}
We now describe a construction ``above'' a given gap walk $\gm$ in a modified submonoid $U$.

\begin{construction}\label{cont:f10.8}
Let $t$ be a minimal element of $\Idm(R) \sm \00$, and let
$
\xymatrix@R=1.9em@C=1.7em{
 \gm: v_0 \ar@{->}[r]^{  z_1} & v_1 \ar@{->}[r]^{z_2} & \cdots \ar@{->}[r]^{z_r } & v_r
}$
be a gap walk of length $r \geq 1$ in $(U \Ng t)_0$. Applying  Construction  ~\ref{cont:f10.6}
to $\gm + t$, we obtain a commutative diagram of walks as follows
 $$
\xymatrix@R=1.3em@C=1.8em{
v_0 + t \ar@{->}[r]^{  z_1} & v_1 +t  \ar@{->}[r]^{z_2} & v_2 + t \ar@{->}[r] & \cdots  \ar@{->}[r]^{z_r } & v_r +t  \\
 & & &  \cdots   & v_{r}   \ar@{->}[u]^{  t_r}  \\
  & &  v_2 \ar@{->}[uu]^{t_2} \ar@{->}[r] & \cdots  \ar@{->}[r]^{z_r}     & v_{r-1} + z_r  \ar@{->}[u]^{  h_r}
\\
   &  v_1 \ar@{->}[uuu]^{t_1} \ar@{->}[r]^{z_2} & v_1+ z_2   \ar@{->}[u]^{h_2}  &
\\ v_0  \ar@{->}[uuuu]^{t = t_0} \ar@{->}[r]^{z_1} & v_0 + z_1 \ar@{->}[u]^{h_1}
}$$
where each $t_i$ minimal lies in $\Idm(R) \sm \00$, $h_1, \dots, h_r \in \Idm(R)$, $v_i \in (U \Ng t_i)_0$, $0 \leq i \leq r$.
Hence,  $v_i + z_{i+1} \in (U \Ng t_i)_0$, and
$
\xymatrix@R=1.9em@C=1.7em{
\sig: v_0 \ar@{->}[r]^{  z_1 + h_1} & v_1 \ar@{->}[r]^{z_2 + h_2} &\cdots \ar@{->}[r]^{z_r  + h_r} & v_r
}$
is an idempotent gap walk supporting $\gm +t$ by
$  \xymatrix@R=1.5em@C=1.4em{
v_0  \ar@{->}[r]^{t} & v_0 +t.
}$

\end{construction}
We take a closer look at the nodes $u_i$ in the diagram \eqref{eq:f10.2} from Construction~\ref{cont:f10.6}.

\begin{defn}\label{def:f10.9}We say that a node $u_k$ of $\gm$, $0 < k \leq r$, is \textbf{active}
(for $t$), if the node $v_k$ is different from $v_{k-1}$. In other words, the idempotent $h_k$ is $> 0$, equivalently, the minimal idempotent  $t_k$ is smaller than $t_{k-1}$. Otherwise, we say that $u_k$ is \textbf{passive} (for $t$).
\end{defn}
\noindent (Example: in diagram \eqref{eq:f10.2} all nodes $u_k$ are active.)

If $u_k$ is {passive}, then removing $u_k$ from the walk $\gm$ and $v_k$ from the walk $\sig$, and replacing $z_k$ by $z_{k-1} + z_k$, we obtain walks $\gm'$ and $\sig'$ of length $r-1$ with $\sig'$ again supporting $\gm'$ by $t$.
Doing this successively for all passive nodes of~ $\gm$, we obtain a walk $\htgm$ with only active nodes under $t$ and a walk $\htsig$ of the same length, which supports $\htgm$ by $t$. We say that $\htgm$ and $\htsig$ are \textbf{condensed by $t$}, or for short, that $\htgm$ and $\htsig$
are \textbf{$t$-condensed}.

Given a commutative diagram \eqref{eq:f10.2} built from a minimal nonzero idempotent  gap
$ \xymatrix@R=1.5em@C=1.4em{
v_0 \ar@{->}[r]^{t_0} & u_0
}$ and a gap $ \xymatrix@R=1.5em@C=1.4em{
u_0 \ar@{->}[r]^{z_i} & u_1
}$
as above (Construction \ref{cont:f10.6}), we say that
$ \xymatrix@R=1.5em@C=1.4em{
v_1 \ar@{->}[r]^{t_1} & u_1
}$
is a \textbf{modification} of $ \xymatrix@R=1.5em@C=1.4em{
v_0 \ar@{->}[r]^{t_0} & u_0
}$ by
$ \xymatrix@R=1.5em@C=1.4em{
u_0 \ar@{->}[r]^{z_i} & u_1
}$.
\begin{thm}\label{thm:f10.10}
  Modifications correspond uniquely to the minimal fixed points (with respect to $\leq_R$) of the action of the semigroup $(\N, +)$  on the convex set
  $$ [v_0 + z_1 , u_1] = \{ x\in R \ds | v_0 + z_1 \leq x \leq u_1  \}.   $$
  These are precisely the elements $x$ in $[v_0 + z_1, u_1]$ satisfying $x = 2x$. Moreover, every decomposition
  $t = t_0 + t_1$ in $\Idm(R)$ with $t_0$ minimal nonzero arises in this way.
\end{thm}
\begin{proof}
  An element $x \in R$ is a fixed point under the action of $(\N, +) $ if and only if $x = 2x$ (cf. Lemma \ref{lem:f10.1}). Applying the additive map $x \mapsto 2x$ to the diagram \eqref{eq:f10.0},
 we obtain a commutative diagram of increasing walks in $R$ with nodes in $U$,
 \begin{equation}\label{eq:f10.4}
   \xymatrix@R=0.5em@C=2.7em{
2 u_0  \ar@{->}[r]^{2 z_1} & 2 u_1   & \\
& & 2 v_1 \ar@{->}[lu]_{t_1} \\
2 v_0  \ar@{->}[uu]^{t_0}  \ar@{->}[r]^{2z_1} &2 v_0 + 2 z_1  \ar@{->}[uu]^{t_0} \ar@{->}[ru]_{h_1}
}
\end{equation}
Since $U$ is \as\ in $R$, we conclude the $t_0 = t_1 + h_1$. Now observe that the additive map $x \mapsto 2x$ on $[v_0 + z_1, u_1]$ is a contraction whose  set of fixed points ia
$[v_0 + z_1, u_1] \cap \Idm(R)$. (Note that $u_0 + z = u_1$.)
  \end{proof}

  We can build blue gap walks $\gm = (u_0, u_1, u_2, \dots )$ of infinite length $\om$ in the obvious sense, as unions of gap walks of finite  length $r \in \N$, $r \to \om$. Given any such walk $\gm$
  and an idempotent gap $ \xymatrix@R=1.5em@C=1.7em{
v_0 \ar@{->}[r]^{t_0} & u_0
}$, we have a walk $\sig = (v_0, v_1, v_2, \dots )$ supporting $\gm $
in various ways. In consequence of Axiom \eqref{eq:f10.0},
there is a minimal    $k < \om$ such that every node $u_r$ with $r > k $ is passive, in other words
$t_k = t_{k+1} = \cdots$, cf. Definition \ref{def:f10.9}.

 We use  Theorem \ref{thm:f10.10} to study idempotent supports of walks in the grid $\Gm \times \Dl$, $\Gm = \{ u_{0,0},  u_{1,0},  u_{2,0}, \dots \}$,
 $\Dl= \{ u_{0,0},  u_{0,1},  u_{0,2}, \dots \}$,
 restriced to an $(r \times s)$-rectangle of $\Gm \times \Dl$, with the meshes $(1 \leq i \leq r, 1 \leq j \leq j)$
   \begin{equation}\label{eq:f10.5}
    \xymatrix@R=1.9em@C=2.7em{
u_{i, j-1}  \ar@{->}[r]^{z_{0,j}} &  u_{i,j} \\
u_{i-1, j-1}  \ar@{->}[r]^{z_{0,j}} \ar@{->}[u]^{z_{i,0}}  &  \ar@{->}[u]^{z_{i,0}}  u_{i-1,j} .
}
\end{equation}

Concerning the ``first'' mesh  \eqref{eq:f10.5} starting at $u_{0,0}$, we have a commutative diagram as follows
 \begin{equation}\label{eq:f10.6}
    \xymatrix@R=1.2em@C=2.7em{
& u_{1, 0}  \ar@{-}[rr]^{z_{0,1}} & &  u_{1,1} \\
u_{0, 0}  \ar@{-}[rr] \ar@{-}[ru]^{z_{1,0}} & \ar@{-}[r]^{z_{0,1}} &  u_{0,1} \ar@{-}[ru]^{z_{1,0}}& v_{11} \ar@{-}[u]_{t_{1,1}} \\
&&& \ar@{-}[u]_{h_{1,1}} \cdot \\
& v_{1,0} \ar@{-}[uuu]^{t_{1,0}}
\ar@{-}[urr]^{\qquad z_{0,1}  }
& v_{0,1} \ar@{-}[uu]^{t_{0,1}}
\ar@{-}[ur]_{z_{1,0}}& \\
&  \ar@{-}[u]^{h_{1,0}}  \ar@{-}[r] & \ar@{-}[r]^{ z_{0,1}} &  \ar@{-}[uu]\\
v_{0, 0}  \ar@{-}[rr]^{z_{0,1}} \ar@{-}[ru]^{z_{1,0}}
 \ar@{-}[uuuu]^{t_{0,0}}  & &   \ar@{-}[ru]^{z_{1,0}}
\ar@{-}[u]^{ h_{0,1}} \ar@{-}[uu]&
}
\end{equation}
Here we have two supporting gap walks from $v_{0,0}$ to $v_{1,1}$,
$\xymatrix@R=1.2em@C=2.7em{
v_{0,0} \ar@{-}[r]^{z_{1,0} + h_{1,0}} & v_{1,0} \ar@{-}[r]^{z_{0,1} + h_{1,1}} & v_{1,1}}$
and
$\xymatrix@R=1.2em@C=2.7em{
v_{0,0} \ar@{-}[r]^{z_{0,1} + h_{0,1}} & v_{1,0} \ar@{-}[r]^{z_{1,0} + h_{1,0}} & v_{1,1}}$.

More generally, for a mesh  \eqref{eq:f10.5} with arbitrary $i \geq 1$, $j\geq 1$, we obtain a diagram \eqref{eq:f10.7}, as  below, in which the west and rast bridge from $v_{i-1, j-1}$ to $v_{i,j}$ have lengths $z_{i,0} + h_{i,j-1}$ and  $z_{i,0} + h_{i,j}$, respectively.  These lengths are different if and only if $u_{i,j}$ is active. Similarly,  the south and north bridges  have lengths
 $z_{0,j} + h_{i-1,j}$ and $z_{0,j} + h_{i,j}$, respectively, and are different if and only if $u_{i,j}$ is active.

 \begin{equation}\label{eq:f10.7}
    \xymatrix@R=1.2em@C=2.7em{
& u_{i, j-1}  \ar@{-}[rr]^{z_{0,j}} & &  u_{i,j} \\
u_{i-1, j-1}  \ar@{-}[rr] \ar@{-}[ru]^{z_{i,0}} & \ar@{-}[r]^{z_{0,j}} &  u_{i-1,j} \ar@{-}[ru]^{z_{i,0}}& \\
&&&   \\
& %\ar@{-}[uuu]^{t_{1,0}}
& & \\
&  v_{i-1, j}  \ar@{-}[uuuu]^{t_{i,j-1}}  \ar@{-}[r] & \ar@{-}[r]^{ z_{0,j}+h_{i,j}} &  \ar@{-}[uuuu]^{t_{i,j}} v_{i, j} \\
v_{i-1, j-1}  \ar@{-}[rr]^{z_{0,j} + h_{i-1,j}} \ar@{-}[ru]^{z_{i,0}+ h_{i,j-1}}
 \ar@{-}[uuuu]^{t_{i-1,j-1}}  & &   \ar@{-}[ru]_{z_{i,0}+h_{i,j}}
\ar@{-}[uuuu]^{ t_{i-1,j}} \ar@{-}[uu] v_{i-1, j}  &
}
\end{equation}

Returning to Construction~\ref{cont:f10.6}, we reinterpret it with the goal of extending the construction by adding meshes below the supporting walk $\sig$ of $\gm$.

\begin{construction}\label{cont:f10.11}
In diagram \eqref{eq:f10.2}, we rewrite the first rectangle as follows:      \begin{equation*}%\label{eq:e10.2}
   \xymatrixcolsep{6mm}
    \xymatrix@R=1.3em@C=1.5em{
     u_{0} \ar@{->}[rr]^{z_1} &  &  u_{1} \\
     v_{0} \ar@{->}[rr]^{z'_1} \ar@{->}[u]^{t= t_0} &  &  v_{1}  \ar@{->}[u]_{t_1}
     & \qquad \text{ with } z'_1 = z_1 + h_1. \\
     }      \end{equation*}
     Our idea is that the minimal nonzero idempotent $t_0$ is `` reduced'' to a minimal nonzero idempotent $t_1$ by a ``catalyst'' $z_1$ (a term borrowed from chemistry). In exactly the same way, we reduce the $(k-1)$st minimal nonzero idempotent $t_{k-1}$ to a nonzero idempotent $t_k$, and thus arrive at a commutative diagram.
      \begin{equation*}%\label{eq:e10.2}
   \xymatrixcolsep{6mm}
    \xymatrix@R=1.5em@C=1.9em{
     \gm :& u_{0} \ar@{->}[r]^{z_1}   &  u_{1}  \ar@{->}[r]^{z_2} & u_2  \ar@{->}[r]^{z_3} & u_3  \ar@{->}[r] & \cdots \\
     \sig: &v_{0} \ar@{->}[r]^{z'_1} \ar@{->}[u]_{t_0}  &  v_{1}  \ar@{->}[r]^{z'_2} \ar@{->}[u]_{t_1} & v_{2}  \ar@{->}[r]^{z'_3} \ar@{->}[u]_{t_2} & v_3
     \ar@{->}[r] \ar@{->}[u]_{t_3} & \cdots \\
     }      \end{equation*}
If we find a nonzero idempotent gap ending at $v_0$, then, by Axiom~\eqref{eq:f10.1}, we may choose a minimal nonzero idempotent gap
     $ \xymatrix@R=1.5em@C=1.4em{
v'_0 \ar@{->}[r]^{t'_0} & v_0
}$  and repeat the construction. We sketch a diagram with a two-step iteration
   \begin{equation}\label{eq:f10.8}
   \xymatrixcolsep{6mm}
    \xymatrix@R=1.5em@C=1.9em{
     \gm :& u_{0} \ar@{->}[r]^{z_1}   &  u_{1}  \ar@{->}[r]^{z_2} & u_2  \ar@{->}[r]^{z_3} & u_3  \ar@{->}[r] & \cdots \\
     \sig: &v_{0} \ar@{->}[r]^{z'_1} \ar@{->}[u]_{t_0}  &  v_{1}  \ar@{->}[r]^{z'_2} \ar@{->}[u]_{t_1} & v_{2}  \ar@{->}[r]^{z'_3} \ar@{->}[u]_{t_2} & v_3
     \ar@{->}[r] \ar@{->}[u]_{t_3} & \cdots
     \\
     \sig': &v'_{0} \ar@{->}[r]^{z''_1} \ar@{->}[u]_{t'_0}  &  v'_{1}  \ar@{->}[r]^{z''_2} \ar@{->}[u]_{t'_1} & v'_{2}  \ar@{->}[r]^{z''_3} \ar@{->}[u]_{t'_2} & v'_3
     \ar@{->}[r] \ar@{->}[u]_{t'_3} & \cdots
     \\
     \sig'': &v''_{0} \ar@{->}[r]^{z'''_1} \ar@{->}[u]_{t''_0}  &  v''_{1}  \ar@{->}[r]^{z'''_2} \ar@{->}[u]_{t''_1} & v''_{2}  \ar@{->}[r]^{z'''_3} \ar@{->}[u]_{t''_2} & v''_3
     \ar@{->}[r] \ar@{->}[u]_{t''_3} & \cdots \\
     }      \end{equation}
     This diagram is determined  by the walk north
      $ \xymatrix@R=1.5em@C=1.3em{
v''_0 \ar@{->}[r] & v'_0 \ar@{->}[r]  &
v_0 \ar@{->}[r] & u_0
}$
and the  walk east
$ \xymatrix@R=1.5em@C=1.3em{
u_0 \ar@{->}[r] & u_1 \ar@{->}[r]  &
u_2  \ar@{->}[r] & u_3
}$
in a controlled way (cf. Lemma \ref{lem:f10.4}).
\end{construction}

\begin{remark}\label{rem:f10.12}
We have decompositions
$$\begin{array}{llll}
    z'_1 = z_1 + h_1, & z'_2 = z_2 + h_2, & \dots  \\
    z''_1 = z_1 + h_1 + h'_1, & z''_2 = z_2 + h_2 +h'_2,& \dots  \\
    z'''_1 = z_1 + h_1 + h'_1 + h''_1, & z'''_2 = z_2 + h_2 + h'_2 +h''_2,& \dots  \\
  \end{array} $$
We speak of active and passive nodes of $\sig$ and $\sig'$ in the same way as for $\gm$.
\end{remark}

\begin{remark}\label{def:f10.13} $ $
\begin{enumerate}\ealph
  \item Given $u \in U $, let $\mig(u)$ denote the set of minimal nonzero idempotent gaps
  $ \xymatrix@R=1.5em@C=1.4em{
v \ar@{->}[r]^{t} & u
}$ below $u$. In other words, $\mig(u)$ is the flock of maximal fixed points of $(\N, +)$ in $u^\downarrow$.

  \item  Given an infinite blue walk
  $\xymatrix@R=1.9em@C=1.7em{\gm:
u_0 \ar@{->}[r]^{  z_1} & u_1 \ar@{->}[r]^{z_2}  & \cdots
}$, for any $ \xymatrix@R=1.5em@C=1.7em{
v \ar@{->}[r]^{t} & u_0
 }$ in $\mig(u_0)$, there is  a unique idempotent gap walk $\sig(t)$ supporting $\gm$ by $t$, using the catalysts $z_1, z_2, \dots$.

  \item Given a new idempotent gap
  $ \xymatrix@R=1.5em@C=1.7em{
v' \ar@{->}[r]^{t'} & u_0
}$, we say that $\sig(t)$ and $\sig(t')$ are \textbf{confluent under} $\gm$, if $\sig(t)$ and $\sig(t')$ have at least one common node. Since the same catalysts are used to build $\sig(t)$ and $\sig(t')$, it is evident that there is a first number $k>0$ such that
$\sig(t)_{\geq k} = \sig(t')_{\geq k} .$
We say that $\sig(t)$ and $\sig(t')$ are \textbf{confluent at} $u_k$.
(Note that confluence at $u_0$ holds if and only if $t=t'$.)
\end{enumerate}

\end{remark}

We consider the special case in which only idempotent gaps occur.

\begin{prop}\label{prop:f10.14}
Assume that in the infinite blue walk
  $\xymatrix@R=1.9em@C=1.7em{\gm:
u_0 \ar@{->}[r]^{  z_1} & u_1 \ar@{->}[r]^{z_2} & u_2 \ar@{->}[r] & \cdots
}$
all bridges $z_i$ are idempotents, and that
 $ \xymatrix@R=1.5em@C=1.7em{
v_0 \ar@{->}[r]^{t_0} & u_0
 }$ in $\mig(u_0)$ is a minimal nonzero idempotent. Then, in the commuting diagram ~ \eqref{eq:f10.2}, all $h_i$ are idempotents, and all $t_i$ are  minimal nonzero idempotents.
\end{prop}

\begin{proof}
  It is immediate from the diagram that $z_i = 2 z_i$, $t_i = 2 t_i$, and $h_i = 2h_i$.
\end{proof}

More generally, by the same argument, we see that in the diagram \eqref{eq:f10.8}, all arrows are idempotent gaps and all vertical arrows are minimal nonzero idempotents. We call this situation an \textbf{idempotent confluence diagram}. (All catalysts $z_i',z_i''$ are idempotents.)

Given a blue walk
$
\xymatrix@R=1.9em@C=1.7em{
\gm: u_0 \ar@{->}[r]^{  z_1} & u_1 \ar@{->}[r]^{z_2} & \cdots \ar@{->}[r]^{z_r } & u_r
}$
of length $r$ and finitely many nonzero idempotent gaps $ t_{0,1}, \dots, t_{0,s} \in \mig(u_0)$, we can study idempotent confluence diagrams below $\gm$. This seems to suggest a direction for further research into idempotent confluence.
\begin{example}\label{exmp:f10.15} We sketch such an  idempotent confluence diagram for $r =2$ and $s=5$.
$$
\xymatrix@R=0.7em@C=1.7em{
u_0 \ar@{->}[r]^{  z_1} & u_1 \ar@{->}[r]^{z_2} & u_2 \\
\sig(t_{0,1}) \ar@{->}[rd] & & \\
\sig(t_{0,2})\ar@{->}[r] & \sig(t_{1,1}) = \sig(t_{1,2}) \ar@{->}[rd] \\
\sig(t_{0,3}) \ar@{->}[r] & \sig(t_{1,3}) = \sig(t_{1,4}) \ar@{->}[r] & \sig(t_{2,1}) = \cdots  = \sig(t_{{2,5}})\\
\sig(t_{0,4})\ar@{->}[ru] & & \\
\sig(t_{0,5})\ar@{->}[r] & \sig(t_{1,5}) \ar@{->}[ruu]  & \\
}$$

\end{example}

\end{document}